\documentclass[11pt,leqno]{article}
\usepackage{amsfonts}
\usepackage{amsmath, amsthm, amsfonts, amssymb, color}
\usepackage{mathrsfs}
\usepackage{hyperref}
\hypersetup{
hidelinks,
hypertex=true,
colorlinks=true,
linkcolor=red,
citecolor=blue,
urlcolor=blue
}

\renewcommand{\bar}{\overline}
\renewcommand{\hat}{\widehat}
\renewcommand{\tilde}{\widetilde}

\newtheorem{thm}{Theorem}[section]

\newtheorem{lem}[thm]{Lemma}
\newtheorem{prp}[thm]{Proposition}

\theoremstyle{definition}

\definecolor{wco}{rgb}{0.9,0.2,0.6}

\numberwithin{equation}{section} \theoremstyle{remark}

\newcommand{\ua}{\uparrow}
\newcommand{\da}{\downarrow}

\title{{\bf Distribution dependent SDEs with multiplicative fractional noise}
}
\author{
{\bf  Xiliang Fan$^{a)}$, Shao-Qin Zhang$^{b)}$}\\
\footnotesize{$^{a)}$School of Mathematics and Statistics, Anhui Normal University, Wuhu 241002, China}\\
\footnotesize{Email: fanxiliang0515@163.com}\\
\footnotesize{$^{b)}$School of Statistics and Mathematics, Central University of Finance and Economics, Beijing 100081, China}\\
\footnotesize{Email: zhangsq@cufe.edu.cn}
}

\begin{document}
\def\R{\mathbb R}
\def\N{\mathbb N}
\def\E{\mathbb E}
\def\H{\mathcal{H}}
\def\Q{\mathbb Q}
\def\P{\mathbb{P}}
\def\S{\mathbb{S}}
\def\Y{\mathbb{Y}}
\def\W{\mathbb{W}}
\def\D{\mathbb{D}}
\def\G{\mathcal{G}}
\def\A{\mathcal{A}}
\def\F{\mathcal{F}}

\def\cH{\mathcal{H}}
\def\cS{\mathcal{S}}

\def\sA{\mathscr{A}}
 \def\sB{\mathscr {B}}
 \def\sC{\mathscr {C}}
  \def\sD{\mathscr {D}}
 \def\sF{\mathscr{F}}
\def\sG{\mathscr{G}}
\def\sL{\mathscr{L}}
\def\sP{\mathscr{P}}
\def\sM{\mathscr{M}}
\def\eq{\equation}
\def\beg{\begin}
\def\ep{\epsilon}
\def\ve{\varepsilon}
\def\vp{\varphi}
\def\vr{\varrho}
\def\om{\omega}
\def\Om{\Omega}
\def\si{\sigma}
\def\ff{\frac}
\def\sq{\sqrt}
\def\kk{\kappa}
\def\de{\delta}
\def\<{\langle} \def\>{\rangle}
\def\Ga{\Gamma}
\def\ga{\gamma}
\def\na{\nabla}
\def\be{\beta}
\def\al{\alpha}
\def\pp{\partial}
 \def\ti{\tilde}
\def\1{\lesssim}
\def\ra{\rightarrow}
\def\da{\downarrow}
\def\upa{\uparrow}
\def\l{\ell}
\def\8{\infty}
\def\3{\triangle}
 \def\DD{\Delta}
\def\m{{\bf m}}
\def\B{\mathbf B}
\def\e{\text{\rm{e}}}
\def\la{\lambda}
\def\th{\theta}
\def\cW{\mathcal W}

\def\d{\text{\rm{d}}}
   \def\ess{\text{\rm{ess}}}
\def\Ric{\text{\rm{Ric}}} \def \Hess{\text{\rm{Hess}}}
 \def\ua{\underline a}
  \def\Ric{\text{\rm{Ric}}}
\def\cut{\text{\rm{cut}}}      \def\alphaa{\mathbf{r}}     \def\r{r}
\def\gap{\text{\rm{gap}}} \def\prr{\pi_{{\bf m},\varrho}}  \def\r{\mathbf r}

\def\Tilde{\tilde} \def\TILDE{\tilde}\def\II
{\mathbb I}
\def\i{{\rm in}}\def\Sect{{\rm Sect}}

\renewcommand{\bar}{\overline}
\renewcommand{\hat}{\widehat}
\renewcommand{\tilde}{\widetilde}

\allowdisplaybreaks
\maketitle
\begin{abstract}
The well-posedness is investigated for distribution dependent stochastic differential equations driven by fractional Brownian motion with Hurst parameter $H\in (\ff {\sq 5-1} 2,1)$ and distribution dependent multiplicative noise. To this aim, we introduce a H\"older space of probability measure paths which is a complete metric space under a new metric.
Our arguments rely on a mix of contraction mapping principle on the H\"older space and fractional calculus tools.
We also establish the large and moderate deviation principles for this type of equations via the weak convergence criteria in the factional Brownian motion setting, which extend previously known results in the additive setting.
\end{abstract}

\section{Introduction}

Distribution dependent stochastic differential equation (DDSDE), also known as mean-field equation or McKean-Vlasov equation,
was first introduced by McKean \cite{McKean66} to model plasma dynamics.
DDSDE of the form
\begin{align}\label{In1}
\d X_t=b(X_t,\sL_{X_t})\d t+\sigma(X_t,\sL_{X_t})\d W_t, \ \ X_0=\xi,
\end{align}
provides the description of evolution $(\sL_{X_t})_{t\geq0}$ in the space of probability measures
with the help of its intrinsic link with nonlinear Fokker-Planck equations,
where $\sL_{X_t}$ stands for the law of $X_t$.
Another important feature of DDSDEs is their relation to the empirical behaviour of large systems of interacting dynamics, i.e.,
the mean-field limit of the following systems of particles subject to a mean-field interaction
\beg{align}\label{In2}
\d X^{i,N}_t=b(X^{i,N}_t,\sL^N_{X^N_t})\d t+\sigma(X^{i,N}_t,\sL^N_{X^N_t})\d W_t,\ \ \sL^N_{X^N_t}=\ff 1 N \sum\limits_{i=1}^N\delta_{X_t^{i,N}}
\end{align}
to \eqref{In1}, as the number $N$ of particles tends to infinity.
For these reasons, DDSDEs have found applications in numerous fields, including statistical physics, mean-field games and mathematical finance etc, see for instance \cite{BT97,CD13,JW17,LL07} and the references therein.
Classical results on the well-posedness of \eqref{In1} and related mean-field limit may date back to Sznitman \cite{Sznitman91}.
The subsequent works mainly put emphasis on the investigation of these two topics under more and more general settings,
see \cite{DV95,Me96,BDF12,JW18,RZ21} for a tiny sample.
Recently, the field of DDSDEs has witnessed a flourishing growth,
one can refer to \cite{BLPR17,CM18} for value functions and associated PDEs,
\cite{BRW,HW,RW,Song,Wang18} for  Bismut formula for the Lions derivative, Harnack type inequality, gradient estimate and exponential ergodicity,
and many other aspects.

Classical stochastic calculus makes sense of equation \eqref{In1} in a probability space $(\Omega,\sF,\P)$, only when the driving noise $W$ is a semimartingale under $\P$ and for some filtration or a Markov process.
However, this setting seems to be too restrictive in many cases, especially when $W$ exhibits long range dependence.
In addition, it was pointed out in \cite{CDFM20} which studied equation \eqref{In1} with $\sigma$ being the identity matrix,
that under the Lipschitz condition on $b$, the mean-field limit of equation \eqref{In2} to equation \eqref{In1} may hold for a larger class of driving processes $W$ regardless of being Markov or semimartingale.
These prompted several authors to address equation \eqref{In1} with driving process being fractional Brownian motion (fBM) in the additive situation:
\begin{align}\label{In3}
\d X_t=b(X_t,\sL_{X_t})\d t+\sigma(\sL_{X_t})\d B^H_t ,\ \ X_0=\xi.
\end{align}
Indeed, the fBM is commonly regarded as the simplest stochastic process modeling time-correlated noise.
A $d$-dimensional fBM $B^H=(B^{H,1},\cdots,B^{H,d})$ with Hurst parameter $H\in(0,1)$ is a centered Gaussian process,
whose covariance structure is defined by
\beg{align}\label{In7}
 R_H(t,s)=\frac{1}{2}\left(t^{2H}+s^{2H}-|t-s|^{2H}\right),\ \  t,s\in[0,T].
\end{align}
This means that $B^H$ is $(H-\ve)$-order H\"{o}lder continuous a.s. for any $\ve\in(0,H)$ and is an $H$-self similar process with stationary increments,
which convert fBM into a natural generalization of Wiener process ($H=1/2)$.
However, the increments of $B^H$ are correlated with a power law correlation decay that leads fBM into a non-Markov process,
which is the dominant feature of equation \eqref{In3} and beyond the realm of It\^{o} calculus.
Huang, Suo, Yuan and the first author \cite{FHSY} proved the well-posedness of equation \eqref{In3} in $\cS^p([0,T])$
(the space of  $\R^d$-valued, continuous $(\sF_t)_{t\in[0,T]}$-adapted processes $\psi$ on $[0,T]$ satisfying
$\E(\sup_{t\in[0,T]}|\psi_t|^p)<+\infty$) under the Lipschitz condition,
in which  the integral  $\int_0^t\si(\sL_{X_s})\d B_s^H$ is interpreted in the Wiener sense because of the determinacy of $\si(\sL_{X_\cdot^\ep})$.
The method of proof relies on the Picard iteration and the Hardy-Littlewood inequality associated with the fractional Riemann-Liouville integral operator.
Besides, we also establish the Bismut formulas for the Lions derivative to equation \eqref{In3} in both non-degenerate and degenerate cases via Malliavin calculus.
Galeati, Harang and Mayorcas \cite{GHM2} showed the well-posedness result to equation \eqref{In3} with $\sigma=\mathrm{I}_{d\times d}$ (the identity matrix on $\R^d\otimes\R^d$)
under the irregular, possibly distributional drift by using some stability estimates.
But, to the best of our knowledge, a more challenging problem that whether equation \eqref{In3} is well-posed in the multiplicative setting
is still open.

In view of these mentioned works, our first goal is to answer the following question:\\

Q{\footnotesize{UESTION}} 1. Is equation \eqref{In3} with multiplicative fractional noise well-posed or ill-posed?
If well-posed, what are the conditions imposed on the coefficients $b$ and $\sigma$? \\

More precisely, we are concerned with the following DDSDE with multiplicative fractional noise:
\begin{align}\label{In4}
\d X_t=b(X_t,\sL_{X_t})\d t+\si(X_t,\sL_{X_t})\d B_t^H,
\end{align}
where $b:\R^d\times\sP_\th(\R^d)\rightarrow\R^d, \si:\R^d\times\sP_\th(\R^d)\rightarrow\R^d\otimes\R^d$ and $H\in(1/2,1)$,
and where $\sP_\th(\R^d)$ is the set of probability measures on $\R^d$ with finite $\th$-th moment.
As shown in our previous work \cite[Remark 3.2(ii)]{FHSY}, even though $\sigma=\sigma(x)$ independent of the measure, we can't be allowed to construct a Cauchy sequence in $\cS^p([0,T])$.
Indeed, for the Picard iteration sequence $(X^n)_{n\geq1}$ corresponding to \eqref{In4},
by a cumbersome calculation we derive that there exist two positive constants $C_{p,T,H}$ and $\tilde{C}_{p,T,H}$ such that
\beg{align}\label{In5}
\sup\limits_{s\in[0,t]}|X^n_s-X^{n-1}_s|^p
\leq& C_{p,T,H}\exp\left\{\tilde{C}_{p,T,H}\|B^H\|_{T,\be}^{\ff 1 \be}\left(1+\|X^{n-1}\|_{T,\be}^{\ff 1 \be}\right)\right\}\cr
&\times\int_0^t\E\bigg(\sup\limits_{u\in[0,r]}|X^{n-1}_u-X^{n-2}_u|^p\bigg)\d r,\ \ t\in[0,T]
\end{align}
with $\tilde{C}_{p,T,H}\ra 0$ as $T\ra0$ and $1-H<\be<H$ (see \eqref{In8} below for the definition of $\|\cdot\|_{T,\be}$).
In view of our previous work \cite[Theorem 3.1 and relation (3.25)]{FZ},
the appearance of the term $\|X^{n-1}\|_{T,\be}^{1/\be}$ in \eqref{In5} leads to a failure of establishing a Cauchy sequence $(X^n)_{n\geq1}$ in $\cS^p([0,T])$. In this paper, we adopt the fixed point argument, which dates back to \cite{Sznitman91}. For the iteration method, $\int_0^t\si(X^{n}_s,\sL_{X^n_s})\d B^H_s$ makes sense, and $\int_0^t\si(X_s,\sL_{X_s})\d B_s^H$ can be established by taking limits if the iteration sequence converges. 
While, to use the contraction mapping principle, we first have to choose a suitable metric space on  probability measures paths to make sense and estimate the (fractional) stochastic integral $\int_0^t\si(X_s,\sL_{X_s})\d B_s^H$. To this end, we introduce a new  metric
\beg{equation*}
\W_{2,T,\be}(\mu,\nu)=\W_{2,T}(\mu,\nu)+\sup_{0\leq s_1<s_2\leq T}\ff {W_{2,s_1,s_2}^c(\mu,\nu)} {(s_2-s_1)^\be},~\mu,\nu\in\sP_2(W_T^d),
\end{equation*}
where $\sP_2(W_T^d)$ is the collection of all probability measures $\mu$ on $W_T^d:=C([0,T];\R^d)$ such that
\[\int_{W_T^d}\left(\sup_{0\leq t\leq T}|\ga(t)|^2\right)\mu(\d \ga)<+\infty,\]
and $\W_{2,T}, W_{2,s_1,s_2}^c$ are specified in later sections (see \eqref{AddHO1} and \eqref{AddHO2} below).  The metric $\W_{2,T,\be}$ on probability measures paths is the equal of the H\"older norm on the function space.
We can prove that the H\"older space $\sP_{2,\be}(W_T^d)$ which is a subspace of $\sP_{2}(W_T^d)$  (see \eqref{AddHO3} below)  is a complete metric space under the metric $\W_{2,T,\be}$.
This allows the well-definedness  of $\int_s^t\si(X_r,\sL_{X_r})\d B_r^H$ and  gives fractional calculus for the distribution dependent part (see Lemma \ref{EsNoi}) to prove the existence and uniqueness of solutions to equation \eqref{In4}. This is our main result, see Theorem \ref{ExU}. We believe that  results concerned the metric $\W_{2,T,\be}$ might be also of independent interest, and they are presented in Subsection 3.2.
We also develop a fractional calculus technique, that is borrowing a slightly more path regularity from fractional Brownian motion to estimate the H\"older semi-norm of the solution $\|X\|_{s,t,\be}$.
With the solution in hand, what may be interesting is that whether equation \eqref{In4} could be well-behaved.
So, the second part of the paper is motivated by the following question:\\

Q{\footnotesize{UESTION}} 2.  What will small-noise asymptotic behaviors of equation \eqref{In4} look like? \\

To answer this question, we are concerned with equation \eqref{In4} with small noise, namely, for any $\ep>0$,
\begin{align}\label{In6}
\d X_t^\epsilon =b(X_t^\ep,\mathscr{L}_{X_t^\ep})\d t+\ep^H\sigma(X_t^\ep,\mathscr{L}_{X_t^\ep})\d B_t^H,\ \ X_0^\ep =x.
\end{align}
Our main purpose is to investigate large deviation principle (LDP) and moderate deviation principle (MDP) of \eqref{In6} as $\ep$ goes to $0$.
More precisely, let
\begin{align*}
Y^\ep=\ff {X^\ep-X^0}{\ep^H\zeta(\ep)}
\end{align*}
with $X^0$ being the limit of $X^\ep$ in some sense, we prove that if $\zeta(\ep)=1/\ep^H$,
$X^\ep$ satisfies the LDP with speed $\ep^{2H}$ (see Theorem \ref{Th(LDP)}),
while if $\zeta(\ep)\ra+\infty$ and $\ep^H\zeta(\ep)\ra0$ as $\ep\ra0$,
$Y^\ep$ satisfies the LDP with speed $\zeta^{-2}(\ep)$, i.e., $X^\ep$ satisfies the MDP (see Theorem \ref{Th(mdp)}).
Recall that large and moderate deviation principles are mainly to calculate the probability of a rare event.
In the case of stochastic processes, the idea lies in finding a deterministic path around which the diffusion is concentrated with high probability,
which implies that the stochastic motion can be regarded as a small perturbation of this deterministic path.
A conventional way to handle large and moderate deviation principle problems is the time discretization method and the contraction principle originated by  \cite{FW84}, which could be very complicated for infinite dimensional nonlinear stochastic dynamical systems.
Another general approach for investigating these problems is the well-known weak convergence method introduced in \cite{BD00,BDM08},
which depends on a variational representation for positive functionals of Brownian motion or Poisson random measure.
After that, many authors have applied this approach to various stochastic dynamical systems,
among which let us just mention, for instance, the woks of  \cite{BGJ17,BDG16,BDM11,BS20,DWZZ20,DXZZ17,DST19,HLL21,LSZZ22,MSZ21,SY21,WZZ15,ZZ15} and the references therein.

Based on the weak convergence method, we obtain the large and moderate deviation principles for equation \eqref{In6},
and we use the variational framework and the weak convergence criteria in the factional Brownian motion setting that has been constructed in our recent work \cite{FYY}.
Let us stress that in comparison to \cite{FYY} where the large and moderate deviation principles for equation \eqref{In6} with additive fractional noise (namely, $\si=\si(\mathscr{L}_{X_t^\ep})$) are established,
we now work on equation \eqref{In6} with multiplicative fractional noise.
Therefore, we need to obtain some apriori estimates of solutions whose calculations are much more complicated than those in the additional situation.

The rest of the paper is structured as follows.
Section 2 is devoted to recalling some useful facts on fractional calculus, fractional Brownian motion and the $L$-derivative.
In Section 3, we focus on proving the existence and uniqueness of a solution to distribution dependent SDE with multiplicative fractional noise.
The main technical result is to show that the H\"older space of probability measure paths is a complete metric space under the introduced metric $\W_{2,T,\be}$.
In Section 4, we establish the large and moderate deviation principles for this type equations.
Section 5 will be devoted to the proofs of some auxiliary results.

$\mathbf{Notation.}$ The following notations are used in the sequel.

$\bullet$  $C([a,,b];\R^d)$ denotes the space of all $\R^d$-valued continuous functions $\varphi:[a,b]\ra\R^d$ with the norm
$\|\varphi\|_{a,b,\infty}=\sup_{a\leq s\leq b}|\varphi(s)|.$
For any $\alpha\in(0,1), C^\alpha([a,b];\R^d)$ is the set of all $\R^d$-valued $\alpha$-H\"{o}lder continuous functions.
If $\varphi\in C^\alpha([a,b];\R^d)$, we will make use of the notation
\beg{align}\label{In8}
\|\varphi\|_{a,b,\alpha}=\sup\limits_{a\leq s<t\leq b}\frac{|\varphi(t)-\varphi(s)|}{|t-s|^\alpha}.
\end{align}
In the case of the interval $[0,T]$, we abbreviate $\|\varphi\|_{0,T,\infty}$ and $\|\varphi\|_{0,T,\alpha}$ as $\|\varphi\|_{T,\infty}$ and $\|\varphi\|_{T,\alpha}$, respectively.

$\bullet$ For a bounded function $f$, we denote $\|f\|_\infty$ by its bounded constant,
and for a random variable $X\in L^p(\Omega,\sF,\P)$ with $p\in[1,+\infty)$, we set $\|X\|_{L^p(\Omega)}=(\E\|X\|^p)^{\ff 1 p}$.

$\bullet$  For $p\in[1,+\infty)$, $\sP_p(\mathbb{B})$ represents the set of $p$-integrable probability measures over a separable Banach space $\mathbb{B}$,
and define for $\mu\in\sP_p(\mathbb{B})$,
\beg{align*}
\|\mu\|_p:=\left(\int_{\mathbb{B}}\|x\|_{\mathbb{B}}^p\mu(\d x)\right)^{\ff 1 p}.
\end{align*}

$\bullet$ We let  $C, C_{H,T}, C_{H,T,M},$ etc., denote generic constants, whose values may vary at each appearance.

\section{Preliminaries}

This section is devoted to giving some of the basic elements of fractional calculus, Wiener space associated to fBM and the Lions derivative.

\subsection{Fractional integral and derivative}

Let $a,b\in\R$ with $a<b$.
For $f\in L^1([a,b],\R)$ and $\alpha>0$, the left-sided (respectively right-sided) fractional Riemann-Liouville integral of $f$ of order $\alpha$
on $[a,b]$ is defined as
\beg{align}\label{FrIn}
&I_{a+}^\alpha f(x)=\frac{1}{\Gamma(\alpha)}\int_a^x\frac{f(y)}{(x-y)^{1-\alpha}}\d y\\
&\qquad\left(\mbox{respectively}\ \ I_{b-}^\alpha f(x)=\frac{(-1)^{-\alpha}}{\Gamma(\alpha)}\int_x^b\frac{f(y)}{(y-x)^{1-\alpha}}\d y\right),\nonumber
\end{align}
where $x\in(a,b)$ a.e., $(-1)^{-\alpha}=\e^{-i\alpha\pi}$ and $\Gamma$ denotes the Gamma function.
In particular, if $\alpha=n\in\N$, they are consistent with the usual $n$-order iterated integrals.

Fractional differentiation can be given as an inverse operation.
Let $\alpha\in(0,1)$ and $p\geq1$.
If $f\in I_{a+}^\alpha(L^p([a,b],\R))$ (respectively $I_{b-}^\alpha(L^p([a,b],\R)))$, then the function $g$ satisfying $I_{a+}^\alpha g=f$ (respectively $I_{b-}^\alpha g=f$) is unique in $L^p([a,b],\R)$ and it coincides with the left-sided (respectively right-sided) Riemann-Liouville derivative
of $f$ of order $\alpha$ shown by
\beg{align*}
&D_{a+}^\alpha f(x)=\frac{1}{\Gamma(1-\alpha)}\frac{\d}{\d x}\int_a^x\frac{f(y)}{(x-y)^\alpha}\d y\\
&\qquad\left(\mbox{respectively}\ D_{b-}^\alpha f(x)=\frac{(-1)^{1+\alpha}}{\Gamma(1-\alpha)}\frac{\d}{\d x}\int_x^b\frac{f(y)}{(y-x)^\alpha}\d y\right).
\end{align*}
The corresponding Weyl representation is of the form
\beg{align}\label{FrDe}
&D_{a+}^\alpha f(x)=\frac{1}{\Gamma(1-\alpha)}\left(\frac{f(x)}{(x-a)^\alpha}+\alpha\int_a^x\frac{f(x)-f(y)}{(x-y)^{\alpha+1}}\d y\right)\\
&\qquad\left(\mbox{respectively}\ \ D_{b-}^\alpha f(x)=\frac{(-1)^\alpha}{\Gamma(1-\alpha)}\left(\frac{f(x)}{(b-x)^\alpha}+\alpha\int_x^b\frac{f(x)-f(y)}{(y-x)^{\alpha+1}}\d y\right)\right),\nonumber
\end{align}
where the convergence of the integrals at the singularity $y=x$ holds pointwise for almost all $x$ if $p=1$ and in the $L^p$ sense if $p>1$.
For more details, we refer the reader to \cite{SKM93}.

Suppose that $f\in C^\lambda([a,b];\R^d)$ and $g\in C^\mu([a,b];\R^d)$ with $\lambda+\mu>1$.
By \cite{Young36a}, the Riemann-Stieltjes integral $\int_a^bf\d g$ exists.
In \cite{Zahle}, Z\"{a}hle provides an explicit expression for the integral $\int_a^bf\d g$ in terms of fractional derivatives.
That is, let $\lambda>\alpha$ and $\mu>1-\alpha$ with $\alpha\in(0,1)$.
Then the Riemann-Stieltjes integral can be expressed as
\beg{align}\label{Fibpf}
\int_a^bf\d g=(-1)^\alpha\int_a^bD_{a+}^\alpha f(t)D_{b-}^{1-\alpha}g_{b-}(t)\d t,
\end{align}
where $g_{b-}(t)=g(t)-g(b)$.
The relation \eqref{Fibpf} can be regarded as fractional integration by parts formula.

\subsection{Wiener space associated to fBM}

For some fixed $H\in(1/2,1)$, we consider $(\Omega,\sF,\P)$ the canonical probability space associated with fBM with Hurst parameter $H$.
That is, $\Omega=C_0([0,T],\R^d)$ is the Banach space of continuous functions vanishing at $0$ equipped with the supremum norm,
$\sF$ is the Borel $\si$-algebra and $\P$ is the unique probability measure on $\Omega$ such that the canonical process $\{B^H_t; t\in[0,T]\}$ is a $d$-dimensional fBM with Hurst parameter $H$.
We suppose that there is a sufficiently rich sub-$\si$-algebra $\sF_0\subset\sF$ independent of $B^H$ such that for any $\mu\in\sP_2(\R^d)$ there exists a random variable $X\in L^2(\Omega\ra\R^d,\sF_0,\P)$ with distribution $\mu$.
Let $\{\sF_t\}_{t\in[0,T]}$ be the filtration generated by $B^H$, completed and augmented by $\sF_0$.

Let $\mathscr{E}$ be the space of $\R^d$-valued step functions on $[0,T]$, and $\mathcal {H}$ the Hilbert space defined as the closure of
$\mathscr{E}$ with respect to the scalar product
\beg{align*}
\left\langle (\mathbb I_{[0,t_1]},\cdot\cdot\cdot,\mathbb I_{[0,t_d]}),(\mathbb I_{[0,s_1]},\cdot\cdot\cdot,\mathbb I_{[0,s_d]})\right\rangle_\cH=\sum\limits_{i=1}^dR_H(t_i,s_i).
\end{align*}
Recall here that $R_H(\cdot,\cdot)$ is given in \eqref{In7}.
Notice that by \cite{Decreusefond&Ustunel98a,ND},  $R_H(t,s)$ has the integral representation of the form
\begin{align*}
 R_H(t,s)=\int_0^{t\wedge s}K_H(t,r)K_H(s,r)\d r,
\end{align*}
where $K_H(t,s)$ is the square integrable kernel defined by
\begin{align*}
K_H(t,s)=C_Hs^{\ff 1 2-H}\int_s^t(r-s)^{H-\ff 3 2}r^{H-\ff 1 2}\d r, \ \ t>s
\end{align*}
with $C_H=\sqrt{\ff {H(2H-1)}{\mathcal{B}(2-2H,H-1/2)}}$ and $\mathcal{B}$ standing for the Beta function.
When $t\leq s$, we let $K_H(t,s)=0$.
The mapping $(\mathbb I_{[0,t_1]},\cdot\cdot\cdot,\mathbb I_{[0,t_d]})\mapsto\sum_{i=1}^dB_{t_i}^{H,i}$ can be extended to an isometry between $\H$ (also called the reproducing kernel Hilbert space) and the Gaussian space $\mathcal {H}_1$ associated to $B^H$. We denote this isometry by $\phi\mapsto B^H(\phi)$.

Now, let $(e_1,\cdots,e_d)$ designate the canonical basis of $\R^d$, one constructs the linear operator $K_H^*:\mathscr{E}\rightarrow L^2([0,T],\R^d)$ such that $K_H^*(\mathrm{I}_{[0,t]}e_i)=K_H(t,\cdot)e_i$.
According to \cite{Alos&Mazet&Nualart01a}, there holds the relation $\langle K_H^*\psi,K_H^*\phi\rangle_{L^2([0,T],\R^d)}=\langle\psi,\phi\rangle_\H$ for all $\psi,\phi\in\mathscr{E}$,
and then with the help of the bounded linear transform theorem, $K_H^*$ can be extended to an isometry between $\mathcal{H}$ and $L^2([0,T],\R^d)$.
Consequently, by \cite{Alos&Mazet&Nualart01a} again,
there exists a $d$-dimensional Wiener process $W$ defined on $(\Omega,\sF,\P)$ such that $B^H$ has the following Volterra-type representation
\beg{align*}
B_t^H=\int_0^tK_H(t,s)\d W_s, \ \ t\in[0,T].
\end{align*}
Furthermore, let us stress that $K_H^*$ has the following integral representations: for any $\psi,\phi\in\H$,
\begin{align}\label{IRKH}
(K_H^*\psi)(t)=\int_t^T\psi(s)\ff {\partial K_H(s,t)}{\partial s}\d s.
\end{align}
and
\begin{align}\label{Isom}
\langle K_H^*\psi,K_H^*\phi\rangle_{L^2([0,T],\R^d)}=\langle\psi,\phi\rangle_\H=H(2H-1)\int_0^T\int_0^T|t-s|^{2H-2}\langle\psi(s),\phi(t)\rangle_{\R^d}\d s\d t.
\end{align}
As a consequence, for any $\psi\in L^2([0,T],\R^d)$, we have
\begin{align}\label{EsH}
\|\psi\|^2_\H\leq2HT^{2H-1}\|\psi\|^2_{L^2}.
\end{align}
In addition, it can be shown that $L^{1/H}([0,T],\R^d)\subset\H$.\\

Next, we define the operator $K_H: L^2([0,T],\mathbb{R}^d)\rightarrow I_{0+}^{H+1/2}(L^2([0,T],\mathbb{R}^d))$ by
\begin{align*}
 (K_H f)(t)=\int_0^tK_H(t,s)f(s)\d s.
\end{align*}
Let us point out that the set $I_{0+}^{H+1/2}(L^2([0,T],\mathbb{R}^d))$ is the fractional version of the Cameron-Martin space.
Moreover, we denote by $R_H=K_H\circ K_H^*:\mathcal{H}\rightarrow I_{0+}^{H+1/2}(L^2([0,T],\mathbb{R}^d))$ the operator
\begin{align*}
(R_H\psi)(t)=\int_0^tK_H(t,s)(K_H^*\psi)(s)\d s.
\end{align*}
Since  $I_{0+}^{H+1/2}(L^2([0,T],\mathbb{R}^d))\subset C^{H}([0,T],\R^d)$ thanks to \cite[Theorem 3.6]{SKM93},
we know that for every $\psi\in\H$, $R_H\psi$ is H\"{o}lder continuous of order $H$, i.e., $R_H\psi\in C^{H}([0,T],\R^d)$.
On the other hand, in view of the Fubini theorem and the fact that $\ff {\partial K_H(s,r)}{\partial s}=C_H(\ff s r)^{H-\ff 1 2}(s-r)^{H-\ff 3 2}$,
it is readily checked that for each $\psi\in\H, R_H\psi$ is absolutely continuous and
\begin{align}\label{RFRH}
(R_H\psi)(t)=\int_0^t\left(\int_0^s\ff {\partial K_H}{\partial s}(s,r)(K_H^*\psi)(r)\d r\right)\d s.
\end{align}

\subsection{The Lions derivative}

For any $\th\in[1,+\infty)$, define the $L^\th$-Wasserstein distance on $\sP_\th(\R^d)$ as follows
\begin{align*}
W_\th(\mu,\nu):=\inf_{\pi\in\sC(\mu,\nu)}\left(\int_{\R^d\times\R^d}|x-y|^\th\pi(\d x, \d y)\right)^\ff 1 \th,\ \ \mu,\nu\in\sP_\th(\R^d).
\end{align*}
Here $\sC(\mu,\nu)$ is the set of all probability measures on $\R^d\times\R^d$ with marginals $\mu$ and $\nu$.
It is well-known that $(\sP_\th(\R^d),W_\th)$ is a Polish space, usually referred to as the $\th$-Wasserstein space on $\R^d$.
Throughout this paper, we use $|\cdot|$ and $\<\cdot,\cdot\>$  for the Euclidean norm and inner product on $\R^d$, respectively,
and for a matrix, we denote by $\|\cdot\|$ the operator norm.
$\|\cdot\|_{L^2_\mu}$ denotes for the norm of the space $ L^2(\R^d\ra\R^d,\mu)$ and for a random variable $X$, $\sL_X$ denotes its distribution.

\beg{defn}
Let $f:\sP_2(\R^d)\ra\R$ and $g:\R^d\times\sP_2(\R^d)\ra\R$.
\begin{enumerate}
\item[(1)]  $f$ is called $L$-differentiable at $\mu\in\sP_2(\R^d)$, if the functional
\begin{align*}
L^2(\R^d\ra\R^d,\mu)\ni\phi\mapsto f(\mu\circ(\mathrm{Id}+\phi)^{-1}))
\end{align*}
is Fr\'{e}chet differentiable at $0\in L^2(\R^d\ra\R^d,\mu)$. That is, there exists a unique $\gamma\in L^2(\R^d\ra\R^d,\mu)$ such that
\begin{align*}
\lim_{\|\phi\|_{L^2_\mu}\ra0}\ff{f(\mu\circ(\mathrm{Id}+\phi)^{-1})-f(\mu)-\mu(\<\gamma,\phi\>)}{\|\phi\|_{L^2_\mu}}=0.
\end{align*}
In this case, $\gamma$ is called the $L$-derivative of $f$ at $\mu$ and denoted by $D^Lf(\mu)$.

\item[(2)] $f$ is called $L$-differentiable on $\sP_2(\R^d)$, if the $L$-derivative $D^Lf(\mu)$ exists for all $\mu\in\sP_2(\R^d)$.
Furthermore, if for every $\mu\in\sP_2(\R^d)$ there exists a $\mu$-version $D^Lf(\mu)(\cdot)$ such that $D^Lf(\mu)(x)$ is jointly continuous in $(\mu,x)\in\sP_2(\R^d)\times\R^d$, we denote $f\in C^{(1,0)}(\sP_2(\R^d))$.

\item[(3)] $g$ is called differentiable on $\R^d\times\sP_2(\R^d)$, if for any $(x,\mu)\in\R^d\times\sP_2(\R^d)$,
$g(\cdot,\mu)$ is differentiable and $g(x, \cdot)$ is $L$-differentiable.
If $\nabla g(\cdot,\mu)(x)$ and $D^Lg(x,\cdot)(\mu)(y)$ are jointly continuous in $(x,y,\mu)\in\R^d\times\R^d\times\sP_2(\R^d)$,
we denote $g\in C^{1,(1,0)}(\R^d\times\sP_2(\R^d))$.
If moreover the derivatives
$$\na^2g(\cdot,\mu)(x),\  \na(D^Lg(\cdot,\cdot)(\mu)(y))(x),\ \na(D^Lg(x,\cdot)(\mu)(\cdot))(y),\ D^L(D^Lg(x,\cdot)(\cdot)(y))(\mu)(z),$$
exist and are jointly continuous in the corresponding arguments $(x,\mu), (x,\mu,y)$ or $(x,\mu,y,z)$,
we denote $g\in C^{2,(2,0)}(\R^d\times\sP_2(\R^d))$.
If $g\in C^{2,(2,0)}(\R^d\times\sP_2(\R^d))$ and with all these derivatives is bounded on $\R^d\times\sP_2(\R^d)$,
we denote $g\in C_b^{2,(2,0)}(\R^d\times\sP_2(\R^d))$.
\end{enumerate}
\end{defn}
In order to ease notations, we denote $D^{L,2}g=D^L(D^Lg)$, and for a vector-valued function $f=(f_i)$ or a matrix-valued function $f=(f_{ij})$ with $L$-differentiable components, we simply write
\begin{align*}
D^Lf(\mu)=(D^Lf_i(\mu))  \ \  \mathrm{or}\ \ D^Lf(\mu)=(D^Lf_{ij}(\mu)).
\end{align*}
Let us finish this part by giving a useful formula for the $L$-derivative that are needed later on, which is due to \cite[Theorem 6.5]{Cardaliaguet13} and \cite[Proposition 3.1]{RW}.
\beg{lem}\label{FoLD}
Let $(\Omega,\sF,\P)$ be an atomless probability space and $X,Y\in L^2(\Omega\ra\R^d,\P)$.
If $f\in C^{(1,0)}(\sP_2(\R^d))$, then
\beg{align*}
\lim_{\ve\da0}\ff {f(\sL_{X+\ve Y})-f(\sL_X)} \ve=\E\<D^Lf(\sL_X)(X),Y\>.
\end{align*}
\end{lem}

\section{Well-posedness of DDSDE with multiplicative fractional noise}

The main objective of this section concerns the problem of well-posedness of DDSDEs with multiplicative fractional noises.
We first formulate our main result.
In the second part of this section, we will introduce a new H\"{o}lder space $\sP_{2,\be}(W_T^d)$ of probability measure paths,
and then prove that $\sP_{2,\be}(W_T^d)$ is a complete metric space under some metric $\W_{2,T,\be}$,
which plays a key role in the proof of the main result that is addressed in the third part.

\subsection{Main result}

Let us recall that the focus here is a DDSDE with multiplicative fractional noise of the form:
\begin{align}\label{EquM}
\d X_t=b(X_t,\sL_{X_t})\d t+\si(X_t,\sL_{X_t})\d B_t^H,
\end{align}
where the coefficients $b:\R^d\times\sP_2(\R^d)\rightarrow\R^d, \si:\R^d\times\sP_2(\R^d)\rightarrow\R^d\otimes\R^d$,
and $B^H$ is a $d$-dimensional fractional Brownian motion with $H\in(1/2,1)$.

\beg{defn}
We say that $\{X_t\}_{t\in[0,T]}$ is a solution of equation \eqref{EquM}, if $\{X_t\}_{t\in[0,T]}$ is an adapted process such that for any $\be\in(0,H)$, $\P$-a.s. $X\in C^{\be}([0,T];\R^d)$,
\beg{equation}\label{in-mom}
\E\left(\|X\|_{T,\infty}^2+\|X\|_{T,\be}^2\right)<+\infty,
\end{equation}
and $X_t$ satisfies equation \eqref{EquM}.
\end{defn}

Before going on, we first discuss the well-posedness of the stochastic integral which is defined by the fractional integration by parts formula \eqref{Fibpf}:
for any $\be\in (\ff 1 2,H)$ and $\al\in(1-\be,\be)$,
\[\int_s^t\si(X_r,\sL_{X_r})\d B^H_r=(-1)^{\al}\int_s^tD_{s+}^\al(\si(X_\cdot,\sL_{X_\cdot}))(r)D_{t-}^{1-\al}B_{t-}^H(r)\d r.\]
If $\sigma$ is imposed on some regularity condition (for instance, $\si_{ij}\in  C_b^{2,(2,0)}(\R^d\times\sP_2(\R^d)), 1\leq i,j\leq d$),
due to the definition of the fractional derivative, to make sense the stochastic integral term,  the path $t\mapsto\sL_{X_t}$ should be  H\"older continuous under some metric of probability measures.  Noticing that
\beg{equation}\label{in-WEH}
\sup_{0\leq s<t\leq T}\ff {W_2(\sL_{X_t},\sL_{X_s})} {(t-s)^{\be}}\leq \sup_{0\leq s<t\leq T}\ff {\sqrt{\E|X_t-X_s|^2}} {(t-s)^{\be}}\leq \sq{\E\|X\|_{T,\be}^2},
\end{equation}
we find that for a solution of equation \eqref{EquM}, $\int_s^t\si(X_r,\sL_{X_r})\d B_r^H$ is well-defined. When dealing with the uniqueness of solutions to equation \eqref{EquM}, we come across the following integral for two solutions $X_t,Y_t$:
\begin{equation}\label{int-si-si}
\beg{split}
&\int_s^t(\si(X_r,\sL_{X_r})-\si(Y_r,\sL_{Y_r}))\d B^H_r\\
&=(-1)^{\al}\int_s^tD_{s+}^\al(\si(X_\cdot,\sL_{X_\cdot})-\si(Y_\cdot,\sL_{Y_\cdot}))(r)D_{t-}^{1-\al}B_{t-}^H(r)\d r\\
&=\frac{(-1)^{\al}}{\Gamma(1-\alpha)}\int_s^t
\bigg[\bigg(\frac{\sigma(X_r ,\sL_{X_r})-\sigma(Y_r,\sL_{Y_r})}{(r - s)^\alpha}\cr
&\quad\quad+\alpha\int_s^r\frac{\sigma(X_r,\sL_{X_r})-\sigma(Y_r,\sL_{Y_r})-\left(\sigma(X_u,\sL_{X_u})-\sigma(Y_u,\sL_{Y_u})\right)}
{(r - u)^{1+\alpha}}\d u\bigg)\bigg]D_{t-}^{1-\al}B_{t-}^H(r)\d r.
\end{split}
\end{equation}
For the random variable part of the right hand side in the last equality, it involves the random variable $X_r-Y_r-(X_u-Y_u)$, which can be controlled by the H\"older norm on $C^\beta([0,T];\R ^d)$. While for the distribution part, we need a metric to measure the term ``$\sL_{X_r}-\sL_{Y_r}-(\sL_{X_u}-\sL_{Y_u})$". As \eqref{in-WEH}, $\E\|X-Y\|_{T,\be}^2$ can be used to control the distribution part, which may be related to a Wasserstein distance on $C^\beta([0,T];\R^d)$:
\[\inf_{\pi\in\mathscr{C}(\sL_{X},\sL_Y)}\int_{C^{\be}([0,T];\R^d)\times C^{\be}([0,T];\R^d)}\|\ga_1-\ga_2\|_{\be,T}^2\pi(\d \ga_1,\d\ga_2),\]
where  $\mathscr{C}(\sL_{X},\sL_Y)$ consists of all coupling of $\sL_{X},\sL_Y$ as distributions on $C^\beta([0,T];\R ^d)$. It is a pity that $C^\beta([0,T];\R ^d)$ is not a separable space under the H\"older norm, i.e. $C^\beta([0,T];\R ^d)$ is not a Polish space under the H\"older norm. In the following subsection, we introduce  a metric $\W_{2,T,\be}(\mu,\nu)$ (see \eqref{W2beT} below) for probability measure paths $\{\mu_t\}_{t\in[0,T]}$ and $\{\nu_t\}_{t\in [0,T]}$ to control the distribution part in the stochastic integral \eqref{int-si-si} (with a slight abuse
of notations $\mu,\nu$).

By using the metric $\W_{2,T,\be}$, we can establish the following well-posedness result for equation \eqref{EquM} under the hypotheses:
\beg{enumerate}
\item[\textsc{\textbf{(H)}}]
There exists a constant $K_b>0$ such that
\[|b(x,\mu)-b(y,\nu)|\leq K_b(|x-y|+W_2(\mu,\nu)),\ \ x,y\in\R^d,\mu,\nu\in\sP_2(\R^d),\]
and $\si_{ij}\in  C_b^{2,(2,0)}(\R^d\times\sP_2(\R^d)), 1\leq i,j\leq d$.
\end{enumerate}


\beg{thm}\label{ExU}
Let $H\in (\ff {\sq 5-1} 2,1)$ and $T>0$. Assume that {\bf (H)} holds and $\E e^{|X_0|^{\ff {2(1-H)} {H}+\ep_0}}<+\infty$ for some constant $\ep_0>0$.
Then equation \eqref{EquM} is well-posedness on $[0,T]$.
\end{thm}

\subsection{H\"older space of probability measure paths}

Let $T>0$ be fixed in this part. To simplify the notation, we set $W_T^d:=C([0,T];\R^d)$, and let $\sP_2(W_T^d)$ be all probability measures $\mu$ on $W_T^d$ such that $\mu(\|\cdot\|_{T,\infty}^2)<+\infty$,
in which the distance is defined as
\begin{align}\label{AddHO1}
\W_{2,T}(\mu,\nu)=\inf_{\pi\in \sC(\mu,\nu)}\left(\int_{W_T^d\times W_T^d}\|\ga_1-\ga_2\|_{T,\infty}^2\pi(\d \ga_1,\d \ga_2)\right)^\ff 1 2.
\end{align}

As mentioned in the previous part, the above metric space $(\sP_2(W_T^d),\W_{2,T})$ is not enough to overcome the negative effects induced by the multiplicative fractional noise,
so we need to introduce a new metric space.
To this end, for $\mu\in\sP_2(W_T^d)$ and $0\leq s<t\leq T$, let $\mu_s$ and $\mu_{s,t}$ stand for the respective marginals of $\mu$ at $s$ and $(s,t)$,
and let $\mu_{s,t}^\Delta$ be the distribution of the following random variable on $(\R^{2d},\sB(\R^{2d}),\mu_{s,t})$:
\[\R^d\times\R^d\ni(x,y)\mapsto x-y\in\R^d.\]
Now, for fixed $\be\in(0,1)$, we define
\begin{align}\label{AddHO3}
\sP_{2,\be}(W_T^d):=\left\{\mu\in \sP_2(W_T^d)~\Big|~\sup_{0\leq s<t\leq T}\ff {\sq{\mu_{s,t}^\Delta (|\cdot|^2)}} {(t-s)^\be}<+\infty\right\}.
\end{align}


Next, we focus on providing a distance on the set $\sP_{2,\be}(W_T^d)$.
For any $x=(x_1,x_2),y=(y_1,y_2)\in\R^{2d}$, we first let
\beg{align*}
|x-y|_M=|x_1-y_1|\vee |x_2-y_2|,
\end{align*}
and for any $\mu,\nu\in\sP_2(W_T^d)$ and any $0\leq s_1<s_2\leq T$, define
\beg{align*}
\W_{2}(\mu_{s_1,s_2},\nu_{s_1,s_2}):=\inf_{\pi_{s_1,s_2}\in\sC(\mu_{s_1,s_2},\nu_{s_1,s_2})}\left(\int_{\R^{2d}\times\R^{2d}} |x-y|_M^2\pi_{s_1,s_2}(\d x ,\d y)\right)^\ff 1 2.
\end{align*}

Let $\mu,\nu\in\sP_2(W_T^d)$. For $0\leq s_1<s_2\leq T$, and a coupling $\pi_{s_1,s_2}\in \sC(\mu_{s_1,s_2},\nu_{s_1,s_2})$, let $\pi_{s_1}$ and $\pi_{s_2}$  be the marginal distributions of $\pi_{s_1,s_2}$ at $s_1$ and $s_2$ respectively. Then $\pi_{s_i}\in\sC(\mu_{s_i},\nu_{s_i})$, $i=1,2$. Denote by $\sC_{opt}(\mu_{s_1,s_2},\nu_{s_1,s_2})$ the optimal couplings of $(\mu_{s_1,s_2},\nu_{s_1,s_2})$ with respect to the above $\W_2$-distance, 
and we define for any $\mu,\nu\in\sP_2(W_T^d)$,
\begin{align}\label{AddHO2}
W_{2,s_1,s_2}^c(\mu,\nu)=\inf\left\{\sq{\pi_{s_1,s_2}(c)}~\Big|~\pi_{s_1,s_2}\in \sC_{opt}(\mu_{s_1,s_2},\nu_{s_1,s_2})\right\},
\end{align}
where
$\pi_{s_1,s_2}(c):=\int_{\R^{2d}\times\R^{2d}} c(x,y)\pi_{s_1,s_2}(\d x ,\d y)$ with $c$ giving by
\[c(x,y)=|x_1-y_1-(x_2-y_2)|^2, \ \ x=(x_1,x_2),y=(y_1,y_2)\in\R^{2d}.\]
It is obvious that for any $\mu,\nu\in\sP_2(W_T^d)$, we have
\[W^c_{2,s_1,s_2}(\mu,\nu)=W^c_{2,s_1,s_2}(\nu,\mu)\geq 0.\]
In addition, if $\mu=\nu$, put
\[\pi_{s_1,s_2}(\d x_1,\d x_2,\d y_1,\d y_2):=\mu_{s_1,s_2}(\d x_1,\d x_2)\de_{(x_1,x_2)}(\d y_1,\d y_2).\]
It is easy to see that $\pi_{s_1,s_2}$ above belongs to $\sC_{opt}(\mu_{s_1,s_2},\nu_{s_1,s_2})$, and then by a straightforward calculation, we know that $W^c_{2,s_1,s_2}(\mu,\mu)=0$ holds true. We give a remark on couplings in $\sC(\mu_{s_1,s_2},\nu_{s_1,s_2})$. Let
\[\mu_{s_1,s_2}(\d x_1,\d x_2)=F_{\mu}(x_1,\d x_2)\mu_{s_1}(\d x_1),\quad \nu_{s_1,s_2}(\d y_1,\d y_2)=F_{\nu}(y_1,\d y_2)\nu_{s_1}(\d y_1),\]
where $F_{\mu}$ and $F_{\nu}$ are transition probability measure of $\mu_{s_1,s_2}$ and $\nu_{s_1,s_2}$ respectively. Then for any coupling $\tilde\pi_{s_1}\in \sC(\mu_{s_1},\nu_{s_1})$, we find that
\[\tilde\pi_{s_1,s_2}(\d x_1,\d x_2,\d y_1,\d y_2):=F_{\mu}(x_1,\d x_2)F_{\nu}(y_1,\d y_2)\tilde\pi_{s_1}(\d x_1,\d y_1)\in \sC(\mu_{s_1,s_2},\nu_{s_1,s_2}).\]
Then the marginal $\tilde\pi_{s_2}$ at $s_2$ of $\tilde\pi_{s_1,s_2}$ has been fixed. For  $\tilde\pi_{s_1}\in \sC(\mu_{s_1},\nu_{s_1})$ and $\tilde\pi_{s_2}\in \sC(\mu_{s_2},\nu_{s_2})$, there may be $\sC(\tilde\pi_{s_1},\tilde\pi_{s_2})\cap \sC(\mu_{s_1,s_2},\nu_{s_1,s_2})=\emptyset$.

The following lemma shows that  $W^c_{2,s_1,s_2}$ is a degenerate metric.

\begin{lem}\label{lem:metric1}
Let $\mu,\nu\in \sP_2(W_T^d)$.
Then for any $0\leq s_1<s_2\leq T$, there is a $\tilde{\pi}_{s_1,s_2}\in\sC_{opt}(\mu_{s_1,s_2},\nu_{s_1,s_2})$ such that
\[W^c_{2,s_1,s_2}(\mu,\nu)=\sq{\tilde{\pi}_{s_1,s_2}(c)},\]
and $W^c_{2,s_1,s_2}$ satisfies the triangle inequality.
\end{lem}
\beg{proof}
First note that, by \cite[Lemma 4.4 and the proof of Theorem 4.1]{Vil} we obtain that for any $\mu,\nu\in \sP_2(W_T^d)$ and $0\leq s_1<s_2\leq T$, $\sC(\mu_{s_1,s_2},\nu_{s_1,s_2})$ is a weakly compact set.
The next observation is that $\sC_{opt}(\mu_{s_1,s_2},\nu_{s_1,s_2})$ is weakly closed. As a consequence, it is in fact weakly compact.
Indeed, let $\{\hat{\pi}_{s_1,s_2}^k\}_{k\geq1}$ be a sequence of $\sC_{opt}(\mu_{s_1,s_2},\nu_{s_1,s_2})$ which converges weakly to some $\hat{\pi}_{s_1,s_2} \in \sC(\mu_{s_1,s_2},\nu_{s_1,s_2})$,
we derive from \cite[Lemma 4.3 (with $h\equiv 0$)]{Vil} that
\[\int_{\R^{2d}\times\R^{2d}}|x-y|_M^2\hat{\pi}_{s_1,s_2} (\d x,\d y)\leq\liminf_{k\ra+\infty}\int_{\R^{2d}\times\R^{2d}}|x-y|_{M}^2\hat{\pi}_{s_1,s_2}^{k}(\d x,\d y)=\W_2(\mu_{s_1,s_2},\nu_{s_1,s_2})^2.\]
Then, we have
\[\int_{\R^{2d}\times\R^{2d}}|x-y|_M^2\hat{\pi}_{s_1,s_2} (\d x,\d y)=\W_2(\mu_{s_1,s_2},\nu_{s_1,s_2})^2,\]
which means that $\hat{\pi}_{s_1,s_2}$ is optimal, i.e.  $\hat{\pi}_{s_1,s_2}\in\sC_{opt}(\mu_{s_1,s_2},\nu_{s_1,s_2})$.
Hence,  $\sC_{opt}(\mu_{s_1,s_2},\nu_{s_1,s_2})$ is weakly compact.

Now, for a sequence $\{\tilde{\pi}_{s_1,s_2}^k\}_{k\geq1}$ in $\sC_{opt}(\mu_{s_1,s_2},\nu_{s_1,s_2})$ such that
\[\lim_{k\ra +\infty}\tilde{\pi}_{s_1,s_2}^k(c)=W^c_{2,s_1,s_2}(\mu,\nu)^2,\]
extracting  a subsequence if necessary, we may assume that $\tilde{\pi}_{s_1,s_2}^k$ converges weakly to some $\tilde{\pi}_{s_1,s_2}\in\sC_{opt}(\mu_{s_1,s_2},\nu_{s_1,s_2})$.
Due to the lower semicontinuity of $c$ and the fact that $c\geq 0$, we deduce from \cite[Lemma 4.3 (with $h\equiv 0$)]{Vil} again that
\[\tilde{\pi}_{s_1,s_2}(c)\leq\liminf_{k\ra +\infty}\pi_{s_1,s_2}^k(c)=W^c_{2,s_1,s_2}(\mu,\nu)^2.\]
Thus $\tilde{\pi}_{s_1,s_2}$ is minimizing, i.e. $\tilde{\pi}_{s_1,s_2}(c)=W^c_{2,s_1,s_2}(\mu,\nu)^2$.

Next, we want to  prove that $W^c_{2,s_1,s_2}$ satisfies the triangle inequality.
For this, let $\mu,\nu,\chi\in \sP_2(W_T^d)$,
and let random variables $(\hat X,\hat Z)$ and $(\hat Y,\tilde Z)$ be respectively couplings of $(\mu,\chi)$ and $(\nu,\chi)$  (namely $\sL_{\hat X}=\mu,\sL_{\hat Z}=\chi=\sL_{\tilde Z},\sL_{\hat Y}=\nu$)
such that their laws are respectively couplings of $(\mu_{s_1,s_2},\chi_{s_1,s_2})$ and $(\nu_{s_1,s_2},\chi_{s_1,s_2})$  and
\beg{align}\label{1Pflem:metric1}
\E \big|\hat X_1-\hat Z_1-(\hat X_2-\hat Z_2)\big|^2=W_{2,s_1,s_2}^c(\mu,\chi),\ \ \
\E \big|\hat Y_1-\tilde Z_1-(\hat Y_2-\tilde Z_2)\big|^2=W_{2,s_1,s_2}^c(\nu,\chi).
\end{align}
Here for $i=1,2$,  $\hat X_i,\hat Y_i,\hat Z_i,\tilde Z_i$ are components of $\hat X,\hat Y,\hat Z,\tilde Z$ satisfying $\sL_{\hat X_i}=\mu_{s_i},\sL_{\hat Y_i}=\nu_{s_i},\sL_{\hat Z_i}=\sL_{\tilde Z_i}=\chi_{s_i}$, respectively.
It follows from \cite[Gluing lemma, Page 11-12]{Vil} that there is a triple of random variables $(X,Y,Z)$ such that $\sL_{(X,Z)}=\sL_{(\hat X,\hat Z)}$ and
$\sL_{(Y, Z)}=\sL_{(\hat Y,\tilde Z)}$.
Using the Minkowski inequality and \eqref{1Pflem:metric1}, we then conclude that
\beg{align*}
W^c_{2,s_1,s_2}(\mu,\nu)
&\leq\sq{\E |X_1-Y_1-(X_2-Y_2)|^2}\cr
&\leq \sq{\E |X_1-Z_1-(X_2-Z_2)|^2}+\sq{\E |Y_1-  Z_1-(Y_2- Z_2)|^2}\\
&=\sq{\E\big|\hat X_1-\hat Z_1-(\hat X_2-\hat Z_2)\big|^2}+\sq{\E\big|\hat Y_1-\tilde Z_1-(\hat Y_2-\tilde Z_2)\big|^2}\\
&=W_{2,s_1,s_2}^c(\mu,\chi)+W_{2,s_1,s_2}^c(\nu,\chi).
\end{align*}
where for $i=1,2$,  $X_i,Y_i,Z_i$ are components of $X,Y,Z$ satisfying $\sL_{X_i}=\mu_{s_i},\sL_{Y_i}=\nu_{s_i},\sL_{Z_i}=\chi_{s_i}$, respectively.
Therefore, the triangle inequality holds for $W^c_{2,s_1,s_2}$.
This completes the proof.
\end{proof}

\beg{lem}\label{cor:Wc-0}
(1) For any $\mu\in\sP_2(W_T^d)$ and $0\leq s_1<s_2\leq T$, we have
\[W_{2,s_1,s_2}^c(\mu,\de_{\bf 0})=\sq{\mu_{s_1,s_2}^\Delta(|\cdot|^2)}\leq \|\mu_{s_1}\|_2+\|\mu_{s_2}\|_2,\]
where and in what follows,  $\de_{\bf 0}$ denotes the Dirac measure at $\bf 0\in\R^d$.\\
(2) For any $\mu,\nu\in \sP_2(W_T^d)$, we have
\[ \sup_{0\leq s_1<s_2\leq T}W_{2,s_1,s_2}^c(\mu,\nu)\leq 2\sup_{0\leq s_1<s_2\leq T}\W_{2}(\mu_{s_1,s_2},\nu_{s_1,s_2})\leq 2\W_{2,T}(\mu,\nu).\]
\end{lem}

\beg{proof}
(1) For any $\mu\in\sP_2(W_T^d)$ and $0\leq s_1<s_2\leq T$, let
\[\pi_{s_1,s_2}(\d x_1,\d x_2,\d y_1,\d y_2)=\mu_{s_1,s_2}(\d x_1,\d x_2)\de_{(\bf 0,\bf 0)}(\d y_1,\d y_2).\]
Then, it is readily checked that $\pi_{s_1,s_2}\in\sC_{opt}(\mu_{s_1,s_2},\de_{(\bf 0,\bf 0)})$, and moreover we have
\beg{align*}
W_{2,s_1,s_2}^c(\mu,\de_{\bf 0})&\leq\sq{\pi_{s_1,s_2}(c)}
=\sq{\int_{\R^d\times\R^d}|x_1-x_2|^2\mu_{s_1,s_2}(\d x_1,\d x_2)}\\
&=\sq{\mu_{s_1,s_2}^\Delta(|\cdot|^2)}
\leq \|\mu_{s_1}\|_2+\|\mu_{s_2}\|_2.
\end{align*}

(2) For any $\mu,\nu\in \sP_{2 }(W_T^d)$ and $0\leq s_1<s_2\leq T$, there is
\beg{equation}\label{sc-sc}
\sC(\mu_{s_1,s_2},\nu_{s_1,s_2})\supset\{\pi_{s_1,s_2} ~|~\pi \in\sC(\mu,\nu)\}.
\end{equation}
Thus  $\W_{2,T}(\mu,\nu)\geq \W_2(\mu_{s_1,s_2},\nu_{s_1,s_2})$ holds true.
On the other hand, owing to
\[c(x,y)\leq 4 |x-y|_M^2,\ \ x,y\in\R^{2d},\]
we derive that
\[ W_{2,s_1,s_2}^c(\mu,\nu)\leq 2\W_2(\mu_{s_1,s_2},\nu_{s_1,s_2}).\]
The result follows.
\end{proof}

Let
\beg{equation}\label{W2beT}
\W_{2,T,\be}(\mu,\nu)=\W_{2,T}(\mu,\nu)+\sup_{0\leq s_1<s_2\leq T}\ff {W_{2,s_1,s_2}^c(\mu,\nu)} {(s_2-s_1)^\be},\ \ \mu,\nu\in\sP_{2,\be}(W_T^d).
\end{equation}
Then we have the following result.
\begin{thm}\label{ThProb}
The space $\sP_{2,\be}(W_T^d)$ is a complete metric space under the metric $\W_{2,T,\be}$.
\end{thm}
\begin{proof}
We divide the proof into two steps.

\emph{Step 1. Claim: $\W_{2,T,\be}$ is a metric on the space $\sP_{2,\be}(W_T^d)$.}
According to Lemma \ref{lem:metric1}, it is sufficient to show that $\W_{2,T,\be}$ is finite on the space $\sP_{2,\be}(W_T^d)$.
Applying Lemma \ref{cor:Wc-0}(1) and the triangle inequality, we obtain that for any $\mu,\nu\in\sP_{2,\be}(W_T^d)$,
\beg{align*}
&\sup_{0\leq s_1<s_2\leq T}\ff {W_{2,s_1,s_2}^c(\mu,\nu)} {(s_2-s_1)^\be}\\
\leq& \sup_{0\leq s_1<s_2\leq T}\ff {W_{2,s_1,s_2}^c(\mu,\de_{\bf 0})} {(s_2-s_1)^\be}+\sup_{0\leq s_1<s_2\leq T}\ff {W_{2,s_1,s_2}^c(\nu,\de_{\bf 0})} {(s_2-s_1)^\be}\\
=& \sup_{0\leq s_1<s_2\leq T}\ff {\sq{\mu_{s_1,s_2}^\Delta(|\cdot|^2)}} {(s_2-s_1)^\be}+\sup_{0\leq s_1<s_2\leq T}\ff {\sq{\nu_{s_1,s_2}^\Delta(|\cdot|^2)}} {(s_2-s_1)^\be}\\
<&+\infty,
\end{align*}
which, along with \eqref{AddHO1}, yields the desired assertion.

\emph{Step 2. Claim: The metric space $(\sP_{2,\be}(W_T^d),\W_{2,T,\be})$ is complete.}
Suppose that $\{\mu^m\}_{m\geq 1}$ is a Cauchy sequence in the metric space $(\sP_{2,\be}(W_T^d),\W_{2,T,\be})$.
In view of the definition of $\W_{2,T,\be}$, $\{\mu^m\}_{m\geq 1}$ is also a Cauchy sequence in the metric space $(\sP_{2 }(W_T^d),\W_{2,T})$.
Then, there is a $\mu\in \sP_2(W_T^d)$ such that
\beg{align}\label{conv-W2T}
\lim_{m\ra+\infty}\W_{2,T}(\mu^m,\mu)=0.
\end{align}
Consequently, it follows from Lemma \ref{cor:Wc-0}(2) that
\beg{align}\label{Wc-mum-mu}
\lim_{m\ra+\infty} \sup_{0\leq s_1<s_2\leq T}W_{2,s_1,s_2}^c(\mu^m,\mu)=0,
\end{align}
which, together with the triangle inequality, implies
\beg{align}\label{1PfCoME}
\lim_{m\ra +\infty}W_{2,s_1,s_2}^c(\mu^m,\de_{\bf 0})=W_{2,s_1,s_2}^c(\mu,\de_{\bf 0}),\ \ 0\leq s_1<s_2\leq T.
\end{align}
Combining this with Lemma \ref{cor:Wc-0}(1), we derive that
\begin{align*}
\sup_{0\leq s_1<s_2\leq T}\ff {\sq{\mu_{s_1,s_2}^\Delta(|\cdot|^2)}} {(s_2-s_1)^\be}
&= \sup_{0\leq s_1<s_2\leq T}\ff {W_{2,s_1,s_2}^c(\mu,\de_{\bf 0})} {(s_2-s_1)^\be}\\
&= \sup_{0\leq s_1<s_2\leq T}\left(\lim_{m\ra+\infty} \ff {W_{2,s_1,s_2}^c(\mu^m,\de_{\bf 0})} {(s_2-s_1)^{\be}}\right)\\
&\leq  \sup_{0\leq s_1<s_2\leq T}\left(\sup_{m\geq 1}\ff {W_{2,s_1,s_2}^c(\mu^m,\de_{\bf 0})} {(s_2-s_1)^\be}\right)\\
&=\sup_{m\geq 1}\left(\sup_{0\leq s_1<s_2\leq T}\ff {W_{2,s_1,s_2}^c(\mu^m,\de_{\bf 0})} {(s_2-s_1)^\be}\right)\\
&<+\infty,
\end{align*}
where in the last inequality we have used that the Cauchy sequence is bounded, which can be derived from  the triangle inequality of $W_{2,s_1,s_2}^c$. This means that $\mu\in \sP_{2,\be}(W_T^d)$.

Next, we intend to show that $\{\mu^m\}_{m\geq 1}$ converges to $\mu$ in the metric space $(\sP_{2,\be}(W_T^d),\W_{2,T,\be})$.
Since $\{\mu^m\}_{m\geq 1}$ is a Cauchy sequence in $(\sP_{2,\be}(W_T^d),\W_{2,T,\be})$,
we deduce that for every $\ep>0$, there exists $N\in\N$ such that
\beg{align}\label{2PfCoME}
\W_{2,T,\be}(\mu^m,\mu^n)<\ep,\ \ n,m\geq N.
\end{align}
As in \eqref{1PfCoME}, the triangle inequality and \eqref{Wc-mum-mu} imply that
\begin{align*}
\lim_{n\ra +\infty} W_{2,s_1,s_2}^c(\mu^m,\mu^n)=W_{2,s_1,s_2}^c(\mu^m,\mu).
\end{align*}
This, along with \eqref{2PfCoME}, leads to
\begin{align*}
\sup_{0\leq s_1<s_2\leq T}\ff {W_{2,s_1,s_2}^c(\mu^m,\mu)} {(s_2-s_1)^\be}&=\sup_{0\leq s_1<s_2\leq T}\lim_{n\ra +\infty} \ff {W_{2,s_1,s_2}^c(\mu^m,\mu^n)} {(s_2-s_1)^\be}\cr
&\leq\limsup_{n\ra +\infty}\sup_{0\leq s_1<s_2\leq T}\ff {W_{2,s_1,s_2}^c(\mu^m,\mu^n)} {(s_2-s_1)^\be}\cr
&\leq\limsup_{n\ra +\infty}\W_{2,T,\be}(\mu^m,\mu^n)\leq \ep,\ \ m\geq N.
\end{align*}
As a consequence, we have
\[\lim_{m\ra+\infty} \sup_{0\leq s_1<s_2\leq T}\ff {W_{2,s_1,s_2}^c(\mu^m,\mu)} {(s_2-s_1)^\be}=0.\]
In conjunction with \eqref{conv-W2T}, this yields that $\{\mu^m\}_{m\geq 1}$ converges to $\mu$ in $(\sP_{2,\be}(W_T^d),\W_{2,T,\be})$,
namely $(\sP_{2,\be}(W_T^d),\W_{2,T,\be})$ is complete.
The proof is now finished.
\end{proof}

Let us mention that Theorem \ref{ThProb} can extend naturally to the subinterval of $[0,T]$.
That is, for each $S\in[0,T)$, let $W_{S,T}^d=C([S,T];\R^d)$, and $\W_{2,S,T}, \W_{2,S,T,\be}$ be defined on $\sP_{2,\be}(W_{S,T}^d)$ similarly
to \eqref{AddHO1} and \eqref{W2beT}, then $(\sP_{2,\be}(W_{S,T}^d),\W_{2,S,T,\be})$ is a complete metric space.
Finally,
let
\begin{align}\label{AddDef}
\|\mu\|_{2,S,T,\be}&=\sq{\mu(\|\cdot\|_{S,T,\infty}^2)}+\sup_{S\leq s<t\leq T}\ff {\sq{\mu_{s,t}^\Delta(|\cdot|^2)}} {(t-s)^\be},
\end{align}
and we simply write $\|\mu\|_{2,T,\be}$ as $\|\mu\|_{2,S,T,\be}$ if $S=0$. We remark that $\|\mu\|_{2,S,T,\be}=\W_{2,S,T,\be}(\mu,\de_{\bf 0})$.

\subsection{Proof of Theorem \ref{ExU}}

In this part, we will use the contraction mapping principle to prove the existence and uniqueness of solution to equation \eqref{EquM}.
To this end, we first investigate the following equation with the distribution part being frozen on $[S,T]$:
\begin{align}\label{X-mu}
\d X^\mu_t=b(X^\mu_t,\mu_t)\d t+\si(X_t^\mu,\mu_t)\d B_t^H,\ \ X^\mu_S=X_S,~ S\leq t\leq T,
\end{align}
where $S\in[0,T)$.
If $\mu_S=\sL_{X_S}\in\sP_2(\R^d)$ and the equation above is well-posed for any $\mu\in\sP_{2,\be}(W_{S,T}^d)$,  then there is a mapping $\sP_{2,\be}(W_{S,T}^d)\ni\mu\mapsto\sL_{X^\mu}$.
Now, we define for $M>0$,
\beg{align*}
 \sP_{M,\sL_{X_S},S,T,\be}=\left\{\mu\in \sP_{2,\be}(W^{d}_{S,T})~\big|~ \|\mu\|_{2,S,T,\be}\leq M(1+\|\mu_S\|_2),\ \ \mu_S=\sL_{X_S}\right\},
\end{align*}
which is a closed subset of $\sP_{2,\be}(W^{d}_{S,T})$.
Below we will prove that $\sP_{M,\sL_{X_S},S,T,\be}$ is invariant for the mapping $\mu\mapsto\sL_{X^\mu}$.

\begin{prp}\label{lem-inva}
Assume that {\bf (H)} holds. Let $0\leq S<T$ and $\ff 1 {2}<\be<H$. Then, there are positive constants $R_0$ and $M_0$, which are independent of $S,T,\sL_{X_S}$, such that for any $T-S\leq R_0$ and $M\geq M_0$, the set $\sP_{M, \sL_{X_S},S,T,\be}$ is invariant for the mapping $\mu\mapsto \sL_{X^\mu}$, that is, $\mu\in\sP_{M,\sL_{X_S},S,T,\be}$ implies $\sL_{X^\mu}\in\sP_{M,\sL_{X_S},S,T,\be}$.
\end{prp}

\begin{proof}
Owing to {\bf (H)}, we obtain that for any $r,s\in[S,T], x,y\in\R^d$ and $\mu\in\sP_{2,\be}(W^{d}_{S,T})$,
\beg{align*}
|b(x,\mu_r)|&\leq |b(x,\mu_r)-b({\bf 0}, {\bf{\de_{0 }}})|+|b({\bf 0},{\bf {\de_{0 }}})|\\
&\leq  K_b|x|+K_b W_2(\mu_r,{\bf\de_0})+|b({\bf 0},{\bf\de_0})|\\
&=  K_b|x|+K_b\|\mu_r\|_2+|b({\bf 0},{\bf\de_0})|\\
&\leq  K_b|x|+K_b\|\mu\|_{2,S,T,\be}+|b({\bf 0},{\bf\de_0})|\\
&\leq K_b\|\mu\|_{2,S,T,\be}+(K_b\vee|b({\bf 0},{\bf\de_0})| )(1+|x|)
\end{align*}
and
\begin{align*}
|\si(x,\mu_r)-\si(y,\mu_s)|&\leq   \|\na\si\|_\infty|x-y|+ \|D^L\si\|_\infty W_2(\mu_r,\mu_s)\\
&\leq  \|\na\si\|_\infty|x-y|+\|D^L\si\|_\infty\|\mu\|_{2,S,T,\be}(r-s)^{\be}.
\end{align*}
Then invoking Lemma \ref{mom-est} with
\[K_{\tilde{b}}=K_b\|\mu\|_{2,S,T,\be},\quad L_{\tilde{b}}=K_b\vee|b({\bf 0},{\bf\de_0})|,\quad K_{\tilde{\sigma}}=\|D^L\si\|_\infty\|\mu\|_{2,S,T,\be},\quad L_{\tilde{\si}}=\|\na\si\|_\infty,\quad \ga_0=\be,\]
we derive that for each $\mu\in\sP_{M,\mu_S,S,T,\be}$,
\begin{align}\label{in-sLX-X}
&\|\sL_{X^\mu}\|_{2,S,T,\be}=\sqrt{\E\|X^\mu\|^2_{S,T,\infty}}+\sup_{S\leq s<t\leq T}\ff {\sqrt{\E|X^\mu_t-X^\mu_s|^2}} {(t-s)^{\be}}\cr
&\leq \sqrt{\E\|X^\mu\|^2_{S,T,\infty}}+\sq{\E\|X^\mu\|_{S,T,\be}^2}\cr
&\leq C_1(T-S,H,\be)\left(1+ (K_b+\|D^L\si\|_\infty)\|\mu\|_{2,S,T,\be}\right)+ C_2(T-S,H,\be)\left(\E|X_S|^2 \right)^{\ff 1 2}\cr
&\leq C_1(T-S,H,\be)(1+ (K_b+\|D^L\si\|_\infty)M\left(1+\|\mu_S\|_2\right)) +C_2(T-S,H,\be)\|\mu_S\|_2\cr
&\leq (K_b+\|D^L\si\|_\infty)M\left(1+\|\mu_S\|_2\right) +\left(C_1(T-S,H,\be)+C_2(T,H,\be)\|\mu_S\|_2\right)\cr
&\leq \left(C_1(T-S,H,\be)(K_b+\|D^L\si\|_\infty)M \right.\cr
&\quad\left. +C_1(T-S,H,\be)\vee C_2(T-S,H,\be)\right)\left(1+\|\mu_S\|_2\right),
\end{align}
where $C_i(T-S,H,\be),i=1,2$ are two positive constants, which are independent of $S,T,\sL_{X_S}$ and  decreasing as $T-S$ decreases and satisfy
\begin{align}\label{1Pflem-inva}
\lim_{T-S\rightarrow 0^+}C_1(T-S,H,\be)=0\ \ \ \mathrm{and}\ \ \ \lim_{T-S\rightarrow 0^+}C_2(T-S,H,\be)=1.
\end{align}

Now, let
\begin{align*}
R_0=\inf\left\{R>0~|~C_1(R,H,\be) (K_b+\|D^L\si\|_\infty)\geq \ff 1 2\right\}.
\end{align*}
Then, it is easy to show that $R_0>0$ holds true because of \eqref{1Pflem-inva}.
For any $T-S\leq R_0$, we set
\begin{align*}
M_0:=2\left(C_1(R_0,H,\be)\vee C_2(R_0,H,\be)\right).
\end{align*}
By \eqref{in-sLX-X}, we know that for any $T-S\leq R_0$ and $M\geq M_0$,
\[\|\sL_{X^\mu}\|_{2,S,T,\be}\leq \left(\ff M 2+\ff {M_0} 2\right)\left(1+\|\mu_S\|_2\right)\leq M\left(1+\|\mu_S\|_2\right).\]
Thus, the set $\sP_{M,\sL_{X_S},S,T,\be}$ is invariant for the mapping $\mu\mapsto \sL_{X^\mu}$.
This completes the proof.
\end{proof}
With Proposition \ref{lem-inva} in hand, we can show that for a solution $X$ to equation \eqref{EquM}, the law $\|\sL_{X}\|_{2,S,T,\be}$ can be controlled by $\|\sL_{X_S}\|_2$ as in probability measure paths in $\sP_{M,\sL_{X_S},S,T,\be}$.
\beg{cor}\label{pri-est-X}
Assume that {\bf (H)} holds. Let $0\leq S<T$ and $\be\in (\ff 1 2,H)$.
If $\{X_t\}_{t\in [S,T]}$ is a solution of equation \eqref{EquM}, then there is a constant $R_0>0$, which is independent of $\sL_{X_S},S,T$ and the same as $R_0$ in Proposition \ref{lem-inva}, such that for any $T-S\leq R_0$,
\beg{equation}\label{pri-XX0}
\|\sL_{X}\|_{2,S,T,\be}\leq 2(C_1(R_0,H,\be)\vee C_2(R_0,H,\be))(1+\|\sL_{X_S}\|_2).
\end{equation}
Consequently,  there is a constant $C(T,H,\be)>0$ such that
\beg{equation}\label{XTbe-L0}
\|\sL_{X}\|_{2,T,\be}\leq C(T,H,\be)(1+\|\sL_{X_0}\|_2).
\end{equation}
Moreover, if $H>\ff {\sq 5-1} 2$ and there is $\ep_0>0$ such that $\E e^{|X_0|^{\ff {2(1-H)} {H}+\ep_0}}<+\infty$, then for any $\ep_0'\in (0,\ep_0\wedge \ff {2(H^2+H-1)} {H})$  and $\beta\in(\ff 1 2\vee(\ff {1-H} {H}+\ff {\ep_0'} 2),H)$, there is
\beg{equation}\label{exp+ep}
\mathcal{E}_{T,\be,X_0,\ep'_0}:=\sup_{t\in [0,T]}\E e^{|X_t|^{\ff {2(1-H)} {H}+\ep_0'}}<+\infty.
\end{equation}
\end{cor}
\beg{proof}
By {\bf (H)}, we may set in Lemma \ref{mom-est} with
\[K_{\tilde{b}}=K_b\|\sL_X\|_{2,S,T,\be},\quad L_{\tilde{b}}=K_b\vee|b({\bf 0},{\bf\de_0})|,\quad K_{\tilde{\sigma}}=\|D^L\si\|_\infty\|\sL_X\|_{2,S,T,\be},\quad L_{\tilde{\si}}=\|\na\si\|_\infty,\quad \ga_0=\be.\]
Then, along the same lines as in \eqref{in-sLX-X} with $\mu$ replaced by $\sL_{X}$, we obtain that
\beg{align}\label{in-sLX-X0}
\|\sL_{X}\|_{2,S,T,\be}\leq C_1(T-S,H,\be)\left[1+\left(K_b+\|D^L\si\|_\infty\right)\|\sL_{X}\|_{2,S,T,\be}\right]+C_2(T-S,H,\be)\|\sL_{X_S}\|_2,
\end{align}
where $C_i(T-S,H,\be),i=1,2$ are two positive constants, which are decreasing as $T-S$ decreases and satisfy
\beg{align*}
\lim_{T-S\ra 0^+}C_1(T-S,H,\be)=0\ \ \ \mathrm{and}\ \ \ \lim_{T-S\ra 0^+}C_2(T-S,H,\be)=1.
\end{align*}

Next, let
$$R_0=\inf\left\{R>0~\big|~C_1(R,H,\be)\left(K_b+\|D^L\si\|_\infty\right)\geq \ff 1 2\right\}.$$
Then by \eqref{in-sLX-X0}, we derive that there holds \eqref{pri-XX0} for any $T-S\leq R_0$.
Observe that $R_0$ is independent of $\sL_{X_S},S,T$.
Hence, \eqref{XTbe-L0} can be proved by using \eqref{pri-XX0} recursively.

For the last assertion, we see from \eqref{ad-supn} that
\beg{align}\label{1Pf-priesX}
|X_t|&\leq \exp\left\{ 2CT+1\wedge T^{\be }\right\}  |X_0|+ (1\wedge T^{\be })e^{2CT} \left(1+G(T,K_{\tilde{b}},K_{\tilde{\sigma}},B^H)\right)
\end{align}
where $G(T,K_{\tilde{b}},K_{\tilde{\sigma}},B^H)$ is defined in \eqref{Fan-add4} with
\[K_{\tilde{b}}=K_b\|\sL_X\|_{2,T,\be},\quad \quad K_{\tilde{\sigma}}=\|D^L\si\|_\infty\|\sL_X\|_{2,T,\be},\quad \quad \ga_0=\be.\]
By \eqref{XTbe-L0}, there is a constant $C(T,H,\be)\geq 1$ such that
\beg{align*}
G(T,K_{\tilde{b}},K_{\tilde{\sigma}},B^H)\leq C(T,H,\be)\left(\|B^H\|_{T,\be}^{\ff 1 {\be}}+(1+\|\sL_{X_0}\|_{2})\left(1+\|B^H\|_{T,\be}^{\ff 1 {2\be}}\right)\right).
\end{align*}
For $H>\ff {\sq 5-1} 2$ and $0<\ep<\ff {2(H^2+H-1)} {H}$, there is  $\ff {2(1-H)} {H}+\ep<2H$, which implies  that for any $\be> \ff 1 2\vee(\ff {1-H} {H}+\ff {\ep} 2)$,
\[\ff {\ff {2(1-H)} {H}+\ep} {2\be}<\ff {\ff {2(1-H)} {H}+\ep} {\be}<2.\]
Then,  according to the Fernique theorem (see, e.g., \cite[Theorem 1.3.2]{Fernique75a} or \cite[Lemma 8]{Saussereau12}), we obtain that for any $\ep\in (0,\ff {2(H^2+H-1)} {H})$ and $\beta\in(\ff 1 2\vee(\ff {1-H} {H}+\ff {\ep} 2),H)$ and any $c>0$,
\beg{align}\label{2Pf-priesX}
\E e^{c(G(T,K_{\tilde{b}},K_{\tilde{\sigma}},B^H))^{\ff {2(1-H)} {H}+\ep}}<+\infty.
\end{align}
Hence, for any $0<\ep'_0<\ep_0\wedge \ff {2(H^2+H-1)} {H}$, taking into account that
\[\E\exp\{c|X_0|^{\ff {2(1-H)} {H}+\ep_0'}\}<+\infty,\ \ c>0,\]
we conclude from \eqref{1Pf-priesX}, \eqref{2Pf-priesX} and the Young inequality that for any $\beta\in(\ff 1 2\vee(\ff {1-H} {H}+\ff {\ep_0'} 2),H)$, there holds \eqref{exp+ep}.
The proof is complete.
\end{proof}

The following lemma gives an estimate for the difference between two stochastic integrals with different frozen distribution parts $\mu$ and $\nu$,
which will play an important role in the proof of Proposition \ref{LeContr} below.

\begin{lem}\label{EsNoi}
Let $\be\in (\ff 1 2,H)$, $\be_1\in [\be,H)$ and $1-\be<\al<\be$.
Assume that  $\mu,\nu\in\sP_{2,\be}(W_{S,T}^d)$ and $\si_{ij}\in  C_b^{2,(2,0)}(\R^d\times\sP_2(\R^d)), 1\leq i,j\leq d$.
Let $X^\mu$ and $X^\nu$ be respectively the solutions of equation \eqref{X-mu} with the frozen distribution parts $\mu$ and $\nu$.
Then for any $S\leq s<t\leq T$ with $t-s\leq 1$,
\begin{align*}
&\left|\int_s^t\left(\si(X_r^\mu,\mu_r)-\si(X_r^{\nu},\nu_r)\right)\d B^H_r\right|\cr
\le&\|B^H\|_{s,t,\be_1} \left(\Lambda_1 (t-s)^{\be_1}+\Lambda_2  (t-s)^{\al+\be_1}\right)\|X^\mu-X^\nu\|_{s,t,\infty}\cr
&+\Lambda_3 \|B^H\|_{s,t,\be_1} (t-s)^{\be+\be_1}\|X^\mu-X^\nu\|_{s,t,\be}\cr
&+\|B^H\|_{s,t,\be_1} \left(\Lambda_4 (t-s)^{\be_1}+\Lambda_5 (t-s)^{\al+\be_1}\right)\W_{2,S,T,\be}(\mu ,\nu),
\end{align*}
where
\begin{equation}\label{1-EsNoi}
\beg{split}
&\Lambda_1 := \ff {C_0\mathcal{B}(1-\al,\al+\be_1)} {\Gamma(1-\al)} \|\na\si\|_\infty, \\
&\Lambda_2 := \ff {2^{3-\ff {\al} {\be}} \be C_0\|\na\si\|_\infty^{\ff {\be-\al} {\be}}} {(\be-\al)(\al+\be_1)\Gamma(1-\al)}\Big[\left( \|\na^2\si\|_\infty \left( {\|X^\nu\|_{s,t,\be}\vee\|X^\mu\|_{s,t,\be}} \right)\right)^{\ff {\al} {\be}}\cr
&\qquad +\left( \|D^L\na\si\|_\infty \left(\W_{2,S,T,\be}(\mu,\de_{\bf 0})\wedge \W_{2,S,T,\be}(\nu,\de_{\bf 0})\right)\right)^{\ff {\al} {\be}} \Big], \\
&\Lambda_3 :=\ff {\al C_0\|\na\si\|_\infty\mathcal{B}(\be-\al+1,\al+\be_1)} {(\be-\al)\Gamma(1-\al)}, \\
&\Lambda_4 :=\ff {C_0\|D^L \si\|_\infty } {\Gamma(1-\al)} \left( \mathcal{B}(1-\al,\al+\be_1)+\ff {\al  \mathcal{B}(\be-\al+1,\be_1+\al)} {\be-\al} \right), \cr
&\Lambda_5 :=\ff {2^{3-\ff {\al} {\be}} \be C_0 \|D^L\si\|_\infty^{\ff {\be-\al} {\be}}} {(\be-\al)(\be_1+\al)\Gamma(1-\al)} \Big[ (\|\na_1 D^L\si\|_\infty (\|X^\nu\|_{s,t,\be}\wedge\|X^\mu\|_{s,t,\be}))^{\ff {\al} {\be}}\cr
&\qquad + \left((\|D^{L,2}\si\|_\infty+\|\na_2 D^L\si\|_\infty)( \W_{2,S,T,\be}(\mu,\de_{\bf 0})\vee\W_{2,S,T,\be}(\nu,\de_{\bf 0})) \right)^{\ff {\al}{\be}}  \Big],
\end{split}
\end{equation}
in which $C_0:=\ff{\be_1}{\Gamma(\alpha)(\alpha+\beta_1 -1)}$, $\na_1 D^L\si$ and $\na_2 D^L\si$ stand for the gradient operators of $D^L\si(x,\cdot)(\mu)(y)$ on $x$ and $y$, respectively.
\end{lem}

\begin{proof}
With the help of the fractional integration by parts formula \eqref{Fibpf}, one can see that for any $1-\be<\al<\be$ and $S\leq s<t\leq T$,
\begin{align}\label{1PfExU}
&\int_s^t\left(\si(X_r^\mu,\mu_r)-\si(X_r^{\nu},\nu_r)\right)\d B^H_r\cr
=&(-1)^\al\int_s^t {D_{s+}^\al} \left(\si({X_{.}}^\mu,\mu_\cdot)-\si({{X_{.}}^{\nu}},\nu_\cdot )\right)(r)\cdot D_{t-}^{1-\al}B^H_{t-}(r)\d r.
\end{align}
By \eqref{FrDe} and the condition imposed on $\si$, we derive that for every $r\in[s,t]$ and $\be_1\in[\be,H)$,
\begin{align}\label{3PfExU}
\left|D_{t-}^{1-\al}B^H_{t-}(r)\right|
=&\frac{1}{\Gamma(\alpha)}\left|\frac{B^H_r-B^H_t}{(t-r)^{1-\alpha}}+(1-\alpha)\int_r^t\frac{B^H_r-B^H_u}{(u-r)^{2-\alpha}}\d u\right|\cr
\le&\ff{\be_1}{\Gamma(\alpha)(\alpha+\beta_1 -1)}\|B^H\|_{s,t,\be_1}(t-r)^{\al+\be_1-1}\cr
=:& C_0\|B^H\|_{s,t,\be_1}(t-r)^{\al+\be_1-1}
\end{align}
and
\begin{align}\label{2PfEsNoi}
&|D_{s+}^\alpha(\sigma(X_\cdot^\mu,\mu_\cdot)-\sigma(X_\cdot^\nu,\nu_\cdot))(r)|\cr
=&\frac{1}{\Gamma(1-\alpha)}
\bigg|\frac{\sigma(X_r^\mu,\mu_r)-\sigma(X_r^\nu,\nu_r)}{(r - s)^\alpha}\cr
&\qquad\qquad+\alpha\int_s^r\frac{\sigma(X_r^\mu,\mu_r)-\sigma(X_r^\nu,\nu_r)-\left(\sigma(X_u^\mu ,\mu_u)-\sigma(X_u^\nu,\nu_u)\right)}
{(r - u)^{1+\alpha}}\d u\bigg|\cr
\leq&\frac{1}{\Gamma(1-\alpha)}
\Bigg[\frac{\|\na\si\|_\infty\|X^\mu-X^\nu\|_{s,r,\infty}+\|D^L \si\|_\infty\W_{2,S,T,\be}(\mu,\nu)}{(r - s)^\alpha}\cr
&\qquad\qquad+\alpha\int_s^r\frac{\left|\sigma(X_r^\mu,\mu_r)-\sigma(X_r^\nu,\nu_r)-\left(\sigma(X_u^\mu ,\mu_u)-\sigma(X_u^\nu,\nu_u)\right)\right|}
{(r - u)^{1+\alpha}}\d u\Bigg]\cr
=:& \frac{1}{\Gamma(1-\alpha)}
\left[\frac{\|\na\si\|_\infty\|X^\mu-X^\nu\|_{s,r,\infty}+\|D^L \si\|_\infty\W_{2,S,T,\be}(\mu,\nu)}{(r-s)^\alpha}+\alpha I_1 \right].
\end{align}

Next, we focus on dealing with the term $I_1$.
Owing to Lemma \ref{lem:metric1},
for each $u<r$ we may let $(Z_u^\mu, Z_r^\mu,Z_u^\nu,Z_r^\nu)$ be a random vector whose law is in $\sC_{opt}(\mu_{u,r},\nu_{u,r})$ such that
\begin{align*}
W_{2,u,r}^c(\mu,\nu)=\left(\E |Z_r^\mu-Z_u^\mu-Z_r^\nu+Z_u^\nu |^2\right)^{\ff 1 2}.
\end{align*}
Then by the condition of $\si$ again, we can deduce from Lemma \ref{FoLD} that
\begin{align*}
&\sigma(X_r^\mu,\mu_r)-\sigma(X_r^\nu,\nu_r)\cr
=&\int_0^1\ff \d {\d\th}\si(X_{r}^{\nu,\mu}(\th),\mu_r)\d\th+\int_0^1\ff\d {\d\th}\si(X_r^\nu,\sL_{Z_{r}^{\nu,\mu}(\th)})\d\th\cr
=&\int_0^1\na\si(\cdot,\mu_r)(X_{r}^{\nu,\mu}(\th))(X_r^\mu-X_r^\nu)\d\th\cr
&+\int_0^1 \left(\E\left\<D^L\si(x,\cdot)(\sL_{Z_{r}^{\nu,\mu}(\th)})(Z_{r}^{\nu,\mu}(\th)),Z_r^\mu-Z_r^\nu\right\>\right) |_{x=X_r^\nu}\d\th,
\end{align*}
where for any $\th\in(0,1), X_{r}^{\nu,\mu}(\th):=X_r^\nu+\th(X_r^\mu-X_r^\nu), Z_{r}^{\nu,\mu}(\th):=Z_r^\nu+\th(Z_r^\mu-Z_r^\nu)$, and
\begin{align*}
&\left\<D^L\si(x,\cdot)(\sL_{Z_{r}^{\nu,\mu}(\th)})(Z_{r}^{\nu,\mu}(\th)),Z_r^\mu-Z_r^\nu\right\>\\
=&\sum_{i,j=1}^d\left\<D^L\si_{ij}(x,\cdot)(\sL_{Z_{r}^{\nu,\mu}(\th)})(Z_{r}^{\nu,\mu}(\th)),Z_r^\mu-Z_r^\nu\right\>e_i\otimes e_j
\end{align*}
with $\si_{ij}=\<\si e_i,e_j\>$ and $\{e_j\}_{j=1}^d$ being any ONB on $\R^d$.\\
Combining this with the following inequality
\[|Aa-\tilde{A}\tilde{a}|\leq (|A|\wedge|\tilde{A}|) |a-\tilde{a}|+(|a|\vee|\tilde{a}|)|A-\tilde{A}|,\ \ a,\tilde{a}\in\R^d,A,\tilde{A}\in \R^d\otimes\R^d\otimes\R^d,\]
it follows that
\begin{align*}
&\left|\sigma(X_r^\mu,\mu_r)-\sigma(X_r^{\nu},\nu_r)-(\sigma(X_u^\mu,\mu_u)-\sigma(X_u^{\nu},\nu_u))\right|\cr
\leq&\left|\int_0^1 \na\si(\cdot,\mu_r)(X_{r}^{\nu,\mu}(\th))(X_r^\mu-X_r^\nu)-\na\si(\cdot,\mu_u)(X_{u}^{\nu.\mu}(\th))(X_u^\mu-X_u^\nu) \d\th\right|\cr
&+\left|\int_0^1 \bigg[\left(\E\left\<D^L\si(x,\cdot)(\sL_{Z_{r}^{\nu,\mu}(\th)})(Z_{r}^{\nu,\mu}(\th)),Z_r^\mu-Z_r^\nu\right\>\right)\big|_{x=X_r^\nu}\right.\cr
&\qquad\quad\left. - \left(\E\left\<D^L\si(y,\cdot)(\sL_{Z_{u}^{\nu,\mu}(\th)})(Z_{u}^{\nu,\mu}(\th)),Z_u^\mu-Z_u^\nu\right\>\right)\big|_{y=X_u^\nu}\bigg] \d\th\right|\cr
\leq&\int_0^1\left(\left| \na\si(\cdot,\mu_r)(X_{r}^{\nu,\mu}(\th))\right|\wedge\left|\na\si(\cdot,\mu_u)(X_u^{\nu,\mu}(\th))\right| \right)  \left|X_r^\mu-X_r^\nu-X_u^\mu+X_u^\nu\right|\d\th\cr
&+\int_0^1\left|\na\si(\cdot,\mu_r)(X_{r}^{\nu,\mu}(\th))-\na\si(\cdot,\mu_u)(X_{u}^{\nu,\mu}(\th))\right|\left( |X_r^\mu-X_r^\nu|\vee |X_u^\mu-X_u^\nu|\right)\d\th\cr
&+\int_0^1\bigg[\E\Big(\left(\left|D^L\si(x,\cdot)(\sL_{Z_{r}^{\nu,\mu}(\th)})(Z_{r}^{\nu,\mu}(\th))\right|\wedge \left|D^L\si(y,\cdot)(\sL_{Z_{u}^{\nu,\mu}(\th)})(Z_{u}^{\nu,\mu}(\th)) \right|\right)\\
&\qquad\quad \times \left|Z_r^\mu-Z_r^\nu-(Z_u^\mu-Z_u^\nu)\right|\Big)\bigg]\Big|_{x=X_r^\nu,y=X_u^\nu}\d\th\cr
&+\int_0^1\bigg[\E\Big(\left| D^L\si(x,\cdot)(\sL_{Z_{r}^{\nu,\mu}(\th)})(Z_{r}^{\nu,\mu}(\th))-D^L\si(y,\cdot)(\sL_{Z_{u}^{\nu,\mu}(\th)})(Z_{u}^{\nu,\mu}(\th))\right|\\
&\qquad\quad  \times \left(\left| Z_r^\mu-Z_r^\nu\right|\vee \left|Z_u^\mu-Z_u^\nu\right|\right)\Big)\bigg]\Big|_{x=X_r^\nu,y=X_u^\nu} \d\th.
\end{align*}
Then, using Lemma \ref{FoLD} and the condition of $\si$ again and recalling that $\sL_{(Z_u^\mu, Z_r^\mu,Z_u^\nu,Z_r^\nu)}\in\sC_{opt}(\mu_{u,r},\nu_{u,r})$ and the definition of $\mu_{u,r}^\Delta$, we obtain
\begin{align*}
&\left|\sigma(X_r^\mu,\mu_r)-\sigma(X_r^{\nu},\nu_r)-(\sigma(X_u^\mu,\mu_u)-\sigma(X_u^{\nu},\nu_u))\right|\cr
\leq&\|\na\si\|_\infty\|X^\mu-X^\nu\|_{u,r,\be}(r-u)^\be\cr
&+\left[\left(\ff 1 2 \|\na^2\si\|_\infty \left(|X^\nu_r-X^\nu_u|+|X^\mu_r-X^\mu_u|\right)+\|D^L\na\si\|_\infty W_2(\mu_r,\mu_u)\right)\wedge (2\|\na\si\|_\infty)\right]\|X^\mu-X^\nu\|_{u,r,\infty}\cr
&+\|D^L\si\|_\infty(\E|Z_r^\mu-Z_u^\mu-(Z_r^\nu-Z_u^\nu)|^2)^\ff 1 2\cr
&+\Bigg\{\bigg[\left(\|\na_1 D^L\si\|_\infty |X^\nu_r-X^\nu_u|\right)+\left(\ff 1 2\|D^{L,2}\si\|_\infty\left((\E|Z^\nu_r-Z^\nu_u|^2)^\ff 1 2+(\E|Z^\mu_r-Z^\mu_u|^2)^\ff 1 2\right)\right)\cr
&\qquad+\ff 1 2\|\na_2 D^L\si\|_\infty\left(\E|Z^\nu_r-Z^\nu_u|^2)^\ff 1 2+(\E|Z^\mu_r-Z^\mu_u|^2)^\ff 1 2\right)\bigg]\wedge (2\|D^L\si\|_\infty)\Bigg\}\W_{2}(\mu_{u,r},\nu_{u,r})\cr
\leq&\|\na\si\|_\infty\|X^\mu-X^\nu\|_{u,r,\be}(r-u)^\be\cr
&+\left[\left(\|\na^2\si\|_\infty \left( \|X^\nu\|_{u,r,\be}\vee\|X^\mu\|_{u,r,\be} \right)(r-u)^\be\right)\wedge(2\|\na\si\|_\infty)\right]\|X^\mu-X^\nu\|_{u,r,\infty}\cr
&+\left[\Bigg(\|D^L\na\si\|_\infty \ff {\sq{\mu_{u,r}^\Delta (|\cdot|^2)}} {(r-u)^\be}(r-u)^\be\Bigg)\wedge (2\|\na\si\|_\infty)\right]\|X^\mu-X^\nu\|_{u,r,\infty}\cr
&+\|D^L\si\|_\infty \ff {W_{2,u,r}^c(\mu,\nu) } {(r-u)^\be} (r-u)^\be\cr
&+\Bigg\{\left(\|\na_1 D^L\si\|_\infty \|X^\nu\|_{u,r,\be}(r-u)^\be\right)\wedge (2\|D^L\si\|_\infty)\\
&\quad+\Bigg[\left(\|D^{L,2}\si\|_\infty+\|\na_2 D^L\si\|_\infty\right)\ff {\sq{\nu_{u,r}^\Delta (|\cdot|^2)}+\sq{\mu_{u,r}^\Delta (|\cdot|^2)}} {2(r-u)^\be} (r-u)^{\be}\Bigg]\wedge(2\|D^L\si\|_\infty)\Bigg\}\W_{2}(\mu_{u,r},\nu_{u,r}).
\end{align*}
Taking into account the fact that $\|\mu\|_{2,S,T,\be}=\W_{2,S,T,\be}(\mu,\de_{\bf 0})$  and \eqref{AddDef},
and applying Lemma \ref{cor:Wc-0}(2) and \eqref{W2beT}, we have
\begin{align*}
&\left|\sigma(X_r^\mu,\mu_r)-\sigma(X_r^{\nu},\nu_r)-(\sigma(X_u^\mu,\mu_u)-\sigma(X_u^{\nu},\nu_u))\right|\cr
\leq&\|\na\si\|_\infty(r-u)^\be\|X^\mu-X^\nu\|_{u,r,\be}\cr
&+\Big[\left(\|\na^2\si\|_\infty \left( \|X^\nu\|_{u,r,\be}\vee\|X^\mu\|_{u,r,\be} \right)(r-u)^\be\right)\wedge(2\|\na\si\|_\infty)\cr
&\quad+\left(\|D^L\na\si\|_\infty \W_{2,S,T,\be}(\mu,\de_{\bf 0})(r-u)^\be\right)\wedge (2\|\na\si\|_\infty)\Big]\|X^\mu-X^\nu\|_{u,r,\infty}\cr
&+\Big[\|D^L\si\|_\infty(r-u)^\be +(\|\na_1 D^L\si\|_\infty \|X^\nu\|_{u,r,\be}(r-u)^\be)\wedge (2\|D^L\si\|_\infty)\\
&\quad+ \left((\|D^{L,2}\si\|_\infty+\|\na_2 D^L\si\|_\infty)(\W_{2,S,T,\be}(\nu,\de_{\bf 0})\vee\W_{2,S,T,\be}(\mu,\de_{\bf 0})) (r-u)^{\be}\right)\wedge(2\|D^L\si\|_\infty)\Big]\W_{2,S,T,\be}(\mu ,\nu ).
\end{align*}
Exchanging $\mu$ and $\nu$, one can see that
\beg{align*}
&\left|\sigma(X_r^\mu,\mu_r)-\sigma(X_r^{\nu},\nu_r)-(\sigma(X_u^\mu,\mu_u)-\sigma(X_u^{\nu},\nu_u))\right|\\
\leq& \|\na\si\|_\infty(r-u)^\be\|X^\mu-X^\nu\|_{u,r,\be}+\Xi_1(\si,r-u,X^\mu,X^\nu,\mu,\nu)\|X^\mu-X^\nu\|_{u,r,\infty}\\
& +\left(\|D^L\si\|_\infty(r-u)^\be+\Xi_2(\si,r-u,X^\mu,X^\nu,\mu,\nu)\right)\W_{2,S,T,\be}(\mu ,\nu),
\end{align*}
where
\beg{align*}
&\Xi_1(\si,r-u,X^\mu,X^\nu,\mu,\nu)\\
:=& \left(\|\na^2\si\|_\infty \left( \|X^\nu\|_{u,r,\be}\vee\|X^\mu\|_{u,r,\be} \right)(r-u)^\be\right)\wedge(2\|\na\si\|_\infty) \\
& +\left( \|D^L\na\si\|_\infty \left(\W_{2,S,T,\be}(\mu,\de_{\bf 0})\wedge \W_{2,S,T,\be}(\nu,\de_{\bf 0})\right)(r-u)^\be\right)\wedge(2\|\na\si\|_\infty)
\end{align*}
and
\beg{align*}
&\Xi_2(\si,r-u,X^\mu,X^\nu,\mu,\nu)\\
:=&\left(\|\na_1 D^L\si\|_\infty (\|X^\nu\|_{u,r,\be}\wedge\|X^\mu\|_{u,r,\be}) (r-u)^\be\right)\wedge (2\|D^L\si\|_\infty)\\
& + \left((\|D^{L,2}\si\|_\infty+\|\na_2 D^L\si\|_\infty)\left(\W_{2,S,T,\be}(\mu,\de_{\bf 0})\vee\W_{2,S,T,\be}(\nu,\de_{\bf 0})\right) (r-u)^{\be}\right)\wedge (2\|D^L\si\|_\infty).
\end{align*}
Consequently, applying Lemma \ref{TeLe1} to $\Xi_i(\si,r-u,X^\mu,X^\nu,\mu,\nu), i=1,2$, the above relation yields
\beg{align*}
I_1\leq& \int_s^r\frac{\|\na\si\|_\infty\|X^\mu-X^\nu\|_{u,r,\be}}
{(r - u)^{1+\alpha-\be}}\d u\cr
&+ \int_s^r\frac{\Xi_1(\si,r-u,X^\mu,X^\nu,\mu,\nu)\|X^\mu-X^\nu\|_{u,r,\infty}}
{(r - u)^{1+\alpha}}\d u\cr\cr
&+ \int_s^r\frac{\|D^L\si\|_\infty(r-u)^\be+\Xi_2(\si,r-u,X^\mu,X^\nu,\mu,\nu)}
{(r - u)^{1+\alpha}}\d u\cdot\W_{2,S,T,\be}(\mu ,\nu)\cr
\leq& \ff {\|\na\si\|_\infty} {\be-\al} (r-s)^{\be-\al}\|X^\mu-X^\nu\|_{s,r,\be}\\
& +\ff {2^{3-\ff {\al} {\be}}\be} {(\be-\al)\al}\|\na\si\|_\infty^{\ff {\be-\al} {\be}}\bigg[\left( \|\na^2\si\|_\infty \left( {\|X^\nu\|_{s,r,\be}\vee\|X^\mu\|_{s,r,\be}} \right)\right)^{\ff {\al} {\be}}\\
&\qquad\qquad\qquad \qquad+ \left(\|D^L\na\si\|_\infty  \left(\W_{2,S,T,\be}(\mu,\de_{\bf 0})\wedge \W_{2,S,T,\be}(\nu,\de_{\bf 0})\right)\right)^{\ff {\al} {\be}}\bigg]\|X^\mu-X^\nu\|_{s,r,\infty}\\
& +\Bigg\{\ff {\|D^L\si\|_\infty} {\be-\al} (r-s)^{\be-\al} +\ff {2^{3-\ff {\al} {\be}}\be} {(\be-\al)\al}\|D^L\si\|_\infty^{\ff {\be-\al} {\be}}\bigg[(\|\na_1 D^L\si\|_\infty (\|X^\nu\|_{s,r,\be}\wedge\|X^\mu\|_{s,r,\be}))^{\ff {\al} {\be}}\\
&  \qquad\quad+ \left((\|D^{L,2}\si\|_\infty+\|\na_2 D^L\si\|_\infty) (\W_{2,S,T,\be}(\mu,\de_{\bf 0})\vee\W_{2,S,T,\be}(\nu,\de_{\bf 0}) ) \right)^{\ff {\al} {\be}}\bigg]\Bigg\} \W_{2,S,T,\be}(\mu ,\nu).
\end{align*}

Now, substituting this into \eqref{2PfEsNoi} and going back to \eqref{3PfExU} and \eqref{1PfExU}, we conclude that
\begin{align*}
&\left|\int_s^t\left(\si(X_r^\mu,\mu_r)-\si(X_r^{\nu},\nu_r)\right)\d B^H_r\right|\cr
\le&\ff {C_0\mathcal{B}(1-\al,\al+\be_1)\|B^H\|_{s,t,\be_1}}{\Gamma(1-\al)}\left(\|\na\si\|_\infty\|X^\mu-X^\nu\|_{s,t,\infty}
+\|D^L \si\|_\infty\W_{2,S,T,\be}(\mu,\nu)\right)(t-s)^{\be_1}\cr
&+\ff {\al C_0\|\na\si\|_\infty \mathcal{B}(\be-\al+1,\al+\be_1)} {(\be-\al)\Gamma(1-\al)}\|B^H\|_{s,t,\be_1} (t-s)^{\be+\be_1}\|X^\mu-X^\nu\|_{s,t,\be}\cr
&+\ff {2^{3-\ff {\al} {\be}} \be C_0\|\na\si\|_\infty^{\ff {\be-\al} {\be}}\|B^H\|_{s,t,\be_1}} {(\be-\al)(\al+\be_1)\Gamma(1-\al)}\bigg[\left( \|\na^2\si\|_\infty \left( {\|X^\nu\|_{s,t,\be}\vee\|X^\mu\|_{s,t,\be}} \right)\right)^{\ff {\al} {\be}}\cr
&+\left( \|D^L\na\si\|_\infty \left(\W_{2,S,T,\be}(\mu,\de_{\bf 0})\wedge \W_{2,S,T,\be}(\nu,\de_{\bf 0})\right)\right)^{\ff {\al} {\be}} \bigg](t-s)^{\al+\be_1}\|X^\mu-X^\nu\|_{s,t,\infty}\cr
&+\ff {\al C_0\|B^H\|_{s,t,\be_1}} {\Gamma(1-\al)}\Bigg\{\ff {\|D^L\si\|_\infty \mathcal{B}(\be-\al+1,\be_1+\al)} {\be-\al} (t-s)^{\be+\be_1} \cr
&+\ff {2^{3-\ff {\al} {\be}} \be \|D^L\si\|_\infty^{\ff {\be-\al} {\be}}} {(\be-\al)(\al+\be_1)\al}\bigg[ (\|\na_1 D^L\si\|_\infty (\|X^\nu\|_{u,r,\be}\wedge\|X^\mu\|_{u,r,\be}))^{\ff {\al} {\be}}\cr
&+ \left((\|D^{L,2}\si\|_\infty+\|\na_2 D^L\si\|_\infty)(\W_{2,S,T,\be}(\mu,\de_{\bf 0})\vee\W_{2,S,T,\be}(\nu,\de_{\bf 0})) \right)^{\ff {\al}{\be}}  \bigg](t-s)^{\al+\be_1}\Bigg\} \W_{2,S,T,\be}(\mu ,\nu)\cr
\le& \|B^H\|_{s,t,\be_1} \left(\Lambda_1 (t-s)^{\be_1}+\Lambda_2  (t-s)^{\al+\be_1}\right)\|X^\mu-X^\nu\|_{s,t,\infty}\cr
&+\Lambda_3 \|B^H\|_{s,t,\be_1} (t-s)^{\be+\be_1}\|X^\mu-X^\nu\|_{s,t,\be}\cr
&+\|B^H\|_{s,t,\be_1} \left(\Lambda_4 (t-s)^{\be_1}+\Lambda_5 (t-s)^{\al+\be_1}\right)\W_{2,S,T,\be}(\mu ,\nu),
\end{align*}
where in the last inequality, we have used $t-s\leq 1$. This finishes the proof.
\end{proof}

With the help of Lemma \ref{EsNoi} above, we will provide an estimate of $\W_{2,S,T,\be}(\sL_{X^\mu},\sL_{X^\nu})$,
which is controlled by $\W_{2,S,T,\be}(\mu,\nu)$.
By using this estimate and choosing small $T$, we can prove that the mapping $\mu\mapsto \sL_{X^\mu}$ is contractive.

\beg{prp}\label{LeContr}
Let the same assumptions as for Lemma \ref{EsNoi} hold.
Then we have
\[\W_{2,S,T,\be}(\sL_{X^\mu},\sL_{X^\nu})\leq 6(1\wedge (T-S)^{\be_1-\be})\left\|e^{(1+\log 2) (T-S)\Delta^{-1}}\right\|_{L^2(\Om)}\W_{2,S,T,\be}(\mu,\nu),\]
where
\beg{align}\label{1LeContr}
\Delta:=&\left(\ff { 1  \wedge (T-S)^{\be_1}} {3(3\Lambda_1\vee \Lambda_3\vee3\Lambda_4 )\|B^H\|_{S,T,\be_1}}\right)^{\ff 1 {\be_1}}\wedge\left(\ff {1\wedge (T-S)^{\be_1}} {9 (\Lambda_2\vee \Lambda_5)\|B^H\|_{S,T,\be_1}}\right)^{\ff 1 {\al+\be_1}}\wedge\left(\ff {1\wedge (T-S)^{\be_1}} {9 (1\vee K_b)}\right)^{\ff 1 {\be_1}}
\end{align}
with $\Lambda_1,\cdots,\Lambda_5$ giving by \eqref{1-EsNoi}.
\end{prp}
\beg{proof}
Observe that, by \eqref{W2beT} we have
\beg{align}\label{0PfLeContr}
\W_{2,S,T,\be}(\sL_{X^\mu},\sL_{X^\nu})\leq \sq{\E\|X^\mu-X^\nu\|_{S,T,\infty}^2}+\sq{\E\|X^\mu-X^\nu\|_{S,T,\be}^2}.
\end{align}
While in view of \eqref{X-mu}, it is easy to see that for any $S\leq s<t\leq T$,
\beg{align}\label{3PfLeContr}
&X^\mu_t-X^\nu_t-(X^\mu_s-X^\nu_s)\cr
=&\int_s^t\left(b(X_r^\mu,\mu_r)-b(X_r^\nu,\nu_r)\right)\d r+\int_s^t\left(\si(X_r^\mu,\mu_r)-\si(X_r^{\nu},\nu_r)\right)\d B^H_r.
\end{align}
Below we shall apply Lemma \ref{EsNoi} to \eqref{3PfLeContr} to provide an estimate for \eqref{0PfLeContr}.

Taking into account the Lipschitz continuity of $b$, we first have
\begin{align*}
\left|\int_s^t(b(X_r^\mu,\mu_r)-b(X_r^\nu,\nu_r))\d r\right|&\le K_b\int_s^t(|X_r^\mu-X_r^\nu|+W_2(\mu_r,\nu_r))\d r\\
&\le K_b\left(\|X^\mu-X^\nu\|_{s,t,\infty}+\W_{2,S,T,\be}(\mu,\nu)\right)(t-s).
\end{align*}
Then, combining this with Lemma \ref{EsNoi} for $\be_1\in [\be,H)$, we get for $t-s\leq 1$
\begin{align}\label{1PfLeContr}
&\|X^\mu-X^\nu\|_{s,t,\be}\cr
\le& \,\Lambda_3\|B^H\|_{s,t,\be_1} (t-s)^{ \be_1}\|X^\mu-X^\nu\|_{s,t,\be}\cr
&+\left[\|B^H\|_{s,t,\be_1} (\Lambda_1(t-s)^{\be_1-\be}+\Lambda_2(t-s)^{\al+\be_1-\be})+K_b(t-s)^{1-\be}\right]\|X^\mu-X^\nu\|_{s,t,\infty}\cr
&+\left[\|B^H\|_{s,t,\be_1} (\Lambda_4(t-s)^{\be_1-\be}+\Lambda_5(t-s)^{\al+\be_1-\be})+K_b(t-s)^{1-\be}\right]\W_{2,S,T,\be}(\mu ,\nu).
\end{align}

Next, we put $\Delta:=t-s$ and take $\Delta$ as in \eqref{1LeContr} which implies
\begin{align}\label{5PfExU}
\left(\Lambda_3\Delta^{\be_1}\|B^H\|_{s,s+\Delta,\be_1}\right)\vee\left[\|B^H\|_{s,s+\Delta,\be_1} (\Lambda_1\Delta^{\be_1 }+\Lambda_2\Delta^{\al+\be_1 })+K_b\Delta^{\be_1} \right]&\cr
\vee\left[\|B^H\|_{s,s+\Delta,\be_1}(\Lambda_4\Delta^{ \be_1}+\Lambda_5\Delta^{\al+ \be_1})+K_b\Delta^{\be_1}\right] \vee \Delta^{\be_1}&\leq \ff   {1\wedge (T-S)^{\be_1}} 3.
\end{align}
Since $\Lambda_3\Delta^{\be_1}\|B^H\|_{s,s+\Delta,\be_1}<1$, by \eqref{1PfLeContr} we deduce
\begin{align}\label{4PfLeContr}
&\|X^\mu-X^\nu\|_{s,s+\Delta,\be}\cr
\le&\, \ff { \|B^H\|_{s,s+\Delta,\be_1} (\Lambda_1\Delta^{\be_1-\be}+\Lambda_2\Delta^{\al+\be_1-\be}) +K_b\Delta^{1-\be} } {1-\Lambda_3\Delta^{\be_1}\|B^H\|_{s,s+\Delta,\be_1}}\|X^\mu-X^\nu\|_{s,s+\Delta,\infty}\cr
& +\ff { \|B^H\|_{s,s+\Delta,\be_1}(\Lambda_4\Delta^{\be_1-\be}+\Lambda_5\Delta^{\al+\be_1-\be})  +K_b\Delta^{1-\be} } {1-\Lambda_3\Delta^{\be_1}\|B^H\|_{s,s+\Delta,\be_1} }  \W_{2,S,T,\be}(\mu ,\nu).
\end{align}
As a consequence, it follows that
\begin{align*}
&\|X^\mu-X^\nu\|_{s,s+\Delta,\infty}\cr
\le& |X^\mu_s-X^\nu_s|+\|X^\mu-X^\nu\|_{s,s+\Delta,\be}\Delta^\be\cr
\le& |X^\mu_s-X^\nu_s|+\ff { \|B^H\|_{s,s+\Delta,\be_1} (\Lambda_1\Delta^{\be_1 }+\Lambda_2\Delta^{\al+\be_1 }) +K_b\Delta } {1-\Lambda_3\Delta^{\be_1}\|B^H\|_{s,s+\Delta,\be_1}}\|X^\mu-X^\nu\|_{s,s+\Delta,\infty}\cr
&+\ff { \|B^H\|_{s,s+\Delta,\be_1}(\Lambda_4\Delta^{\be_1 }+\Lambda_5\Delta^{ \al+\be_1})  +K_b\Delta } {1-\Lambda_3\Delta^{\be_1}\|B^H\|_{s,s+\Delta,\be_1} }  \W_{2,S,T,\be}(\mu ,\nu)\\
\leq& |X^\mu_s-X^\nu_s|+\ff 1 2\|X^\mu-X^\nu\|_{s,s+\Delta,\infty}+\ff 1 2(1\wedge (T-S)^{\be_1})\W_{2,S,T,\be}(\mu ,\nu),
\end{align*}
where we use \eqref{5PfExU} and $\Delta<1$ in the last inequality.\\
Thus, we arrive at
\begin{align}\label{6PfExU}
\|X^\mu-X^\nu\|_{s,s+\Delta,\infty}\le 2|X^\mu_s-X^\nu_s|+ (1\wedge (T-S)^{\be_1})\W_{2,T,\be}(\mu,\nu).
\end{align}

Now, we let $m=[\ff {T-S} \Delta]+1$ and decompose the interval $[S,T]$ as follows:
\begin{equation*}
[S, T]=[S, S+\Delta]\cup[S+\Delta, S+2\Delta]\cup\cdots\cup[S+(m-1)\Delta, T].
\end{equation*}
Then applying \eqref{6PfExU} recursively and noting the fact that $X^\mu_S=X^\nu_S=X_S$, we derive that
\begin{align*}
\|X^\mu-X^\nu\|_{S,S+\Delta,\infty}\le (1\wedge (T-S)^{\be_1})\W_{2,S,T,\be}(\mu,\nu),
\end{align*}
and for each $1\le k\le m-1$,
\begin{align*}
\|X^\mu-X^\nu\|_{S+k\Delta,(S+(k+1)\Delta)\wedge T,\infty}
\le& 2|X^\mu_{S+k\Delta}-X^\nu_{S+k\Delta}|+ (1\wedge (T-S)^{\be_1})\W_{2,S,T,\be}(\mu,\nu)\cr
\le&2^k|X^\mu_{S+\Delta}-X^\nu_{S+\Delta}|+\sum_{i=0}^{k-1}2^{i} (1\wedge (T-S)^{\be_1})\W_{2,S,T,\be}(\mu,\nu)\cr
\le& \sum_{i=0}^{k} 2^{i}(1\wedge (T-S)^{\be_1})\W_{2,S,T,\be}(\mu,\nu)\cr
\le&2^{k+1}(1\wedge (T-S)^{\be_1})\W_{2,S,T,\be}(\mu,\nu).
\end{align*}
Therefore, we conclude that
\begin{align}\label{2PfLeContr}
\|X^\mu-X^\nu\|_{S,T,\infty}&=\max_{0\leq k\leq m-1}\left\| {{X^\mu} - {X^{\nu}}} \right\|_{S+k\Delta,[S+(k+1)\Delta]\wedge T,\infty}\cr
&\le2^{m}(1\wedge (T-S)^{\be_1})\W_{2,S,T,\be}(\mu,\nu)\cr
& \leq 2 (1\wedge (T-S)^{\be_1})e^{(T-S)\Delta^{-1}\log 2}\W_{2,S,T,\be}(\mu,\nu).
\end{align}

Finally, we are to estimate the term $\|X^\mu-X^\nu\|_{S,T,\be}$.
By \eqref{4PfLeContr}, \eqref{2PfLeContr} and \eqref{5PfExU}, we obtain
\beg{align*}
&\|X^\mu-X^\nu\|_{s,s+\Delta,\be}\cr
\le&\, 2\ff { \|B^H\|_{s,s+\Delta,\be_1} (\Lambda_1\Delta^{\be_1}+\Lambda_2\Delta^{\al+\be_1}) +K_b\Delta } {1-\Lambda_3\Delta^{\be_1}\|B^H\|_{S,T,\be_1}}  (1\wedge (T-S)^{\be_1}) \Delta^{-\be}e^{(T-S)\Delta^{-1}\log 2}\W_{2,S,T,\be}(\mu,\nu)\cr
& +\ff { \|B^H\|_{s,s+\Delta,\be_1}(\Lambda_4\Delta^{ \be_1}+\Lambda_5\Delta^{\al+ \be_1})  +K_b\Delta } {1-\Lambda_3\Delta^{\be_1}\|B^H\|_{S,T,\be_1} } \Delta^{-\be}  \W_{2,S,T,\be}(\mu ,\nu)\\
\le&\, 2\ff { \|B^H\|_{s,s+\Delta,\be_1} (\Lambda_1\Delta^{\be_1}+\Lambda_2\Delta^{\al+\be_1}) +K_b\Delta^{\be_1} } {1-\Lambda_3\Delta^{\be_1}\|B^H\|_{S,T,\be_1}}  (1\wedge (T-S)^{\be_1}) \Delta^{-\be}e^{(T-S)\Delta^{-1}\log 2}\W_{2,S,T,\be}(\mu,\nu)\cr
& +\ff { \|B^H\|_{s,s+\Delta,\be_1}( \Lambda_4\Delta^{ \be_1}+\Lambda_5\Delta^{\al+ \be_1})  +K_b\Delta^{\be_1} } {1-\Lambda_3\Delta^{\be_1}\|B^H\|_{S,T,\be_1} } \Delta^{-\be}  \W_{2,S,T,\be}(\mu ,\nu)\\
\leq &\, \Delta^{-\be}(1\wedge (T-S)^{\be_1}) \left( e^{(T-S)\Delta^{-1}\log 2}+\ff { 1} {2}\right)\W_{2,S,T,\be}(\mu,\nu).
\end{align*}
Here we have used $\Delta<1$ again. \\
Then by using Lemma \ref{TeLe2} and $\Delta<T-S$, we derive
\beg{align*}
\|X^\mu-X^\nu\|_{S,T,\be}&\leq \left(1+\left[\ff {T-S} {\Delta}\right]\right)^{1-\be}\Delta^{-\be}(1\wedge (T-S)^{\be_1})\left(  e^{(T-S)\Delta^{-1}\log 2}+\ff { 1} {2}\right)\W_{2,S,T,\be}(\mu,\nu)\\
& \leq (2(T-S))^{1-\be}\Delta^{-1}(1\wedge (T-S)^{\be_1})\left(  e^{(T-S)\Delta^{-1}\log 2}+\ff { 1} {2}\right)\W_{2,S,T,\be}(\mu,\nu)\\
&\leq  2(1\wedge (T-S)^{\be_1-\be}) \left((T-S)\Delta^{-1}\right)\left(  e^{(T-S)\Delta^{-1}\log 2}+\ff { 1} {2}\right)\W_{2,S,T,\be}(\mu,\nu),
\end{align*}
where the last inequality is due to the relation $(T-S)^{-\be}(1\wedge (T-S)^{\be_1})\leq 1\wedge (T-S)^{\be_1-\be}$.\\
Hence, substituting this and \eqref{2PfLeContr} into \eqref{0PfLeContr} and using the inequality $x\leq e^x$ for $x>0$,
we obtain
\beg{align*}
&\W_{2,S,T,\be}(\sL_{X^\mu},\sL_{X^\nu})\cr
\leq& 2(1\wedge (T-S)^{\be_1-\be}) \bigg[ \left\|e^{(T-S)\Delta^{-1}\log 2}\right\|_{L^2(\Om)}+\left\|(T-S)\Delta^{-1}\left(e^{(T-S)\Delta^{-1}\log 2}+1\right)\right\|_{L^2(\Om)}\bigg]\W_{2,S,T,\be}(\mu,\nu)\\
\leq & 6(1\wedge (T-S)^{\be_1-\be})\left\|e^{(1+\log 2) (T-S)\Delta^{-1}}\right\|_{L^2(\Om)}\W_{2,S,T,\be}(\mu,\nu).
\end{align*}
The proof is complete.
\end{proof}

Now, we are in the position to prove Theorem \ref{ExU}.

\emph{\textbf{Proof of Theorem \ref{ExU}.}}
Let $0\leq T_1\leq T$. Assuming that \eqref{EquM} is well-posed on $[0,T_1]$,  we want to prove that there is $R_1>0$ which is independent of $T_1$ such that \eqref{EquM} is well-posed on $[T_1,(T_1+R_1)\wedge T]$.

We first note that for any $\be\in (\ff 1 2,H)$, with the help of \eqref{XTbe-L0} there exists a positive constant $N$ such that
\beg{equation}\label{Seadd-new}
\|\sL_{X_{T_1}}\|_{2}\leq N(1+\|\sL_{X_0}\|_2),
\end{equation}
and in view of \eqref{exp+ep} there are positive constants $\ep_0'\in (0,\ep_0\wedge \ff {2(H^2+H-1)} {H})$, $\beta\in(\ff 1 2\vee(\ff {1-H} {H}+\ff {\ep_0'} 2),H)$ and $\mathcal{E}_{T,\be,X_0,\ep'_0}>0$ such that
\beg{equation}\label{add-new-2ine}
\E e^{|X_{T_1}|^{\ff {2(1-H)} {H}+\ep_0'}} \leq\mathcal{E}_{T,\be,X_0,\ep'_0}.
\end{equation}
Here the constants $N, \ep_0', \beta$ and $\mathcal{E}_{T,\be,X_0,\ep'_0}$ are all independent of $T_1$.

Due to {\bf (H)} and \cite[Theorem 2.1]{NR},  the frozen equation \eqref{X-mu} has a unique pathwise solution $X^\mu$ on $[T_1,T]$ with any  $\mu\in\sP_{2,\be}(W_{T_1,T}^d)$, and for any $\be'\in (\ff 1 2,H)$, $X^\mu$ belongs to $C^{\be'}([T_1,T];\R^d)$ and satisfies \eqref{in-mom} with $\be$ replaced by $\be'$. Consequently, there exists a mapping $\mu\mapsto \sL_{X^\mu}$ on $\sP_{2,\be}(W_{T_1,T}^d)$. Then with the help of Proposition \ref{lem-inva}, we know that there are positive constants $R_0$ and $M_0$, which are independent of $T_1$, such that for any $M\geq M_0$ and $T_2\leq(T_1+R_0)\wedge T$, the set $\sP_{M,\sL_{X_{T_1}},T_1,T_2,\be}$ is invariant for this mapping and \eqref{pri-XX0} holds on $[T_1,T_2]$.


Next, we want to apply Proposition \ref{LeContr} on the interval $[T_1,T_2]$ to show that the mapping $\mu\mapsto \sL_{X^\mu}$ is also contractive on $\sP_{M,\sL_{X_{T_1}},T_1,T_2,\be}$ for $\be$ closed to $H$ and $T_2$ closed to $T_1$.
For this, we first estimate  $\|e^{c(T_2-T_1)\Delta^{-1}}\|_{L^2(\Om)}$ for $\mu,\nu\in\sP_{M,\sL_{X_{T_1}},T_1,T_2,\be}$,
where $\Delta$ is given in Proposition \ref{LeContr} with $S,T$ replaced by $T_1,T_2$. Before going on, we point out how to choose proper $\al$ and $\be$.

Since $H\in (\ff {\sq 5-1} 2,1)$, we assert that for each $\ff {\sq 5-1} 2<\be<H$, there is  $\al\in(1-\be,\be)$ such that
\begin{equation}\label{eq-ab2}
\ff {\al} {\be^2(\al+\be)}+\ff 1 {\al+\be}<2.
\end{equation}
Indeed, for $\be\geq \ff {\sq 2} 2$, we find that $\ff {\al} {\be^2}+1  <2\al+1 <2\al+2\be$, which is consistent with \eqref{eq-ab2}; for $\ff {\sq 5-1} 2<\be<\ff {\sq 2} 2$, we find that $1-\be<\ff {1-2\be} {2-\be^{-2}}$, which yields that for all $\al\in(1-\be,\ff {1-2\be} {2-\be^{-2}})$ satisfying \eqref{eq-ab2}. On the other hand, it follows
\beg{align*}
\ff {2\al} {(2(\al+\be)-1)\be}>\ff {2(1-\be)} {(2(1-\be+\be)-1)\be}=\ff {2(1-\be)} {\be}>\ff {2(1-H)} {H},
\end{align*}
and if we take $\be=H$ and $\al=1-H$, there is
\[\ff {2\al} {(2(\al+\be)-1)\be}=\ff {2(1-H)} H.\]
Thus, for $\ep_0'$, we can choose $\al$ and $\be$ such that
\[\ff {2\al} {(2(\al+\be)-1)\be}<\ff {2(1-H)} H+\ep_0'\]
holds. We fix $\al$ and $\be$ chosen above. Then for any $\be_1\in (\be,H)$, we have
\begin{align}\label{eq-ab2-new}
\ff {\al} {\be^2(\al+\be_1)}&+\ff 1 {\al+\be_1}<2,\\
\ff {2\al} {(2(\al+\be_1)-1)\be}&<\ff {2(1-H)} H+\ep_0'.\label{2PfMR}
\end{align}

Fix $\be_1>\be$ and $M\geq M_0$.  Now, we estimate  $\|e^{(1+\log 2)(T_2-T_1)\Delta^{-1}}\|_{L^2(\Om)}$.
According to \eqref{1LeContr} and \eqref{1-EsNoi}, we obtain
\beg{align}\label{1PfMR}
&(T_2-T_1)\Delta^{-1}\cr
=& (T_2-T_1)\left[\ff  {\left(3(3\Lambda_1\vee \Lambda_3\vee3\Lambda_4 )\|B^H\|_{T_1,T_2,\be_1}\right)^{\ff 1 {\be_1}}} {1\wedge (T_2-T_1)}\right.\cr
&\qquad\qquad \left.\vee\left(\ff {9 (\Lambda_2\vee \Lambda_5)\|B^H\|_{T_1,T_2,\be_1}} {1\wedge (T_2-T_1)^{\be_1}} \right)^{\ff 1 {\al+\be_1}}\vee \left(\ff {9 (1\vee K_b)}  {1\wedge (T_2-T_1)^{\be_1}}\right)^{\ff 1 {\be_1}}\right]\cr
\leq& C_{\al,\be,\be_1,\si}(1\vee (T_2-T_1))\times\Bigg[1\vee\|B^H\|_{T_1,T_2,\be_1}^{\ff 1 {\be_1}}\cr
&\ \ \ \vee \left(\left(\|X^\nu\|_{T_1,T_2,\be}\vee \|X^\mu\|_{T_1,T_2,\be}+M(1+\|\sL_{X_{T_1}}\|_2)\right)^{\ff {\al} {\be(\al+\be_1 )}}\|B^H\|_{T_1,T_2,\be_1}^{\ff 1 {\al+\be_1}}\right)\Bigg]\cr
\leq& C_{\al,\be,\be_1,\si}(1\vee (T_2-T_1))\times\Bigg[1\vee\|B^H\|_{T_1,T_2,\be_1}^{\ff 1 {\be_1}}\cr
&\ \ \ \vee \left(\left(\|X^\nu\|_{T_1,T_2,\be}\vee \|X^\mu\|_{T_1,T_2,\be}+M(N+1)(1+\|\sL_{X_{0}}\|_2)\right)^{\ff {\al} {\be(\al+\be_1 )}}\|B^H\|_{T_1,T_2,\be_1}^{\ff 1 {\al+\be_1}}\right)\Bigg]\cr
\leq& C_{\al,\be,\be_1,\si,T_2-T_1,M,N,\|\sL_{X_0}\|_2} \Bigg[1+\|B^H\|_{T_1,T_2,\be_1}^{\ff 1 {\be_1}}+\|B^H\|_{T_1,T_2,\be_1}^{\ff 1 {\al+\be_1}}\cr
&\ \ \ + \left(\|X^\nu\|_{T_1,T_2,\be}\vee \|X^\mu\|_{T_1,T_2,\be}\right)^{\ff {\al} {\be(\al+\be_1 )}}\|B^H\|_{T_1,T_2,\be_1}^{\ff 1 {\al+\be_1}}\Bigg],
\end{align}
where we use $\mu,\nu\in\sP_{M,\sL_{X_{T_1}},T_1,T_2,\be}$ in the first inequality and \eqref{Seadd-new} in the second inequality.\\
Due to the Fernique theorem and
the fact that $\ff 1 {\al+\be_1}<\ff 1 {\be_1}<2$, we have for any $c>0$,
\beg{equation}\label{Add-6PfMR}
\E \exp\left\{c\left(\|B^H\|_{T_1,T_2,\be_1}^{\ff 1 {\be_1}}+\|B^H\|_{T_1,T_2,\be_1}^{\ff 1 {\al+\be_1}}\right)\right\}<+\infty,
\end{equation}
and the expectation value depends only on $\al,\be_1,T_2-T_1$ and $c$.\\

Next, we intend to estimate
\beg{equation}\label{eXXB}
\E\exp\left\{c\left(\|X^\nu\|_{T_1,T_2,\be}\vee \|X^\mu\|_{T_1,T_2,\be}\right)^{\ff {\al} {\be(\al+\be_1)}}\|B^H\|_{T_1,T_2,\be_1}^{\ff 1 {\al+\be_1}}\right\},
\end{equation}
where $c$ is a constant depending only on $M,N,\al,\be,\be_1,\si$ and $\|\sL_{X_0}\|_{2}$.
Using {\bf (H)} and Lemma \ref{mom-est} with
\[K_{\tilde{b}}=K_b\|\mu\|_{2,T_1,T_2,\be},\quad L_{\tilde{b}}=K_b\vee|b({\bf 0},{\bf\de_0})|,\quad K_{\tilde{\sigma}}=\|D^L\si\|_\infty\|\mu\|_{2,T_1,T_2,\be},\quad L_{\tilde{\si}}=\|\na\si\|_\infty,\quad \ga_0=\be,\]
for the term $X^\mu$, we can write $G(T_2-T_1,K_{\tilde{b}},K_{\tilde{\si}},B^H)$ in \eqref{Fan-add4} as follows:
\beg{align*}
&G(T_2-T_1,K_b\|\mu\|_{2,T_1,T_2,\be},\|D^L\si\|_\infty\|\mu\|_{2,T_1,T_2,\be},B^H)\\
=&C_{K_b,\|D^L\si\|_\infty,\be}\Bigg\{\left((1\vee (T_2-T_1))\|B^H\|_{T_1,T_2,\be}^{\ff 1 {\be}}\right) \\
&\qquad\qquad\qquad+\left[\left(\left(1+\|D^L\si\|_\infty\|\mu\|_{2,T_1,T_2,\be}\right)\|B^H\|_{T_1,T_2,\be} \right)^{\ff 1 {2\be }}\left((T_2-T_1)\vee (T_2-T_1)^{\ff {1} {2 }}\right)\right]\\
&\qquad\qquad\qquad+\left[\left(1+K_b\|\mu\|_{2,T_1,T_2,\be}\right)\left((T_2-T_1)\vee (T_2-T_1)^{1-\be}\right)\right]\Bigg\}\\
\leq& C_{T_2-T_1,M,\|\sL_{X_0}\|_2,K_b,\|D^L\si\|_\infty,\be}\left(1+\|B^H\|_{T_1,T_2,\be}^{\ff 1 {\be}}\right),
\end{align*}
where the last inequality is due to the relation
\beg{align*}
\|\mu\|_{2,T_1,T_2,\be}&\leq M(1+\|\mu_{T_1}\|_2)=M(1+\|\sL_{X_{T_1}}\|_2)\\
&\leq  M(1+N\|\sL_{X_{T_1}}\|_2)\leq M(N+1)(1+\|\sL_{X_0}\|_2)
\end{align*}
For simplicity, we denote $G=G(T_2-T_1,K_b\|\mu\|_{2,T_1,T_2,\be},\|D^L\si\|_\infty\|\mu\|_{2,T_1,T_2,\be},B^H)$.
Then, it follows from \eqref{ad-1be}, \eqref{ad-supn} and the H\"older inequality that there exists a positive constant $\ti C$  depending on $T_2-T_1$, $N,M$, $\|\sL_{X_0}\|_2$, $K_b$, $\|D^L\si\|_\infty$, $\be$ and being locally bounded on $T_2-T_1$ such that
\beg{align*}
&\|X^\mu\|_{T_1,T_2,\be}\cr
\leq& \left(1+G \right)^{1-\be} G^{\be}+C\left((T_2-T_1)^{\be}\vee (T_2-T_1)  \right)^{1-\be} \exp\left\{ 2C(T_2-T_1)+1\wedge (T_2-T_1)^{\be}\right\}  |X_{T_1}|\cr
&\quad  +C\left((T_2-T_1)^{\be}\vee (T_2-T_1)  \right)^{1-\be} (1\wedge (T_2-T_1)^{\be})e^{2C(T_2-T_1)}\left(1+G\right)\\
\leq& \ti C\left(1+\|B^H\|_{T_1,T_2,\be}^{\ff 1 {\be}}+|X_{T_1}|\right).
\end{align*}
This, together with the $C_r$-inequality and  the fact that
\[\|B^H\|_{T_1,T_2,\be}\leq (T_2-T_1)^{\be_1-\be}\|B^H\|_{T_1,T_2,\be_1},\]
implies that
\beg{align}\label{XXB}
&\left(\|X^\nu\|_{T_1,T_2,\be}\vee \|X^\mu\|_{T_1,T_2,\be}\right)^{\ff {\al} {\be(\al+\be_1)}}\|B^H\|_{T_1,T_2,\be_1}^{\ff 1 {\al+\be_1}}\cr
\leq &\ti C\left(1+\|B^H\|_{T_1,T_2,\be}^{\ff 1 {\be}}+|X_{T_1}|\right)^{\ff {\al} {\be(\al+\be_1)}}\|B^H\|_{T_1,T_2,\be_1}^{\ff 1 {\al+\be_1}}\cr
\leq &\ti C\left(1+\|B^H\|_{T_1,T_2,\be}^{\ff {\al} {\be^2(\al+\be_1)}}+|X_{T_1}|^{\ff {\al} {\be(\al+\be_1)}}\right)\|B^H\|_{T_1,T_2,\be_1}^{\ff 1 {\al+\be_1}}\cr
\leq &\ti C\left(\|B^H\|_{T_1,T_2,\be_1}^{\ff 1 {\al+\be_1}}+\|B^H\|_{T_1,T_2,\be_1}^{\ff {\al} {\be^2(\al+\be_1)}+\ff 1 {\al+\be_1}}+|X_{T_1}|^{\ff {\al} {\be(\al+\be_1)}} \|B^H\|_{T_1,T_2,\be_1}^{\ff 1 {\al+\be_1}}\right).
\end{align}
Note that again thanks to the Fernique theorem and \eqref{eq-ab2-new}, we get for any $c>0$,
\beg{equation}\label{EBB}
\E\exp\left\{c\left(\|B^H\|_{T_1,T_2,\be_1}^{\ff 1 {\al+\be_1}}+\|B^H\|_{T_1,T_2,\be_1}^{\ff {\al} {\be^2(\al+\be_1)}+\ff 1 {\al+\be_1}}\right)\right\}<+\infty,
\end{equation}
and the expectation value depends only on $T_2-T_1,\al,\be_1,\be$ and $c$. \\
In addition, by the Young inequality, we have for any $\ep>0$,
\beg{align}\label{3PfMR}
c|X_{T_1}|^{\ff {\al} {\be(\al+\be_1)}}\|B^H\|_{T_1,T_2,\be_1}^{\ff 1 {\al+\be_1}}\leq  \ff {\left(2(\al+\be_1)-1\right)c^{\ff {2(\al+\be_1)} {2(\al+\be_1)-1}}} {\ep^{\ff 1 {2(\al+\be_1)-1}}(2(\al+\be_1))^{\ff {2(\al+\be_1)} {2(\al+\be_1)-1}}}|X_{T_1}|^{\ff {2\al} {(2(\al+\be_1)-1)\be}}+\ep \|B^H\|_{T_1,T_2,\be_1}^2.
\end{align}
It follows from  \eqref{2PfMR} and  the H\"older inequality that there is a positive constant $C_{\al,\be_1,\ep,\ep_0',H,c}$ such that
\beg{align*}
\ff {\left(2(\al+\be_1)-1\right)c^{\ff {2(\al+\be_1)} {2(\al+\be_1)-1}}} {\ep^{\ff 1 {2(\al+\be_1)-1}}(2(\al+\be_1))^{\ff {2(\al+\be_1)} {2(\al+\be_1)-1}}}|X_{T_1}|^{\ff {2\al} {(2(\al+\be_1)-1)\be}}\leq \ff 1 2|X_{T_1}|^{\ff {2(1-H)} {H}+\ep_0'}+C_{\al,\be_1,\ep,\ep_0',H,c}.
\end{align*}
Then, resorting to this and \eqref{3PfMR} and using the H\"older inequality, \eqref{add-new-2ine} and the Fernique theorem again, we  can take $\ep$ small enough such that
\beg{align}\label{exp-XXB}
&\E\exp\left\{c|X_{T_1}|^{\ff {\al} {\be(\al+\be_1)}} \|B^H\|_{T_1,T_2,\be_1}^{\ff 1 {\al+\be_1}}\right\}\cr
\leq& \E\exp\left\{\ff 1 2|X_{T_1}|^{\ff {2(1-H)} {H}+\ep_0'}+C_{\al,\be_1,\ep,\ep_0',H,c}+\ep  c\|B^H\|_{T_1,T_2,\be_1}^2\right\}\cr
\leq& \sq{\mathcal{E}}\exp\left(C_{\al,\be_1,\ep,\ep_0',H,c}\right)\left(\E\exp\left\{2\ep c\|B^H\|_{T_1,T_2,\be_1}^2\right\}\right)^{\ff 1 2}<+\infty.
\end{align}
where we denote $\mathcal{E}=\mathcal{E}_{T,\be,X_0,\ep'_0}$. Combining this with \eqref{EBB} and \eqref{XXB}, we see that \eqref{eXXB} is finite and the expectation value depends only on $\mathcal{E}$, $T_2-T_1$, $N, M$, $\|\sL_{X_0}\|_2$, $K_b$, $\|D^L\si\|_\infty$, $\be$, $\be_1$, $\|\sL_{X_0}\|_{2}$ and $c$.
Consequently, by \eqref{1PfMR} and \eqref{Add-6PfMR} we see that $\|e^{(1+\log 2)(T_2-T_1)\Delta^{-1}}\|_{L^2(\Om)}$ is finite and the value depends only on $\mathcal{E}$, $T_2-T_1$, $N, M$, $\|\sL_{X_0}\|_2$, $K_b$, $\|D^L\si\|_\infty$, $\be, \be_1$, $\|\sL_{X_0}\|_{2}$.

Since $\|e^{(1+\log2)(T_2-T_1)\Delta^{-1}}\|_{L^2(\Om)}$ is bounded by a constant depending on  $T_2-T_1$. So, we obtain
\[\lim_{T_2-T_1\ra 0^+}(1\wedge (T_2-T_1))^{\be_1-\be}\left\|e^{(1+\log 2) (T_2-T_1)\Delta^{-1}}\right\|_{L^2(\Om)}=0.\]
Then according to Proposition \ref{LeContr}, this implies that there exists a positive constant $R_1\leq R_0$ depending on $\mathcal{E}$, $N,M$, $\|\sL_{X_0}\|_2$, $K_b$, $\|D^L\si\|_\infty$, $\be$, $\be_1$, $\|\sL_{X_0}\|_{2}$ such that the mapping $\mu\mapsto\sL_{X^\mu}$ is contractive for any $T_2=T_1+ R_1$, i.e., there is a constant $c\in(0,1)$ which is independent of $T_1$ such that
\beg{equation}\label{W2Tbe-con}
\W_{2,T_1,T_2,\be}(\sL_{X^\mu},\sL_{X^\nu})\leq c\W_{2,T_1,T_2,\be}(\mu,\nu),\ \ \mu,\nu\in\sP_{M,\sL_{X_{T_1}},T_1,T_2,\be}.
\end{equation}
It follows from Theorem \ref{ThProb} and the contraction mapping principle on $\sP_{M,\sL_{X_{T_1}},T_1,T_2,\be}$ that the mapping $\mu\mapsto\sL_{X^\mu}$
has a unique fixed point, saying $\mu$. Thus the solution of equation \eqref{X-mu} satisfies $\sL_{X^\mu}=\mu$, and then equation \eqref{EquM} has a solution.
By using \eqref{pri-XX0} and $R_1\leq R_0$, any solution of equation \eqref{EquM} on $[T_1,T_2]$, saying $\{X_t\}_{t\in [T_1,T_2]}$, satisfies  $\sL_{X}\in \sP_{M,\sL_{X_{T_1}},T_1,T_2,\be}$. Then there holds $\sL_{X}=\mu$. Since equation \eqref{X-mu}  has a pathwise unique solution, we conclude that $X_t=X^\mu_t$, and equation \eqref{EquM} has a unique strong solution on $[T_1,T_2]$.

Finally, we can establish the existence and uniqueness for the solution of \eqref{EquM} on $[kR_1,(k+1)R_1\wedge T]$. Therefore, we prove the well-posedness of \eqref{EquM} for any $T>0$.

\qed

\section{Large and moderate deviation principles}

In this section, we aim to investigate the asymptotic behaviors for DDSDEs with multiplicative fractional noises.
Section 4.1 serves as a recap of the weak convergence criteria in the fBM setting established in our previous work \cite{FYY}.
In Sections 4.2 and 4.3, we will show respectively how to combine this criteria to establish large and moderate deviation principles for DDSDEs with multiplicative fractional noises.

\subsection{Weak convergence criteria for fBM}

Recall first some definitions of the theory of large deviation principle (LDP).
Let $\mathscr{E}$ be a Polish space with the Borel $\sigma$-field $\sB(\mathscr{E})$.

\beg{defn}\label{De(rate)}
(Rate Function)
A function $I$ is called a rate function, if for each constant $M<+\infty$, the level set $\{x\in\mathscr{E}: I(x)\leq M\}$ is a compact subset of $\mathscr{E}$.
\end{defn}

\beg{defn}\label{De(ldp)}
(Large Deviation Principle)
Let $I$ be a rate function on $\mathscr{E}$.
Given a collection $\{\ell(\ep)\}_{\ep>0}$ of positive reals,
a family $\{\mathbb{X}^\ep\}_{\ep>0}$ of $\mathscr{E}$-valued random variables is said to be satisfied a LDP on $\mathscr{E}$ with speed $\ell(\ep)$ and rate function $I$
if the following two conditions hold:
\begin{itemize}
\item[(i)](Upper bound) For each closed subset $F\subset\mathscr{E}$,
\begin{align*}
\limsup_{\ep\to 0}\ell(\ep)\log{\mathbb{P}(\mathbb{X}^\ep\in F)} \leq -\inf_{x\in F}I(x).
\end{align*}

\item[(ii)](Lower bound) For each open subset $G\subset\mathscr{E}$,
\begin{align*}
\liminf_{\ep\to 0}\ell(\ep)\log{\mathbb{P}(\mathbb{X}^\ep\in G)} \geq -\inf_{x\in G}I(x).
\end{align*}
\end{itemize}
\end{defn}
For any $\ep>0$, let $\G^\ep: C([0,T]; \R^d)\ra\mathscr{E}$ be a measurable map.
Next, we will present a weak convergence criteria for the LDP of $\mathbb{X}^\ep:=\G^\ep(\ep^H B_\cdot^H)$ as $\ep$ goes to $0$.
To this end, we need to give some notations.
Let
\begin{align*}
\A=\left\{\phi:\phi\  \mathrm{is}\ \R^d\textit{-}\mathrm{valued}\ \sF_t\textit{-}\mathrm{predictable\ process\ and \ \|\phi\|_\H^2<+\infty\ \P\textit{-}\mathrm{a.s.} } \right\},
\end{align*}
and for any $M>0$, let
\begin{align*}
S_M=\left\{h\in\H: \ff 1 2\|h\|_\H^2\leq M\right\}.
\end{align*}
It is readily checked that $S_M$ endowed with the weak topology is a Polish space.
We also set
\begin{align*}
\mathcal{A}_M:=\{\phi\in\mathcal{A}: \phi(\omega)\in S_M, \ \mathbb{P}\textit{-}\mathrm{a.s.}\}.
\end{align*}

Now, we introduce the following sufficient condition for the LDP of $\mathbb{X}^\ep$ when $\ep$ tends to $0$.
\begin{enumerate}

\item[\textsc{\textbf{(A)}}] There exists a measurable map  $\G^0: I_{0+}^{H+1/2}(L^2([0,T],\R^d))\ra\mathscr{E}$ for which the following two conditions hold.

\item[(i)] Let $\{h^\ep:\ep>0\}\subset\mathcal{A}_M$ for any $M\in(0,+\infty)$. For each $\delta>0$,
\begin{align*}
\lim_{\ep\ra 0}\P\left(d\left(\G^\ep(\ep^H B_\cdot^H+\ep^H/\ell^{\ff 1 2}(\ep)(R_Hh^\ep)(\cdot)),\G^0((R_Hh^\ep)(\cdot))\right)>\delta\right)=0,
\end{align*}
where $d(\cdot,\cdot)$ stands for the metric on $ \mathscr{E}, \{\ell^\ep\}_{\ep>0}$ are positive reals.

\item[(ii)]  Let $\{h^n:n\in\mathbb{N}\}\subset S_M$ for any $M\in(0,+\infty)$. If $h^n$ converges to some element $h$ in $S_M$ as $n\ra+\infty$, then $\G^0(R_Hh^n)$ converges to $\G^0(R_Hh)$ in $\mathscr{E}$.

\end{enumerate}

\beg{prp}\label{Suf2(LDP)} \cite[Proposition 3.5]{FYY}
If $\mathbb{X}_\cdot^\epsilon=\mathcal{G}^\ep(\ep^H B_\cdot^H)$ and \textsc{\textbf{(A)}} holds, then the family $\{\mathbb{X}^\ep:\ep>0\}$ satisfies the LDP on $\mathscr{E}$ with speed $\ell(\ep)$
and the rate function $I$ given by
\begin{align}\label{1-Suf1(LDP)}
I(f)=\inf_{\{h\in\H:f=\G^0(R_Hh)\}}\left\{\ff 1 2 \|h\|_{\H}^2\right\},\ \ f\in \mathscr{E}.
\end{align}
Here and in the sequel, we follow the convention that the infimum over an empty set is $+\infty$.
\end{prp}

\subsection{Large deviation principle (LDP)}

In this part, we consider the following DDSDE with small multiplicative fractional noise: for every $\ep>0$ and $T>0$,
\begin{align}\label{LDP-DDsde}
\d X_t^\epsilon =b(X_t^\ep,\mathscr{L}_{X_t^\ep})\d t+\ep^H\sigma(X_t^\ep,\mathscr{L}_{X_t^\ep})\d B_t^H,\ \ X_0^\ep =x\in\R^d, \  t\in[0,T],
\end{align}
where $H\in (\ff {\sq 5-1} 2,1)$.
According to Theorem \ref{ExU},  under {\bf (H)} equation \eqref{LDP-DDsde} has a unique solution on $[0,T]$.

In order to apply Proposition \ref{Suf2(LDP)} above, we choose $\mathscr{E}:=C([0,T]; \R^d)$.
For each fixed $\mu\in\sP_{2,\be}(W_T^d)$ with $\be\in(1-H,1)$, we consider the following reference equation:
\beg{align}\label{FrozeEQ}
\d\Upsilon_t=b(\Upsilon_t,\mu_t)\d t+\sigma(\Upsilon_t,\mu_t)\d B_t^H,\ \ \Upsilon_0=y\in\R^d,\ 0\leq t\leq T.
\end{align}
Owing to \cite[Theorem 2.1]{NR}, {\bf (H)} implies that equation \eqref{FrozeEQ} admits a unique solution.
Then, there is a measurable map
\begin{align}\label{MeasMap}
\G_\mu: C([0,T]; \R^d)\rightarrow C([0,T]; \R^d)
\end{align}
such that $\Upsilon_\cdot=\mathcal{G}_\mu(B^H_\cdot).$
Moreover, for any $h\in\mathcal{A}_M$, let $\Upsilon_\cdot^h=\mathcal{G}_\mu(B_\cdot^H+(R_Hh)(\cdot))$, then $\Upsilon^h$ fulfills the following equation
\begin{align}\label{1-Le(Meas)}
\Upsilon_t^h=&y+\int_0^tb(\Upsilon_s^h, \mu_s)\d s+\int_0^t\si(\Upsilon_s^h, \mu_s)\d(R_Hh)(s)
+\int_0^t\si(\Upsilon_s^h, \mu_s)\d B_s^H,\ \ t\in[0,T],\  \P\textit{-}a.s..
\end{align}
Consequently, we immediately get the following result.

\beg{lem}\label{Le(Meas)}
Assume that \textsc{\textbf{(H)}} holds, and let $X$ be a solution of equation \eqref{EquM} with initial value $X_0=x$.
Assume moreover that $y=x$ and $\mu_t=\sL_{X_t}, t\in[0,T]$ for equation \eqref{FrozeEQ}.
Then, there holds $X_\cdot=\mathcal{G}_{\sL_{X}}(B^H_\cdot)$, where $\mathcal{G}_{\sL_{X}}$ is shown in \eqref{MeasMap} with $\mu=\sL_{X}$.
Furthermore, for each $h\in\mathcal{A}_M$, put $X_\cdot^h:=\mathcal{G}_{\sL_X}\left(B_\cdot^H+(R_Hh)(\cdot)\right)$,
then $X^h$ solves the following equation
\begin{align*}
X_t^h=&x+\int_0^tb(X_s^h, \sL_{X_s})\d s+\int_0^t\sigma(X_s^h, \sL_{X_s})\d(R_Hh)(s)+\int_0^t\sigma(X_s^h,\sL_{X_s})\d B_s^H,\ \ t\in[0,T],\ \P\textit{-}a.s..
\end{align*}
\end{lem}

Therefore, for equation \eqref{LDP-DDsde}, there exists a measurable map $\mathcal{G}^\epsilon:= \mathcal{G}_{\mathscr{L}_{X^\epsilon}}$ such that $X^\epsilon_\cdot=\mathcal{G}^\epsilon(\ep^H B^H_\cdot).$
Moreover, for each $h^\ep\in\mathcal{A}_M$, set
\begin{align}\label{Per2-DDsde}
X_\cdot^{\ep, h^\ep}:=\mathcal{G}^\ep\left(\ep^H B_\cdot^H+(R_Hh^\ep)(\cdot)\right),
\end{align}
then $X^{\ep, h^\ep}$ satisfies the following equation
\begin{align}\label{Per-DDsde}
X_t^{\ep, h^\ep}=&x+\int_0^tb(X_s^{\ep, h^\ep}, \sL_{X_s^\ep})\d s+\int_0^t\sigma(X_s^{\ep, h^\ep},\sL_{X_s^\ep})\d(R_Hh^\ep)(s)\cr
&+\ep^H\int_0^t\sigma(X_s^{\ep, h^\ep},\sL_{X_s^\ep})\d B_s^H,\ \ t\in[0,T],\ \ \P\textit{-}a.s..
\end{align}
In addition, by Theorem \ref{ExU} again, it is easy to obtain the following lemma.
 \beg{lem}\label{ODE}
Assume that \textsc{\textbf{(H)}} holds.
Then there exists a unique function $\{X_t^0\}_{t\in[0,T]}$ such that
$X^0$ belongs to $C^\beta([0,T];\R^d)$  with $\be\in(0,1]$ and satisfies the following deterministic equation
\begin{align}\label{LimSDE}
X_t^0=x+\int_0^t b(X_s^0, \sL_{X_s^0})\d s, \ \ t\in[0,T].
\end{align}
\end{lem}
Let us stress that $X^0$ is a deterministic path and then its distribution is $\sL_{X_s^0}=\de_{X_s^0}$.
In the sequel, we will always denote $X^0$ by the unique solution of equation \eqref{LimSDE}.

Now, we are in the position to introduce the following skeleton equation
\begin{align}\label{SkE}
Z_t^h=x+\int_0^tb(Z_s^h,\sL_{X_s^0})\d s+\int_0^t\sigma(Z_s^h,\sL_{X_s^0})\d(R_H h)(s),\ \ t\in[0,T].
\end{align}
Here, $h\in\H, X^0$ is given in \eqref{LimSDE}, and $\int_0^t\sigma(Z_s^h,\sL_{X_s^0})\d(R_H h)(s)$ is interpreted as a Riemann-Stieltjes integral.
For any $s\in[0,T]$ and $x\in\R^d$, set
\begin{align*}
\bar{b}(s,x):=b(x,\sL_{X_s^0}),\ \ \ \ \ \bar{\si}(s,x):=\sigma(x,\sL_{X_s^0}).
\end{align*}
Using \textsc{\textbf{(H)}} and \cite[Theorem 5.1]{NR}, we know that equation \eqref{SkE} admits a unique solution,
which then allows to define a map as follows
\begin{align}\label{RateF}
\G^0: I_{0+}^{H+1/2}(L^2([0,T],\R^d))\ni R_Hh\mapsto Z^h\in C([0,T]; \R^d).
\end{align}
We would like to mention that when $\ep$ goes to $0$, equation \eqref{LDP-DDsde} reduces to \eqref{LimSDE} and then $\sL_{X_\cdot^\ep}$ tends to $\sL_{X_\cdot^0}$.
Hence, using $\sL_{X_\cdot^0}$, instead of $\sL_{\Upsilon_\cdot^h}$, in the skeleton equation to define the rate function of Theorem \ref{Th(LDP)} below is reasonable.

Below is our main result which states the LDP for equation \eqref{LDP-DDsde}.
\beg{thm}\label{Th(LDP)}
Assume that \textsc{\textbf{(H)}} holds. For each $\ep>0$, let $X^\ep=\{X_t^\ep\}_{t\in[0,T]}$ be the solution to  equation \eqref{LDP-DDsde}.
Then the family $\{X^\ep:\ep>0\}$ satisfies the LDP on $C([0,T]; \R^d)$ with speed $\ep^{2H}$,
in which the rate function $I$ is given by \eqref{1-Suf1(LDP)} with $\G^0$ being defined in \eqref{RateF}.
\end{thm}

In order to prove the theorem, we begin with the following useful lemma.
\beg{lem}\label{est-Rh}
Let $A$ be a $\R^m\otimes\R^d$-valued function such that for any ONB in $\R^m$, saying $\{e_j\}_{j=1}^m$, there is $A_\cdot^T e_j\in \mathcal{H}$, where  $A_\cdot^T$ is the transpose matrix of $A_\cdot$.
Then for every $\psi\in\mathcal{H}$,
\[\int_0^tA_s\d (R_H\psi)(s)=\sum_{j=1}^m\left\< \mathrm{I}_{[0,t]} A_\cdot^Te_j,\psi\right\>_{\mathcal{H}}e_j.\]
\end{lem}
\beg{proof}
According to \eqref{RFRH}, the Fubini theorem and the integral representation \eqref{IRKH}, we obtain that for $\psi\in\mathcal{H}$,
\beg{align*}
\int_0^tA_s\d (R_H\psi)(s)&=\sum_{j=1}^m\left(\int_0^t\<e_j,A_s\d (R_H\psi)(s)\>_{\R^m}\right)e_j\\
&=\sum_{j=1}^m\left(\int_0^t\left\<A_s^Te_j,\left(\int_0^s\ff {\pp K_H} {\pp s}(s,r)(K_H^*\psi)(r)\d r\right)\right\>\d s\right) e_j\\
&=\sum_{j=1}^m\left(\int_0^T \int_0^T\left\< \mathrm{I}_{[0,t]}(s) \mathrm{I}_{[0,s]}(r)A_s^Te_j\ff {\pp K_H} {\pp s}(s,r),(K_H^*\psi)(r)\right\>\d r \d s \right)e_j\\
&=\sum_{j=1}^m\left(\int_0^T\left\< \left(\int_r^T \mathrm{I}_{[0,t]}(s) A_s^Te_j\ff {\pp K_H} {\pp s}(s,r)\d s\right),(K_H^*\psi)(r)\right\>\d r\right)e_j\\
&=\sum_{j=1}^m\left(\int_0^T \left\<\left(K_H^*(\mathrm{I}_{[0,t]}  A_\cdot^Te_j)\right)(r),(K_H^*\psi)(r)\right\>\d r\right)e_j\\
&=\sum_{j=1}^m\left\< \mathrm{I}_{[0,t]} A_\cdot^Te_j,\psi\right\>_{\mathcal{H}}e_j,
\end{align*}
where $\<\cdot,\cdot\>_{\R^m}$ stands for the inner product on $\R^m$, and the last equality is due to the fact that $K_H^*$ is an isometry between $\H$ and $L^2([0,T],\R^d)$.
\end{proof}

We now give some priori estimates to the skeleton equation \eqref{SkE}.

\beg{lem}\label{Le(SkE)}
Assume that \textsc{\textbf{(H)}} holds.
Then for any $M>0$, there exists a positive constant $C_{H,T,M,K_b,\si}$ such that
\begin{align}\label{1Le(SkE)}
\sup_{h\in S_M}\|Z^h\|_{T,\infty}^2\leq C_{H,T,M,K_b,\si}
\end{align}
and
\begin{align}\label{2Le(SkE)}
\sup_{h\in S_M}|Z_t^h-Z_s^h|\leq C_{H,T,M,K_b,\si}(t-s)^{H}, \ \ s,t\in[0,T].
\end{align}
\end{lem}

\beg{proof}
Using the change-of-variables formula \cite[Theorem 4.3.1]{Zahle} and \textsc{\textbf{(H)}}, we first obtain
\begin{align}\label{1PfLe(SkE)}
|Z_t^h|^2&=|x|^2+2\int_0^t\langle Z_s^h, b(Z_s^h,\sL_{X_s^0})\rangle\d s+2\int_0^t\langle Z_s^h, \sigma(Z_s^h,\sL_{X_s^0})\d(R_H h)(s)\rangle\cr
&=:|x|^2+I_1(t)+I_2(t).
\end{align}
By \textsc{\textbf{(H)}} again, we get
\begin{align}\label{2PfLe(SkE)}
I_1(t)\leq&2\int_0^t|Z_s^h|\cdot|b(Z_s^h,\sL_{X_s^0})-b({\bf 0},\delta_{\bf 0})|\d s+2\int_0^t|Z_s^h|\cdot|b({\bf 0},\delta_{\bf 0})|\d s\cr
\leq&2K_b\int_0^t\left(|Z_s^h|^2+|Z_s^h|\cdot|X_s^0|\right)\d s+2|b({\bf 0},\delta_{\bf 0})|\int_0^t|Z_s^h|\d s\cr
\leq&(3K_b+1)\int_0^t|Z_s^h|^2\d s+\left(|b({\bf 0},\delta_{\bf 0})|^2t+K_b\int_0^t|X_s^0|^2\d s\right).
\end{align}
As for the term $I_2(t)$, according to Lemma \ref{est-Rh}, we have
\begin{align}\label{AdPfLe(SkE)}
I_2(t)=2\left\langle\si^T(Z^h_\cdot,\sL_{X_\cdot^0})Z^h_\cdot\mathrm{I}_{[0,t]},h\right\rangle_\H.
\end{align}
This, together with \eqref{EsH}, the Young inequality and \textsc{\textbf{(H)}}, leads to
\begin{align}\label{3PfLe(SkE)}
I_2(t)\leq&2\|\si^T(Z^h_\cdot,\sL_{X_\cdot^0})Z^h_\cdot\mathrm{I}_{[0,t]}\|_\H\cdot\|h\|_\H\cr
\leq&2^{\ff 3 2}H^{\ff 1 2}T^{H-\ff 1 2}\|\si^T(Z^h_\cdot,\sL_{X_\cdot^0})Z^h_\cdot\mathrm{I}_{[0,t]}\|_{L^2}\cdot\|h\|_\H\cr
\leq&2HT^{2H-1}\|h\|_\H^2\int_0^T|\si(Z^h_s,\sL_{X_s^0})|^2\d s+\int_0^t|Z_s^h|^2\d s\cr
\leq& C_{H,T,\si}\|h\|_\H^2+\int_0^t|Z_s^h|^2\d s.
\end{align}
Plugging \eqref{2PfLe(SkE)} and \eqref{3PfLe(SkE)} into \eqref{1PfLe(SkE)} and using the Gronwall inequality yield that for each $h\in S_M$,
\begin{align}\label{4PfLe(SkE)}
\|Z^h\|_{T,\infty}^2&\leq\e^{3K_b+2}
\left(|x|^2+|b({\bf 0},\delta_{\bf 0})|^2T+K_b\int_0^T|X_s^0|^2\d s+C_{H,T,\si}\|h\|_\H^2\right)\cr
&\leq C_{H,T,M,K_b,\si},
\end{align}
which is exactly our first claim \eqref{1Le(SkE)}.

In order to prove the other relation \eqref{2Le(SkE)}, invoke \eqref{SkE} and write that for any $0\leq s<t\leq T$,
\begin{align*}
Z_t^h-Z_s^h=\int_s^tb(Z_r^h,\sL_{X_r^0})\d r+\int_s^t\sigma(Z_r^h,\sL_{X_r^0})\d(R_H{h})(r).
\end{align*}
Owing to \textsc{\textbf{(H)}} and \eqref{4PfLe(SkE)}, we derive that for each $h\in S_M$,
\begin{align}\label{5PfLe(SkE)}
\left|\int_s^tb(Z_r^h,\sL_{X_r^0})\d r\right|
\leq&\int_s^t\left|b(Z_r^h,\sL_{X_r^0})-b({\bf 0},\delta_{\bf 0})\right|\d r+\int_s^t|b({\bf 0},\delta_{\bf 0})|\d r\cr
\leq&K_b\int_s^t\left(|Z_r^h|+|X_r^0|\right)\d r+|b({\bf 0},\delta_{\bf 0})|(t-s)\cr
\leq&C_{H,T,M,K_b}(t-s).
\end{align}
Again using \textsc{\textbf{(H)}} and Lemma \ref{est-Rh}, we arrive at
\begin{align*}
\left|\int_s^t\sigma(Z_r^h,\sL_{X_r^0})\d(R_Hh)(r)\right|^2&=\sum_{j=1}^d\<\mathrm{I}_{[s,t]}\si^T(Z^h_\cdot,\sL_{X_\cdot^0})e_j,h\>_{\mathcal{H}}^2\cr
&\leq \left(\sum_{j=1}^d\|\mathrm{I}_{[s,t]}\si^T(Z^h_\cdot,\sL_{X_\cdot^0})e_j\|_{\mathcal{H}}^2\right)\|h\|_{\mathcal{H}}^2.
\end{align*}
Observe that by \eqref{Isom}, we get
\beg{align*}
&\sum_{j=1}^d\|\mathrm{I}_{[s,t]}\si^T(Z^h_\cdot,\sL_{X_\cdot^0})e_j\|_{\mathcal{H}}^2\cr
&=H(2H-1)\int_s^t\int_s^t |u-r|^{2H-2}\sum_{j=1}^d\left\<\si^T(Z^h_r,\sL_{X_r^0})e_j,\si^T(Z^h_u,\sL_{X_u^0})e_j\right\>\d u\d r\cr
&\leq H(2H-1)\|\si\|_\infty^2\int_s^t\int_s^t |u-r|^{2H-2}\d u\d r\cr
&=2 \|\si\|_\infty^2 (t-s)^{2H}.
\end{align*}
Thus, it follows that
\begin{align*}
\left|\int_s^t\sigma(Z_r^h,\sL_{X_r^0})\d(R_Hh)(r)\right|\leq \sq 2\|\si\|_\infty(t-s)^H \|h\|_{\mathcal{H}}.
\end{align*}
Hence, combining this with \eqref{5PfLe(SkE)}, we conclude that \eqref{2Le(SkE)} holds.
\end{proof}

In the following lemma, we characterize the difference between $X^{\ep, h^\ep}$ and $Z^{h^\ep}$.
\beg{lem}\label{VerCon1}
Assume that \textsc{\textbf{(H)}} holds, and for any $M\in(0,+\infty)$, let  $\{h^\ep:\ep>0\}\subset\mathcal{A}_M$.
Then, there holds
\begin{align*}
\lim\limits_{\ep\ra0}\E\|X^{\ep, h^\ep}-Z^{h^\ep}\|_{T,\infty}^2=0.
\end{align*}
\end{lem}

\beg{proof}
By \eqref{Per-DDsde} and \eqref{SkE}, we have
\begin{align*}
&X_t^{\ep, h^\ep}-Z_t^{h^\ep}\cr
=&\int_0^t\left(b(X_s^{\ep, h^\ep}, \sL_{X_s^\ep})-b(Z_s^{h^\ep},\sL_{X_s^0})\right)\d s+\int_0^t\left(\sigma(X_s^{\ep, h^\ep},\sL_{X_s^\ep})-\sigma(Z_s^{h^\ep},\sL_{X_s^0})\right)\d(R_Hh^\ep)(s)\cr
&+\ep^H\int_0^t\sigma(X_s^{\ep, h^\ep}, \sL_{X_s^\ep})\d B_s^H,\ \ t\in[0,T].
\end{align*}
Using the change-of-variables formula \cite[Theorem 4.3.1]{Zahle}, we derive
\begin{align}\label{0PfLe(VC1)}
|X_t^{\ep, h^\ep}-Z_t^{h^\ep}|^2=&2\int_0^t\left\langle X_s^{\ep, h^\ep}-Z_s^{h^\ep},b(X_s^{\ep, h^\ep}, \sL_{X_s^\ep})-b(Z_s^{h^\ep},\sL_{X_s^0})\right\rangle\d s\cr
&+2\int_0^t\left\langle X_s^{\ep, h^\ep}-Z_s^{h^\ep},\left(\sigma(X_s^{\ep, h^\ep},\sL_{X_s^\ep})-\sigma(Z_s^{h^\ep},\sL_{X_s^0})\right)\d(R_Hh^\ep)(s)\right\rangle\cr
&+2\ep^H\int_0^t\left\langle X_s^{\ep, h^\ep}-Z_s^{h^\ep},\sigma(X_s^{\ep, h^\ep}, \sL_{X_s^\ep})\d B_s^H\right\rangle\cr
=:&J_1(t)+J_2(t)+J_3(t).
\end{align}
In view of \textsc{\textbf{(H)}} and the H\"{o}lder inequality, it is easy to see that
\begin{align}\label{1PfLe(VC1)}
J_1(t)\leq3K_b\int_0^t|X_s^{\ep, h^\ep}-Z_s^{h^\ep}|^2\d s+K_bT\E\|X^\ep-X^0\|_{T,\infty}^2.
\end{align}
For the term $J_2(t)$, it follows from Lemma \ref{est-Rh} that
\begin{align*}
J_2(t)&=2\left\langle\left(\sigma(X_\cdot^{\ep, h^\ep},\sL_{X_\cdot^\ep})-\sigma(Z_\cdot^{h^\ep},\sL_{X_\cdot^0})\right)^T
 (X_\cdot^{\ep, h^\ep}-Z_\cdot^{h^\ep})\mathrm{I}_{[0,t]},h^\ep\right\rangle_\H\cr
 &=:2\left\langle\kappa^\ep\mathrm{I}_{[0,t]},h^\ep\right\rangle_\H.
\end{align*}
Then, by \eqref{EsH}, the Young inequality and \textsc{\textbf{(H)}}, we deduce that for every $h^\ep\in\mathcal{A}_M$,
\begin{align*}
J_2(t)\leq&2\|\kappa^\ep\mathrm{I}_{[0,t]}\|_\H\cdot\|h^\ep\|_\H\cr
\leq&2^{\ff 3 2}H^{\ff 1 2}T^{H-\ff 1 2}\|\kappa^\ep\mathrm{I}_{[0,t]}\|_{L^2}\cdot\|h^\ep\|_\H\cr
\leq&2^{\ff 3 2}H^{\ff 1 2}T^{H-\ff 1 2}\|X^{\ep, h^\ep}-Z^{h^\ep}\|_{0,t,\infty}
\cdot\left\|\left(\sigma(X_\cdot^{\ep, h^\ep},\sL_{X_\cdot^\ep})-\sigma(Z_\cdot^{h^\ep},\sL_{X_\cdot^0})\right)\mathrm{I}_{[0,t]}\right\|_{L^2}\cdot\|h^\ep\|_\H\cr
\leq&\ff 1 3\|X^{\ep, h^\ep}-Z^{h^\ep}\|_{0,t,\infty}^2+C_{H,T}\left\|\left(\sigma(X_\cdot^{\ep, h^\ep},\sL_{X_\cdot^\ep})-\sigma(Z_\cdot^{h^\ep},\sL_{X_\cdot^0})\right)\mathrm{I}_{[0,t]}\right\|^2_{L^2}\cdot\|h^\ep\|^2_\H\cr
\leq&\ff 1 3\|X^{\ep, h^\ep}-Z^{h^\ep}\|_{0,t,\infty}^2
+C_{H,T,M,\si}\left[\int_0^t|X_s^{\ep, h^\ep}-Z_s^{h^\ep}|^2\d s+\E\|X^\ep-X^0\|_{T,\infty}^2\right].
\end{align*}
For the term $J_3(t)$, applying the fractional by parts formula with $\alpha\in(1-\beta,\beta)$, we derive
\begin{align*}
&\int_0^t\left\langle X_s^{\ep, h^\ep}-Z_s^{h^\ep},\sigma(X_s^{\ep, h^\ep}, \sL_{X_s^\ep})\d B_s^H\right\rangle\cr
=&(-1)^\alpha\int_0^tD_{0+}^\alpha((X_\cdot^{\ep, h^\ep}-Z_\cdot^{h^\ep})^T\sigma(X_\cdot^{\ep, h^\ep}, \sL_{X_\cdot^\ep}))(s)D_{t-}^{1-\alpha}B^H_{t-}(s)\d s.
\end{align*}
By \textsc{\textbf{(H)}} and \eqref{2Le(SkE)}, we have
\begin{align*}
&\left|D_{0+}^\alpha((X_\cdot^{\ep, h^\ep}-Z_\cdot^{h^\ep})^T\sigma(X_\cdot^{\ep, h^\ep}, \sL_{X_\cdot^\ep}))(s)\right|\cr
=&\frac{1}{\Gamma(1-\alpha)}\Bigg|\frac{(X_s^{\ep, h^\ep}-Z_s^{h^\ep})^T\sigma(X_s^{\ep, h^\ep}, \sL_{X_s^\ep})}{s^\alpha}\cr
&\qquad\qquad+\alpha\int_0^s\frac{(X_s^{\ep, h^\ep}-Z_s^{h^\ep})^T\sigma(X_s^{\ep, h^\ep}, \sL_{X_s^\ep})-(X_u^{\ep, h^\ep}-Z_u^{h^\ep})^T\sigma(X_u^{\ep, h^\ep}, \sL_{X_u^\ep})}{(s-u)^{\alpha+1}}\d u\Bigg|\cr
\leq&C_{\si,\al}\bigg[|X_s^{\ep, h^\ep}-Z_s^{h^\ep}|s^{-\alpha}+\int_0^s\|Z^{h^\ep}\|_{u,s,H}(s-u)^{H-\alpha-1}\d u\cr
&\qquad+\int_0^s\left(\|X^{\ep, h^\ep}\|_{u,s,\beta}+\|X^{\ep, h^\ep}-Z^{h^\ep}\|_{u,s,\infty}\left(\|X^{\ep, h^\ep}\|_{u,s,\beta}+(\E\|X^\ep\|^2_{u,s,\beta})^\ff 1 2\right)\right)(s-u)^{\beta-\alpha-1}\d u\bigg]\cr
\leq&C_{H,\si,\al,\be}\bigg[|X_s^{\ep, h^\ep}-Z_s^{h^\ep}|s^{-\alpha}+\|Z^{h^\ep}\|_{0,s,H}s^{H-\alpha}\cr
&\qquad\qquad+\left(\|X^{\ep, h^\ep}\|_{0,s,\beta}+\|X^{\ep, h^\ep}-Z^{h^\ep}\|_{0,s,\infty}\left(\|X^{\ep, h^\ep}\|_{0,s,\beta}+(\E\|X^\ep\|^2_{0,s,\beta})^\ff 1 2\right)\right)s^{\beta-\alpha}\bigg].
\end{align*}
Consequently, combining this with the fact that $\left|D_{t-}^{1-\alpha}B^H_{t-}(s)\right|\leq C_{\al,\be}\|B^H\|_{T,\beta}(t-s)^{\alpha+\beta-1}$ and the Young inequality yields
\begin{align}\label{3PfLe(VC1)}
J_3(t)\leq&C_{H,\si,\al,\be}\ep^H\|B^H\|_{T,\beta}\bigg[\|X^{\ep, h^\ep}-Z^{h^\ep}\|_{0,t,\infty} t^\beta+\|Z^{h^\ep}\|_{T,H}t^{H+\beta }\cr
&\qquad\qquad\qquad+\left(\|X^{\ep, h^\ep}\|_{T,\beta}+\|X^{\ep, h^\ep}-Z^{h^\ep}\|_{0,t,\infty}\left(\|X^{\ep, h^\ep}\|_{T,\beta}+(\E\|X^\ep\|^2_{T,\beta})^\ff 1 2\right)\right)t^{2\beta}\bigg]\cr
\leq&\ff 1 3\|X^{\ep, h^\ep}-Z^{h^\ep}\|^2_{0,t,\infty}+C_{H,T,\si,\al,\be}\ep^H\|B^H\|_{T,\beta}\Big[\|Z^{h^\ep}\|_{T,H}+\|X^{\ep, h^\ep}\|_{T,\beta}\cr
&\qquad\qquad\qquad\qquad\qquad\qquad\qquad\qquad\qquad+\ep^H\|B^H\|_{T,\beta}(1+\|X^{\ep, h^\ep}\|^2_{T,\beta}+\E\|X^\ep\|^2_{T,\beta})\Big].
\end{align}
In the spirit of the proofs of Lemmas \ref{Le(SkE)} and \ref{mom-est},
we can derive that the moment estimates of $X^\ep,Z^{h^\ep}$ and $X^{\ep, h^\ep}$ are finite.
Then, plugging \eqref{1PfLe(VC1)}-\eqref{3PfLe(VC1)} into \eqref{0PfLe(VC1)} and taking the expectation
as well as taking into account of the fact that $B^H$ is independent of $\sF_0$, we conclude that for each $h^\ep\in\mathcal{A}_M$,
\begin{align*}
\E\|X^{\ep, h^\ep}-Z^{h^\ep}\|_{0,t,\infty}^2\leq C_{H,T,M,K_b,\si,\al,\be}\left(\int_0^t\E\|X^{\ep, h^\ep}-Z^{h^\ep}\|^2_{0,s,\infty}\d s+\ep^{H}+\E\|X^\ep-X^0\|_{T,\infty}^2\right).
\end{align*}
Here, without loss of generality, we let $\ep\in(0,1)$.\\
Hence, the Gronwall inequality leads to
\begin{align}\label{4PfLe(VC1)}
\E\|X^{\ep, h^\ep}-Z^{h^\ep}\|_{T,\infty}^2\leq C_{H,T,M,K_b,\si,\al,\be}\left(\ep^{H}+\E\|X^\ep-X^0\|_{T,\infty}^2\right).
\end{align}

For $\ep\in(0,1)$, we claim that
\begin{align}\label{5PfLe(VC1)}
\E\|X^\ep-X^0\|_{T,\infty}^2\leq C_{H,T,K_b,\si,\al,\be}\ep^{H}(1+\|X^0\|_{T,\beta}).
\end{align}
Once this is shown, in conjunction with \eqref{4PfLe(VC1)}, we can conclude that
\begin{align*}
\lim\limits_{\ep\ra0}\E\|X^{\ep, h^\ep}-Z^{h^\ep}\|_{T,\infty}^2=0,
\end{align*}
which is the desired assertion.

It remains to show the claim.
From \eqref{LDP-DDsde} and \eqref{LimSDE}, it follows  that
\begin{align*}
X_t^\epsilon-X_t^0=\int_0^t(b(X_s^\ep,\mathscr{L}_{X_s^\ep})-b(X_s^0, \sL_{X_s^0}))\d s+\ep^H\int_0^t\sigma(X_s^\ep,\mathscr{L}_{X_s^\ep})\d B_s^H.
\end{align*}
Similar to \eqref{1PfLe(VC1)} and \eqref{3PfLe(VC1)}, we get
\begin{align*}
&|X_t^\epsilon-X_t^0|^2\cr
=&2\int_0^t\left\langle X_s^\epsilon-X_s^0,b(X_s^\ep,\mathscr{L}_{X_s^\ep})-b(X_s^0, \sL_{X_s^0})\right\rangle\d s
+2\ep^H\int_0^t\left\langle X_s^\epsilon-X_s^0,\sigma(X_s^\ep,\mathscr{L}_{X_s^\ep})\d B_s^H\right\rangle\cr
\leq&K_b\int_0^t(3|X_s^\epsilon-X_s^0|^2+\E|X_s^\epsilon-X_s^0|^2)\d s+\ff 1 2\|X^\ep-X^0\|^2_{0,t,\infty}\cr
&+C_{H,T,\si,\al,\be}\ep^H\|B^H\|_{T,\beta}\Big[\|X^\ep\|_{T,\beta}+\|X^0\|_{T,\beta}+\ep^H\|B^H\|_{T,\beta}(1+\|X^\ep\|^2_{T,\beta}+\E\|X^\ep\|^2_{T,\beta})\Big].
\end{align*}
Taking the expectation and using the fact that $B^H$ is independent of $\sF_0$ again, the Gronwall inequality  yields \eqref{5PfLe(VC1)}.
This completes the proof.
\end{proof}

Now, we can go back to the proof of Theorem \ref{Th(LDP)}.

\emph{\textbf{Proof of Theorem \ref{Th(LDP)}.}}
With the help of Proposition \ref{Suf2(LDP)}, to complete the proof of Theorem \ref{Th(LDP)}, it is sufficient to check that
\textsc{\textbf{(A)}} holds with $\G^0,\G^\ep$ and $\ell(\ep)$ given respectively by  \eqref{RateF}, \eqref{Per2-DDsde} and $\ep^{2H}$.
Note that by \eqref{Per2-DDsde} and \eqref{RateF}, we know that $\mathcal{G}^\ep(\ep^HB^H+R_Hh^\ep)=X^{\ep, h^\ep}$ and
$\G^0(R_Hh^\ep)=Z^{h^\ep}$.
Then from Lemma \ref{VerCon1}, one can see that \textsc{\textbf{(A)}}(i) holds.

Now, we are to verify that \textsc{\textbf{(A)}}(ii) holds.

For any $M\in(0,+\infty)$, let $\{h^n:n\in\mathbb{N}\}\subset\mathcal{S}_M$ satisfying that $h^n$ converges to some element $h$ in $S_M$ as $n\ra+\infty$,
First note that for any $n\geq1$, by \eqref{RateF} we can write $\G^0(R_Hh^n)$ and $\G^0(R_Hh)$ as follows:
\begin{align}\label{0PfT(LDP)}
\G^0(R_Hh^n)=Z^{h^n},\ \ \G^0(R_Hh)=Z^h,
\end{align}
where $Z^{h^n}$ is the solution of equation \eqref{SkE} for $h^n$ replacing $h$ there.
By Lemma \ref{Le(SkE)}, we know that $\{Z^{h^n}\}_{n\geq1}$ is uniformly bounded and equi-continuous in $C([0,T]; \R^d)$.
Then, it follows from the Arzel\`{a}-Ascoli theorem that $\{Z^{h^n}\}_{n\geq1}$ is relatively compact in $C([0,T]; \R^d)$,
which means that for any subsequence of $\{Z^{h^n}\}_{n\geq1}$, there is a further subsequence (not relabelled) such that  $Z^{h^n}$ converges to some $\bar{Z}$ in $C([0,T]; \R^d)$.

Next, we intend to prove the relation $\bar{Z}=Z^h$.
This, along with a standard subsequential argument, easily yields that the full sequence  $Z^{h^n}$ converges to $Z^h$ in $C([0,T];\R^d)$.
Consequently, we have
\begin{align}\label{AdPfT(LDP)}
\lim_{n\ra+\infty}\left\|\G^0(R_Hh^n)-\G^0(R_Hh)\right\|_{T,\infty}=0,
\end{align}
because of \eqref{0PfT(LDP)}. This means that \textsc{\textbf{(A)}}(ii) holds.

We now focus on dealing with this relation.
By \textsc{\textbf{(H)}}, it is easy to see that for each $t\in[0,T]$,
\begin{align*}
\left|\int_0^tb(Z_s^{h^n},\sL_{X_s^0})\d s-\int_0^tb(\bar{Z}_s,\sL_{X_s^0})\d s\right|
\leq& K_b\int_0^t\left|Z_s^{h^n}-\bar{Z}_s\right|\d s\cr
\leq & K_bT\|Z^{h^n}-\bar{Z}\|_{T,\infty}\ra0, \ \ n\ra+\infty.
\end{align*}
Consequently, we derive that for every $t\in[0,T]$,
\begin{align}\label{1PfT(LDP)}
\lim_{n\ra+\infty}\int_0^tb(Z_s^{h^n},\sL_{X_s^0})\d s=\int_0^tb(\bar{Z}_s,\sL_{X_s^0})\d s.
\end{align}
On the other hand, using Lemma \ref{est-Rh}, we obtain
\begin{align}\label{2PfT(LDP)}
&\int_0^t\sigma(Z_s^{h^n},\sL_{X_s^0})\d(R_H{h^n})(s)-\int_0^t\sigma(\bar{Z}_s,\sL_{X_s^0})\d(R_H{h})(s)\cr
=&\int_0^t\left(\sigma(Z_s^{h^n},\sL_{X_s^0})-\sigma(\bar{Z}_s,\sL_{X_s^0})\right)\d(R_H{h^n})(s) +\int_0^t\sigma(\bar{Z}_s,\sL_{X_s^0})\d (R_H({h^n}-h))(s)\cr
=&\sum_{j=1}^d \left\<\mathrm{I}_{[0,t]}\left(\sigma^T(Z_\cdot^{h^n},\sL_{X_\cdot^0})-\sigma^T(\bar{Z}_\cdot,\sL_{X_\cdot^0})\right)e_j,h^n \right\>_{\mathcal{H}}e_j\cr
&+\sum_{j=1}^d\left\<\mathrm{I}_{[0,t]}\sigma^T(\bar{Z}_\cdot,\sL_{X_\cdot^0})e_j,h^n-h\right\>_{\mathcal{H}}e_j\cr
=:&I_1(t)+I_2(t).
\end{align}
Owing to \eqref{EsH} and \textsc{\textbf{(H)}}, we have
\begin{align}\label{3PfT(LDP)}
|I_1(t)|^2=&\sum_{j=1}^d\left\<\mathrm{I}_{[0,t]}\left(\sigma^T(Z_\cdot^{h^n},\sL_{X_\cdot^0})-\sigma^T(\bar{Z}_\cdot,\sL_{X_\cdot^0})\right)e_j,h^n \right\>_{\mathcal{H}}^2\cr
\leq&2HT^{2H-1}\sum_{j=1}^d \left\|\mathrm{I}_{[0,t]}\left(\sigma^T(Z_\cdot^{h^n},\sL_{X_\cdot^0})-\sigma^T(\bar{Z}_\cdot,\sL_{X_\cdot^0})\right)e_j\right\|_{L^2}^2
\cdot\|h^n  \|_{\mathcal{H}}^2\cr
\leq&2dHT^{2H}\|\nabla\si\|_\infty^2\|Z^{h^n}-\bar{Z}\|_{T,\infty}^2\cdot\|h^n  \|_{\mathcal{H}}^2\cr
\leq& C_{H,T,M,d,\si}\|Z^{h^n}-\bar{Z}\|_{T,\infty}^2\ra0,\ \ n\ra+\infty.
\end{align}
For the term $I_2(t)$, since $h^n$ converges to $h$ weakly in $S_M$ as $n\ra+\infty$, we get
\beg{align*}
|I_2(t)|^2=\sum_{j=1}^d\left\<\mathrm{I}_{[0,t]}\left(\sigma^T(\bar{Z}_\cdot,\sL_{X_\cdot^0})\right)e_j,h^n-h\right\>_{\mathcal{H}}^2\ra 0,\quad n\ra +\infty.
\end{align*}
Combining this with \eqref{3PfT(LDP)} and \eqref{2PfT(LDP)}, we derive that for each $t\in[0,T]$,
\begin{align}\label{4PfT(LDP)}
\lim\limits_{n\ra+\infty}\int_0^t\sigma(Z_s^{h^n},\sL_{X_s^0})\d(R_H{h^n})(s)=\int_0^t\sigma(\bar{Z}_s,\sL_{X_s^0})\d(R_H{h})(s).
\end{align}
Recall that $Z^{h^n}$ solves the following equation:
\begin{align*}
Z_t^{h^n}=x+\int_0^tb(Z_s^{h^n},\sL_{X_s^0})\d s+\int_0^t\sigma(Z_s^{h^n},\sL_{X_s^0})\d(R_H h^n)(s),\ \ t\in[0,T].
\end{align*}
Letting $n\ra+\infty$ and using \eqref{1PfT(LDP)} and \eqref{4PfT(LDP)}, we conclude that $\bar{Z}$ satisfies \eqref{SkE},
which implies $\bar{Z}=Z^h$ due to the uniqueness of the solutions to equation \eqref{SkE}.
The proof is now complete.
\qed

\subsection{Moderate deviation principle (MDP)}

In this part, we will focus on the MDP for equation \eqref{LDP-DDsde} as $\ep$ tends to $0$.
More precisely, let
\begin{align}\label{Meq-1}
Y^\ep=\ff {X^\ep-X^0}{\ep^H\zeta(\ep)},
\end{align}
where $\zeta(\ep)\ra+\infty, \ep^H\zeta(\ep)\ra0$ as $\ep\ra0$.
The moderate deviations problem for $\{X^\ep:\ep>0\}$ is indeed to study the asymptotics of
\begin{align*}
\ff 1{\zeta^2(\ep)}\log\P(Y^\ep\in\cdot).
\end{align*}
In view of \eqref{LDP-DDsde} and \eqref{LimSDE}, one can see that $Y^\ep=\{Y^\ep_t,t\in[0,T]\}$ satisfies the following equation
\begin{align}\label{Meq-2}
Y^\ep_t=&\ff 1{\ep^H\zeta(\ep)}\int_0^t\left(b(X^0_s+\ep^H\zeta(\ep)Y_s^\ep,\mathscr{L}_{X_s^\ep})-b(X_s^0, \mathscr{L}_{X_s^0})\right)\d s\cr
&+\ff 1 {\zeta(\ep)}\int_0^t\si\left(X^0_s+\ep^H\zeta(\ep)Y_s^\ep,\mathscr{L}_{X_s^\ep}\right)\d B_s^H.
\end{align}

To show the MDP for equation \eqref{LDP-DDsde}, in additional to \textsc{\textbf{(H)}}, we also need to impose the following condition.
\begin{enumerate}
\item[\textsc{\textbf{(\~{H})}}] The derivative $\nabla b(\cdot,\mu)(x)$ exists and there is a constant $\widetilde{K}_b(>0)$ such that for each $x,y\in\R^d$ and $\mu\in\sP_2(\R^d)$,
\begin{align*}
\left\|\nabla b(\cdot, \mu)(x)-\nabla b(\cdot, \mu )(y) \right\|\leq \widetilde{K}_b|x-y|.
\end{align*}
\end{enumerate}
In order to provide the rate function associated with the MDP, we introduce the following skeleton equation: for every $h\in\H$,
\begin{align}\label{Sk(mdp)}
\tilde{Z}_t^h=\int_0^t\nabla_{\tilde{Z}_s^h} b(\cdot,\sL_{X_s^0})(X_s^0)\d s+\int_0^t\sigma(X_s^0,\sL_{X_s^0})\d(R_H h)(s),\ \ t\in[0,T].
\end{align}
With the help of \textsc{\textbf{(H)}} and \textsc{\textbf{(\~{H})}} and by a similar analysis of \eqref{SkE},
we obtain that there is a unique solution to equation \eqref{Sk(mdp)}.
As a consequence, we may define a map as follows: for any $h\in\H$,
\begin{align}\label{RaFu(mdp)}
\widetilde{\G}^0: I_{0+}^{H+1/2}(L^2([0,T],\R^d))\ni R_Hh\mapsto \tilde{Z}^h\in C([0,T]; \R^d).
\end{align}
Then, our main result in this part can be stated in the following theorem.

\beg{thm}\label{Th(mdp)}
Assume that \textsc{\textbf{(H)}} and \textsc{\textbf{(\~{H})}} hold. For each $\ep>0$, let $Y^\ep=\{Y_t^\ep\}_{t\in[0,T]}$ be shown in \eqref{Meq-1}.
Then the family $\{Y^\ep:\ep>0\}$ satisfies the LDP on $C([0,T]; \R^d)$ with speed $\zeta^{-2}(\ep)$,
in which the rate function $I$ is given by
\begin{align}\label{1-Th(mdp)}
I(f)=\inf_{\{h\in\H:f=\widetilde{\G}^0(R_Hh)\}}\left\{\ff 1 2 \|h\|_{\H}^2\right\},\ \ f\in C([0,T]; \R^d)
\end{align}
with $\widetilde{\G}^0$ being defined in \eqref{RaFu(mdp)}.
\end{thm}

\beg{proof}
We first set
\begin{align*}
\widetilde{\G}^\ep(\cdot):=\ff {\mathcal{G}^\ep(\cdot)-X^0}{\ep^H\zeta(\ep)}.
\end{align*}
Recalling that the definition of $\G^\epsilon$, we derive that $\widetilde{\G}^\ep$ is a measurable map from $C([0,T];\R^d)$ to  $C([0,T];\R^d)$
satisfying $\widetilde{\G}^\ep(\ep^H B_\cdot^H)=Y_\cdot^\ep$.

Next, for any $h^\ep\in\mathcal{A}_M$, let
\begin{align*}
Y_\cdot^{\ep, h^\ep}=\widetilde{\G}^\ep\left(\ep^H B_\cdot^H+\ep^H\zeta(\ep)(R_Hh^\ep)(\cdot)\right),
\end{align*}
then, $Y^{\ep, h^\ep}$ is the unique solution of the following equation:
\begin{align*}
Y_t^{\ep, h^\ep}=&\ff 1{\ep^H\zeta(\ep)}\int_0^t\left(b(X^0_s+\ep^H\zeta(\ep)Y_s^{\ep, h^\ep}, \sL_{X_s^\ep})- b(X_s^0, \mathscr{L}_{X_s^0})\right)\d s\cr
&+\int_0^t\si(X^0_s+\ep^H\zeta(\ep)Y_s^{\ep, h^\ep},\sL_{X_s^\ep})\d(R_Hh^\ep)(s)\cr
&+\ff 1{\zeta(\ep)}\int_0^t\si(X^0_s+\ep^H\zeta(\ep)Y_s^{\ep, h^\ep}, \sL_{X_s^\ep})\d B_s^H,\ \ t\in[0,T],\ \mathbb{P}\textit{-}a.s..
\end{align*}

With Proposition \ref{Suf2(LDP)} in hand, it is enough to verify that \textsc{\textbf{(A)}} holds respectively with $\G^0, \G^\ep$ and $\ell(\ep)$ replaced by $\widetilde{\G}^0, \widetilde{\G}^\ep$ and $\zeta^{-2}(\ep)$.
Along the same lines as in Lemma \ref{VerCon1} and \eqref{AdPfT(LDP)}, we can prove that for any $M\in(0,+\infty)$,
if $\{h^\ep:\ep>0\}\subset\mathcal{A}_M$ and $\{h^n:n\in\mathbb{N}\}\subset\mathcal{S}_M$ such that $h^n$ converges to some element $h$ in $S_M$ as $n\ra+\infty$, we have
\begin{align*}
\lim\limits_{\ep\ra0^+}\E\|Y^{\ep, h^\ep}-\widetilde{\G}^0(R_Hh^\ep)\|_{T,\infty}^2=\lim\limits_{\ep\ra0^+}\E\|Y^{\ep, h^\ep}-\tilde{Z}^{h^\ep}\|_{T,\infty}^2=0
\end{align*}
and
\begin{align*}
\lim_{n\ra+\infty}\left\|\widetilde{\G}^0(R_Hh^n)-\widetilde{\G}^0(R_Hh)\right\|_{T,\infty}=0,
\end{align*}
which mean that \textsc{\textbf{(A)}} holds. The proof is therefore finished.
\end{proof}

\section{Appendix}

\subsection{Two technical lemmas}

\begin{lem}\label{TeLe1}
Let $x,y$ be two positive constants and $0<\al<\be<1$. Then for any $0<r<s$, we have
\begin{align*}
\int_s^r\ff { (x(r-u)^\be)\wedge y } {(r-u)^{\al+1}}\d u\leq \ff {4\be} {(\be-\al)\al}x^{\ff {\al} {\be}}y^{\ff {\be-\al} {\be}}.
\end{align*}
\end{lem}

\begin{proof}
Observe that for any $a>0,b>0$, there holds $a\wedge b\leq \ff{2ab}{a+b}.$
Then, we deduce that
\beg{equation*}\label{ine-sig}
\beg{split}
\int_s^r\ff { (x(r-u)^\be)\wedge y } {(r-u)^{\al+1}}\d u&\leq \int_s^r\ff { 2x(r-u)^{\be-\al-1}y} {x(r-u)^\be+y}\d u\\
&= 2x(r-s)^{\be-\al }\int_0^1 \ff {    v^{\be-\al-1}} {1+ x(r-s)^\be y^{-1} v^\be} \d v\\
&\leq  \ff {4\be x(r-s)^{\be-\al }} {(\be-\al)\al\left(1+\left(x(r-s)^\be y^{-1}\right)^{\ff {\be-\al} {\be}}\right)}\\
&=\left(\ff {4\be} {(\be-\al)\al}\right)\left(\ff {x(r-s)^{\be-\al }y^{\ff {\be-\al} {\be}}} {y^{\ff {\be-\al} {\be}}+\left(x(r-s)^\be \right)^{\ff {\be-\al} {\be}}}\right)\\
&\leq \ff {4\be} {(\be-\al)\al}x^{\ff {\al} {\be}}y^{\ff {\be-\al} {\be}},
\end{split}
\end{equation*}
where in the second inequality, we have used the inequality in \cite[Lemma 3.4]{FZ} just setting $p=\be,p'=\al+1$ (then $C(p,p')=\ff {2\be} {(\be-\al)\al}$) and $x=1$ there.
\end{proof}
The following result has been obtained in the proof of \cite[Theorem 3.1]{FZ}, and we give it here for  the convenience of readers.
\begin{lem}\label{TeLe2}
Assume that $f\in C^\be([0,T];\R^d)$ with $\be\in (0,1]$,
and let $0=t_0<t_1<t_2<\cdots<t_n=T$. Then there holds
\begin{align*}
\|f\|_{T,\be}\leq n^{1-\be}\max_{0\leq k\leq n-1}\{\|f\|_{t_k , t_{k+1} ,\be }\}.
\end{align*}
\end{lem}

\begin{proof}
By using the Jensen inequality, we obtain that for any $0\leq s<t\leq T$,
\begin{align*}
\ff {|f(t)-f(s)|} {|t-s|^{\be}}&\leq \sum_{k=0}^{n-1}\ff {|f((t_{k+1}\wedge t)\vee s)-f((t_k\wedge t)\vee s)|} {|t-s|^{\be}}\\
&\leq \sum_{k=0}^{n-1}\ff {|(t_{k+1}\wedge t)\vee s-(t_k\wedge t)\vee s|^\be} {|t-s|^{\be}}\|f\|_{(t_k\wedge t)\vee s, (t_{k+1}\wedge t)\vee s,\be }\\
&\leq \max_{0\leq k\leq n-1}\{\|f\|_{t_k , t_{k+1} ,\be }\}\sum_{k=0}^{n-1}\ff {|(t_{k+1}\wedge t)\vee s-(t_k\wedge t)\vee s|^\be} {|t-s|^{\be}}\\
&=\max_{0\leq k\leq n-1}\{\|f\|_{t_k , t_{k+1} ,\be }\}\ff n {|t-s|^{\be}}\sum_{k=0}^{n-1}\ff {|(t_{k+1}\wedge t)\vee s-(t_k\wedge t)\vee s|^\be} {n}\\
&\leq\max_{0\leq k\leq n-1}\{\|f\|_{t_k , t_{k+1} ,\be }\}\ff n {|t-s|^{\be}}\left(\sum_{k=0}^{n-1}\ff { (t_{k+1}\wedge t)\vee s-(t_k\wedge t)\vee s } {n}\right)^{\be}\\
&= n^{1-\be}\max_{0\leq k\leq n-1}\{\|f\|_{t_k , t_{k+1} ,\be }\},
\end{align*}
which yields the desired result.
\end{proof}

\subsection{Moment estimate for multiplicative fractional SDE with time dependent coefficients}

In this part, we consider the following multiplicative fractional SDE with time dependent coefficients
\begin{equation}\label{Lemma 3.2-0}
\tilde{X}(t)=\tilde{X}(S)+\int_S^t\tilde{b}(s,\tilde{X}(s))\d s+\int_S^t\tilde{\sigma}(s,\tilde{X}(s))\d B_s^H,\ \ S\leq t\leq T,
\end{equation}
where $S\in[0,T)$, $\tilde{X}(S)$ is a $\sF_S$-measurable random variable, the coefficients $\tilde{b}:[S,T]\times\R^d\ra\R^d, \tilde{\si}:[S,T]\times\R^d\ra\R^d\otimes\R^d$ are two measurable functions, and $B^H$ is a $d$-dimensional fractional Brownian motion with $H\in(1/2,1)$.

Below is a moment estimate for equation \eqref{Lemma 3.2-0}, which plays a crucial role in the proof of Theorem \ref{ExU}.

\begin{lem}\label{mom-est}
Let $\ff 1 {2}<\be<H$. Assume that $\tilde{b}$ is of linear growth:
$$|\tilde{b}(r,x)|\leq K_{\tilde{b}}+L_{\tilde{b}}(1+|x|),\ \ r\in[S,T],\ \ x\in\R^d,$$
and $\tilde{\sigma}$ is bounded  by a positive constant $M_{\ti{\si}}$ and
\beg{align*}
\|\tilde{\si}(r,x)-\tilde{\si}(s,y)\|\leq K_{\tilde{\si}}|r-s|^{\ga_0}+L_{\tilde{\si}}|x-y|,\ \ r,s\in[S,T], \ \ x,y\in\R^d
\end{align*}
for some $\ga_0 >1-\be$, where $K_{\tilde{b}},L_{\tilde b},K_{\tilde{\si}},L_{\tilde{\si}}$ are all positive constants.
Then for the solution $\tilde{X}$ of equation \eqref{Lemma 3.2-0} and any $p\geq 1$, there holds
\begin{align*}
&\left[\E\left(\|\tilde{X}\|_{S,T,\infty}+\|\tilde{X}\|_{S,T,\be}\right)^p\right]^{\ff 1 p}\cr
\leq& C_1(T-S,H,\be,\ga_0,p)\left(1+K_{\tilde{b}}+K_{\tilde{\sigma}} \right)+ C_2(T-S,H,\be,\ga_0,p)\left(\E|\tilde{X}(S)|^p\right)^{\ff 1 p},
\end{align*}
where $C_i(T-S,H,\be,\ga_0,p),i=1,2$ are two positive constants which are independent of $S,T$, $K_{\tilde b},K_{\tilde \si}$, $\tilde X(S)$ and decreasing as $T-S$ decreases, and satisfy
\begin{align*}
\lim_{T-S\rightarrow 0^+}C_1(T-S,H,\be,\ga_0,p)=0\ \ \ \mathrm{and}\ \ \ \lim_{T-S\rightarrow 0^+}C_2(T-S,H,\be,\ga_0,p)=1.
\end{align*}
\end{lem}

\begin{proof}
By \eqref{Lemma 3.2-0}, we derive that for any $S\leq s<t\leq T$,
\begin{align}\label{1Pfmom}
\tilde{X}(t)-\tilde{X}(s)=\int_s^t\tilde{b}(r,\tilde{X}(r))\d r+\int_s^t\tilde{\sigma}(r,\tilde{X}(r))\d B_r^H.
\end{align}
Observe that the conditions $\ff 1 {2}<\be<H$ and $\ga_0 >1-\be$ imply $1-\be<\be\wedge \ga_0$.
Applying the fractional integration by parts formula \eqref{Fibpf}  with $\al\in (1-\be, \beta\wedge \ga_0)$, we have
\begin{equation}\label{Lemma 3.2-1}
\int_s^t\tilde{\sigma}(r,\tilde{X}(r))\d B_r^H=(-1)^\alpha\int_s^tD_{s+}^\alpha\tilde{\sigma}(\cdot,\tilde{X}(\cdot))(r)D_{t-}^{1-\alpha}B^H_{t-}(r)\d r.
\end{equation}
Using \eqref{FrDe}, the condition imposed on $\ti\si$ and Lemma \ref{TeLe1}, we obtain
\begin{align}\label{Lemma 3.2-2}
&\left|D_{s+}^\alpha\tilde{\sigma}(\cdot,\tilde{X}(\cdot))(r)\right|\nonumber\\
=&\frac{1}{\Gamma(1-\alpha)}\left|\frac{\tilde{\sigma}(r,\tilde{X}(r))}{(r-s)^\alpha}+\alpha\int_s^r\frac{\tilde{\sigma}(r,\tilde{X}(r))-\tilde{\sigma}(u,\tilde{X}(u))}{(r-u)^{\alpha+1}}\d u\right|\nonumber\\
\leq &\frac{1}{\Gamma(1-\alpha)}\Bigg[\ff {M_{\ti{\si}}}{(r-s)^{\alpha}}
+\alpha\int_s^r\left(K_{\tilde{\sigma}}(r-u)^{\ga_0-\alpha-1}+\ff{(L_{\tilde{\si}}\|\tilde{X}\|_{u,r,\beta} (r-u)^{\be})\wedge(2M_{\ti{\si}})}
{(r-u)^{\alpha+1}}
\right)\d u\Bigg]\nonumber\\
\leq&\frac{1}{\Gamma(1-\alpha)}\left[M_{\ti{\si}}(r-s)^{-\alpha}+\alpha K_{\tilde{\sigma}}(r-s)^{\ga_0-\al}+\ff {2^{3-\ff {\al} {\be}}\be} {\be-\al}M_{\ti{\si}}^{\ff {\be-\al} {\be}}(L_{\tilde{\si}}\|\tilde{X}\|_{s,r,\beta})^{\ff {\al} {\be}} \right],
\end{align}
and
\begin{align}\label{Lemma 3.2-3}
\left|D_{t-}^{1-\alpha}B^H_{t-}(r)\right|
=&\frac{1}{\Gamma(\alpha)}\left|\frac{B^H_r-B^H_t}{(t-r)^{1-\alpha}}+(1-\alpha)\int_r^t\frac{B^H_r-B^H_u}{(u-r)^{2-\alpha}}\d u\right|\nonumber\\
\leq&\ff{\be}{\Gamma(\alpha)(\alpha+\beta -1)}\|B^H\|_{S,T,\beta}(t-r)^{\alpha+\beta-1}.
\end{align}
Then, plugging \eqref{Lemma 3.2-2} and \eqref{Lemma 3.2-3} into \eqref{Lemma 3.2-1} yields
\begin{align}\label{Lemma 3.2-4}
&\left|\int_s^t\tilde{\sigma}(r,\tilde{X}(r))\d B_r^H\right|\cr
\leq& C\|B^H\|_{S,T,\beta}\Bigg[\int_s^t(r-s)^{-\alpha}(t-r)^{\alpha+\beta -1}\d r\nonumber\\
&\ \ \ \ \ \ \ \ \ \ \ \ \ \ \ +\int_s^t\left(K_{\tilde{\sigma}}(r-s)^{\ga_0-\al}(t-r)^{\al+\be -1}+ (t-r)^{\alpha+\beta -1}\|\tilde{X}\|_{s,r,\beta}^{\ff {\al} {\be}} \right)\d r\Bigg]\cr
\leq & C\|B^H\|_{S,T,\beta}\left[(t-s)^{\be }+K_{\tilde{\si}}(t-s)^{\ga_0+\be}+ (t-s)^{\al+\be }\|\tilde{X}\|_{s,t,\beta}^{\ff {\al} {\be}}\right].
\end{align}
Here and in what follows, $C$ denotes a generic constant independent of $K_{\tilde{b}},K_{\tilde{\si}},T,S$ and $B^H$.\\
As for the drift term, owing to the linear growth of $\ti b$, it is easy to see that
\begin{align*}
\left|\int_s^t\tilde{b}(r,\tilde{X}(r))\d r\right|\leq \int_s^t\left[K_{\tilde{b}}+L_{\ti b}\left(1+|\tilde{X}(r)|\right)\right]\d r.
\end{align*}
So, combining this with \eqref{Lemma 3.2-4} and \eqref{1Pfmom} yields
\begin{align}\label{3.1-2}
&|\tilde{X}(t)-\tilde{X}(s)|\leq \int_s^t|\tilde{b}( r,\tilde{X}(r))|\d r+\left|\int_s^t\tilde{\si}( r, \tilde{X}(r))\d  B_r^H\right|\cr
&\quad \leq K_{\tilde{b}}(t-s)+L_{\ti b}\int_s^t(1+|\tilde{X}(r)|)\d r\cr
&\qquad +C \|B^H\|_{S,T,\be}\left[(t-s)^{\be }+ K_{\tilde{\si}}(t-s)^{\ga_0+\be }+(t-s)^{\al+\be }\| \tilde{X}\|_{s,t,\be}^{\ff {\al} {\be}}\right]\cr
&\quad\leq \left(K_{\tilde{b}}+L_{\ti b}\right)(t-s)+L_{\ti b}(t-s)^\be\left(\int_s^t|\tilde{X}(r)|^{\ff 1 {1-\be}}\d r\right)^{1-\be} \cr
&\qquad +C \| B^H\|_{S,T,\be}\left[(t-s)^{\be }+ K_{\tilde{\sigma}}(t-s)^{\ga_0+\be }+(t-s)^{\al+\be }\| \tilde{X}\|_{s,t,\be}^{\ff {\al} {\be}}\right],
\end{align}
where we use the H\"{o}lder inequality in the last inequality. \\
Consequently, it follows that
\begin{align*}
\ff {|\tilde{X}(t)-\tilde{X}(s)|}{(t-s)^\be}
& \leq  \left(K_{\tilde{b}}+L_{\ti b}\right)(t-s)^{1-\be}+L_{\ti b}\left(\int_s^t|\tilde{X}(r)|^{\ff 1 {1-\be}}\d r\right)^{1-\be}\cr
&\quad +C \|B^H\|_{S,T,\be}\left[1+ K_{\tilde{\sigma}}(t-s)^{\ga_0}+(t-s)^{\al }\|\tilde{ X}\|_{s,t,\be}^{\ff {\al} {\be}}\right].
\end{align*}
Note that there holds
$$(t-s)^{\al }\|\tilde{X}\|_{s,t,\be}^{\ff {\al} {\be}}\leq \ff {\be-\al} \be+\ff{\al}{\be}(t-s)^\be\|\tilde{X}\|_{s,t,\be}$$
because of the Young inequality.
Consequently, we get
\begin{align}\label{3.1-add_1}
\|\tilde{X}\|_{s,t,\be}&\leq\left(K_{\tilde{b}}+L_{\ti b}\right)(t-s)^{1-\be}+L_{\ti b}\left(\int_s^t|\tilde{X}(r)|^{\ff 1 {1-\be}}\d r\right)^{1-\be}\nonumber\\
&\quad+C \|  B^H\|_{S,T,\be}\left(1+ K_{\tilde{\sigma}}(t-s)^{\ga_0 }+(t-s)^\be\| \tilde{X}\|_{s,t,\be}\right)\nonumber\\
&\leq \left(K_{\tilde{b}}+L_{\ti b}(1+|\tilde{X}(s)|)\right)(t-s)^{1-\be}+L_{\ti b}\left(\int_s^t|\tilde{X}(r)-\tilde{X}(s)|^{\ff 1 {1-\be}}\d r\right)^{1-\be}\nonumber\\
&\quad+C \|B^H\|_{S,T,\be}\left(1+ K_{\tilde{\sigma}}(t-s)^{\ga_0 }+(t-s)^\be\| \tilde{X}\|_{s,t,\be}\right)\nonumber\\
&\leq \left(K_{\tilde{b}}+L_{\ti b}(1+|\tilde{X}(s)|)\right)(t-s)^{1-\be}+L_{\ti b}\left(\int_s^t(r-s)^{\ff \be {1-\be}}\d r\right)^{1-\be}\|\tilde{X}\|_{s,t,\be}\nonumber\\
&\quad +C \|B^H\|_{S,T,\be}\left(1+ K_{\tilde{\sigma}}(t-s)^{\ga_0 }+(t-s)^\be\| \tilde{X}\|_{s,t,\be}\right)\nonumber\\
&=C\left(\|B^H\|_{S,T,\be}\left(1+K_{\tilde{\sigma}}(t-s)^{\ga_0}\right)+(t-s)^{1-\be}\left(1+K_{\tilde{b}}+|\tilde{X}(s)|\right)\right)\nonumber\\
&\quad+C(t-s)^\be\left(\|B^H\|_{S,T,\be}+(t-s)^{1-\be}\right)\|\tilde{X}\|_{s,t,\be},
\end{align}
where the second inequality is due to the Minkowski inequality.\\
Without of lost generality, we assume that $C\geq1$ in \eqref{3.1-add_1}, and let
\beg{align}\label{ad-De}
\Delta=\left(\ff {1\wedge (T-S)^{\be }} {6C\|B^H\|_{S,T,\be}}\right)^{\ff {1} { \be}}\wedge\left(\ff {1\wedge (T-S)^{\be }} {6C(1+K_{\tilde{\si}})\|B^H\|_{S,T,\be }}\right)^{\ff 1 {\ga_0+\be}}\wedge \ff {1\wedge (T-S)^{\be }} {6C(1+K_{\tilde{b}})},
\end{align}
which implies that $\Delta<1$ and
\begin{align}\label{2Pfmom}
\left[\|B^H\|_{S,T,\be}\left(\Delta^{\be} \vee\left((1+K_{\tilde{\si}}) \Delta^{\ga_0+\be  } \right)\right)\right]\vee \left[(1+K_{\tilde{b}})\Delta\right]\leq\ff {1\wedge (T-S)^{\be }} {6C}.
\end{align}
Taking $t=s+\Delta$ in \eqref{3.1-add_1}, we arrive at
\begin{align}\label{ad-n-Xbe}
\|\tilde{X}\|_{s,s+\Delta,\be}&\leq \ff {C\left(\|B^H\|_{S,T,\be}(1+K_{\tilde{\sigma}}\Delta^{\ga_0 })+\Delta^{1-\be}(1+K_{\tilde{b}}+|\tilde{X}(s)|)\right)} {1-C\Delta^\be(\|B^H\|_{S,T,\be}+\Delta^{1-\be})}\cr
&\leq \ff {C\left(\|B^H\|_{S,T,\be}(1+K_{\tilde{\sigma}}\Delta^{\ga_0 })+\Delta^{1-\be}\left(1+K_{\tilde{b}}+|\tilde{X}(s)|\right)\right)} {1-\ff 1 6-\ff 1 6}\cr
&\leq2C\left(\|B^H\|_{S,T,\be}(1+K_{\tilde{\sigma}}\Delta^{\ga_0 })+\Delta^{1-\be}\left(1+K_{\tilde{b}}+|\tilde{X}(s)|\right)\right).
\end{align}
To simplify the notations, we set $\varsigma:=2C\Delta$ and for every $k\in \mathbb{N}$,
$$\Pi_k:=\|\tilde{X}\|_{S+k\Delta,(S+(k+1)\Delta)\wedge T,\be}\Delta^\be.$$
By \eqref{ad-n-Xbe} and \eqref{2Pfmom}, we deduce that for every $k\in \mathbb{N}$,
\begin{align}\label{Fan-add1}
\Pi_k&\leq 2C\left( \|B^H\|_{S,T,\be}(\Delta^{\be}+K_{\tilde{\sigma}}\Delta^{\ga_0+\be })+(1+K_{\tilde{b}})\Delta+\Delta |\tilde{X}(S+k\Delta)|\right)\cr
&\leq 2C\left( \|B^H\|_{S,T,\be}(\Delta^{\be}+(1+K_{\tilde{\sigma}})\Delta^{ \ga_0+\be  })+(1+K_{\tilde{b}})\Delta+\Delta |\tilde{X}(S+k\Delta)|\right)\cr
&\leq  2C\left(\ff {1\wedge (T-S)^{\be }} {6C}+\ff {1\wedge (T-S)^{\be}} {6C}+\ff {1\wedge (T-S)^{\be }} {6C} \right)+\varsigma|\tilde{X}(S+k\Delta)|\cr
&= 1\wedge (T-S)^{\be } +\varsigma|\tilde{X}(S+k\Delta)|,
\end{align}
which leads to
\begin{align}\label{3Pfmom}
|\tilde{X}(S+k\Delta)| & \leq \left|\tilde{X}(S+k\Delta)-\tilde{X}(S+(k-1)\Delta)\right|+\left|\tilde{X}(S+(k-1)\Delta)\right|\cr
&\leq \Pi_{k-1}+\left|\tilde{X}(S+(k-1)\Delta)\right|\cr
&\leq 1\wedge (T-S)^{\be } +(1+\varsigma)\left|\tilde{X}(S+(k-1)\Delta)\right|.
\end{align}
Using \eqref{3Pfmom} recursively, one can see that
\begin{align*}
|\tilde{X}(S+k\Delta)|&\leq (1+\varsigma)^k|\tilde{X}(S)|+\left(\sum_{i=0}^{k-1}(1 +\varsigma)^i\right)(1\wedge (T-S)^{\be })\\
&=(1+ \varsigma)^k|\tilde{X}(S)|+\ff {(1+ \varsigma)^{k}-1} {\varsigma}(1\wedge (T-S)^{\be }).
\end{align*}
Then, substituting this into \eqref{Fan-add1} implies
\begin{align*}
\Pi_k\leq (1+ \varsigma)^{k}(1\wedge (T-S)^{\be }) +\varsigma(1+ \varsigma)^k|\tilde{X}(S)|
= (1+ \varsigma)^{k}\left(1\wedge (T-S)^{\be }+\varsigma|\tilde{X}(S)|\right).
\end{align*}
Consequently, it follows that
\begin{align}\label{Fan-add2}
\|\tilde{X}\|_{S,T,\infty}& \leq |\tilde{X}(S)|+\sum_{i=0}^{\left[\ff {T-S} {\Delta}\right] }\|\tilde{X}\|_{S+i\Delta,(S+(i+1)\Delta)\wedge T,\be}\Delta^\be\cr
&=|\tilde{X}(S)|+\sum_{i=0}^{\left[\ff T {\Delta}\right] }\Pi_{i}\cr
&\leq |\tilde{X}(S)|+\ff {(1+ \varsigma)^{\left[\ff {T-S} {\Delta}\right]+1}-1} {\varsigma}\left(1\wedge (T-S)^{\be } +\varsigma|\tilde{X}(S)|\right)\cr
&=(1+\varsigma)^{1+[ \ff {T-S} {\Delta} ]} |\tilde{X}(0)|+\ff {(1+ \varsigma)^{\left[\ff {T-S} {\Delta}\right]+1}-1} {\varsigma}\left(1\wedge (T-S)^{\be}\right) \cr
& \leq(1+\varsigma)^{1+[ \ff {T-S} {\Delta} ]} |\tilde{X}(S)|+ (1+\varsigma)^{ [\ff {T-S} {\Delta} ]}  \left(\left[\ff {T-S} {\Delta}\right]+1\right)\left(1\wedge (T-S)^{\be }\right),
\end{align}
where the last inequality is due to the following relation
\beg{align*}
\ff {(1+\varsigma)^x-1} {\varsigma} =\ff x {\varsigma}\int_0^{\varsigma} (1+r)^{x-1}\d r \leq x (1+\varsigma)^{x-1}, \ \ x>1.
\end{align*}
Noting that there hold $\varsigma=2C\Delta\leq 1\wedge (T-S)^{\be}$ and $\varsigma^{-1}\log(1+\varsigma)\leq 1$, we find that
\begin{align}\label{Fan-add3}
(1+ \varsigma)^{ [\ff {T-S} {\Delta} ]+1}&\leq\exp\left\{\ff {T-S+\Delta} {\Delta} \log(1+ \varsigma)\right\}\cr
&= \exp\left\{ (2C(T-S)+2C\Delta) \ff {\log(1+ \varsigma)} {\varsigma}\right\}\cr
&\leq \exp\left\{ 2C(T-S)+1\wedge (T-S)^{\be} \right\}
\end{align}
and
\begin{align}\label{4Pfmom}
(1+ \varsigma)^{ [\ff {T-S} {\Delta} ]}\leq \e^{2C(T-S)}.
\end{align}
Additionally, by \eqref{ad-De} we get
\begin{align}\label{Fan-add4}
&\left[\ff {T-S} {\Delta}\right] \leq\ff {T-S} {\Delta}\nonumber\\
&\leq  C\Bigg\{\left((1\vee (T-S))\|B^H\|_{S,T,\be}^{\ff 1 {\be }}\right) \vee \left[\left(  { (1+K_{\tilde{\si}})\|B^H\|_{S,T,\be}}  \right)^{\ff 1 {\ga_0+\be}}\left((T-S)\vee (T-S)^{\ff {\ga_0} {\ga_0+\be}}\right)\right]\cr
&\qquad\quad\vee \left[ (1+K_{\tilde{b}})\left( (T-S)\vee (T-S)^{1-\be}\right)\right]\Bigg\}\cr
&=:G(T-S,K_{\tilde{b}},K_{\tilde{\si}}, B^H).
\end{align}
Then, plugging \eqref{Fan-add3}-\eqref{Fan-add4} into \eqref{Fan-add2}, we obtain
\begin{align}\label{ad-supn}
\|\tilde{X}\|_{S,T,\infty}&\leq \exp\left\{ 2C(T-S)+1\wedge (T-S)^{\be }\right\}  |\tilde{X}(S)|\nonumber \\
&\quad\,+ (1\wedge (T-S)^{\be })e^{2C(T-S)} \left(1+G(T-S,K_{\tilde{b}},K_{\tilde{\sigma}},B^H)\right).
\end{align}
Since $\E\|B^H\|_{S,T,\be}^{q}\leq C_{q,H,\be}(T-S)^{q(H-\be)}$ for any $q\geq1$ (see, e.g., \cite[Lemma 8]{Saussereau12}), we derive that for each $p\geq1$,
\begin{align}\label{ad-EH1}
&\left[\E (G(T-S,K_{\tilde{b}},K_{\tilde{\sigma}},B^H)^p)\right]^{\ff 1 p}\cr
\leq&  C\bigg((T-S)^{\ff H {\be }}\vee (T-S)^{\ff {H-\be } {\be }}+(T-S)^{\ff {H+\ga_0} {\ga_0+\be}}\vee (T-S)^{\ff {H+\ga_0-\be} {\ga_0+\be}}(1+K_{\tilde{\sigma}})^{\ff 1 {\ga_0+\be}}\cr
&\quad +((T-S)\vee (T-S)^{1-\be})(1+K_{\tilde{b}})\bigg)\cr
 \leq& C\left((T-S)^{\ff H {\be}}\vee (T-S)^{\ff {H-\be} {\be}\wedge(1-\be)}\right)\left(1+ (1+K_{\tilde{\sigma}})^{\ff 1 {\ga_0+\be}}+K_{\tilde{b}}\right)\cr
 \leq& C\left((T-S)^{\ff H {\be}}\vee (T-S)^{\ff {H-\be} {\be}\wedge(1-\be)}\right)(1+K_{\tilde{b}}+ K_{\tilde{\sigma}}),
\end{align}
Here, we have used the fact that
\[\ff H {\be}>\ff {H+\ga_0} {\ga_0+\be}>1,\quad \ff {H+\ga_0-\be} {\ga_0+\be}>\ff {H-\be} {\be},\]
in the second inequality, and invoked the $C_r$-inequality and the fact that $\ga_0+\be>1$ in the last inequality.\\
This, along with \eqref{ad-supn}, implies that for any $p\geq 1$,
\begin{align}\label{5Pfmom}
&\left[\E\left(\|\tilde{X}\|^p_{S,T,\infty}\right)\right]^\ff 1 p\cr
\leq& \exp\left\{ 2C(T-S)+1\wedge (T-S)^{\be }\right\} \left(\E|\tilde{X}(S)|^p\right)^{\ff 1 p}  \cr
&+ C(1\wedge (T-S)^{\be })e^{2C(T-S)} \left(1+\left((T-S)^{\ff H {\be}}\vee (T-S)^{\ff {H-\be} {\be}\wedge(1-\be)}\right)(1+K_{\tilde{b}}+ K_{\tilde{\sigma}})\right).
\end{align}

Next, we shall provide an estimate for the term $[\E(\|\tilde{X}\|_{S,T,\be}^p)]^{\ff 1 p}$.
Applying Lemma \ref{TeLe2} with $n=1+\left[\ff {T-S} {\Delta} \right]$ and using \eqref{Fan-add1}, \eqref{Fan-add4}, \eqref{2Pfmom}, \eqref{ad-supn}, we first have
\begin{align}\label{ad-1be}
\|\tilde{X}\|_{S,T,\be}&\leq \left({ 1+\left[\ff {T-S} {\Delta} \right]}\right)^{1-\be}\max_{0\leq k\leq \left[ \ff {T-S} {\Delta} \right]} \|\tilde{X}\|_{S+k\Delta,(S+(k+1)\Delta)\wedge T,\be}\nonumber\\
&\leq  \left(1+\ff {T-S} {\Delta} \right)^{1-\be} \left(\Delta^{-\be}( 1\wedge (T-S)^{\be})+2C\Delta^{1-\be}\max_{0\leq k\leq \left[ \ff {T-S} {\Delta} \right] }|\tilde{X}(S+k\Delta)|\right)\cr
&\leq\left(1+\ff {T-S} {\Delta} \right)^{1-\be}\left(\ff {T-S} \Delta\right)^\be
+C\left(\Delta+ T-S  \right)^{1-\be} \|\tilde{X}\|_{S,T,\infty}\cr
&\leq \left(1+G(T-S,K_{\tilde{b}},K_{\tilde{\si}},B^H) \right)^{1-\be}G(T-S,K_{\tilde{b}},K_{\tilde{\si}},B^H)^{\be}\cr
&\quad\,+C\left(\Delta+ T-S  \right)^{1-\be} \|\tilde{X}\|_{S,T,\infty}\cr
&\leq \left(1+G(T-S,K_{\tilde{b}},K_{\tilde{\si}},B^H) \right)^{1-\be}G(T-S,K_{\tilde{b}},K_{\tilde{\si}},B^H)^{\be}\cr
&\quad\, +C\left((T-S)^{\be}\vee (T-S)  \right)^{1-\be} \|\tilde{X}\|_{S,T,\infty}.
\end{align}
Then by the Minkowski inequality, the H\"{o}lder inequality, \eqref{ad-EH1}  and  \eqref{5Pfmom}, we have for any $p\geq 1$,
\begin{align*}
&\left[\E\left(\|\tilde{X}\|_{S,T,\be}^p\right)\right]^{\ff 1 p}\\
\leq&
\left[\E\left(1+G(T-S,K_{\tilde{b}},K_{\tilde{\si}},B^H)\right)^p\right]^{\ff {1-\be}p}
\left[\E G(T-S,K_{\tilde{b}},K_{\tilde{\si}},B^H)^p\right]^{\ff {\be}p}\cr
&+C\left((T-S)^{\be}\vee (T-S)  \right)^{1-\be}\left[\E\left(\|\tilde{X}\|^p_{S,T,\infty}\right)\right]^\ff 1 p\cr
\leq&C\left((T-S)^{\ff H {\be}}\vee (T-S)^{\ff {H-\be} {\be}\wedge(1-\be)}\right)^\be\\
&\quad\,\times\left(1+\left((T-S)^{\ff H {\be}}\vee (T-S)^{\ff {H-\be} {\be}\wedge(1-\be)}\right)^{1-\be}\right)(1+K_{\tilde{b}}+ K_{\tilde{\sigma}})\cr
&+C\left((T-S)^{\be}\vee (T-S)  \right)^{1-\be}\bigg[\exp\left\{ 2C(T-S)+1\wedge (T-S)^{\be }\right\} \left(\E|\tilde{X}(S)|^p\right)^{\ff 1p} \cr
&\quad+(1\wedge (T-S)^{\be })e^{2C(T-S)} \left(1+\left((T-S)^{\ff H {\be}}\vee (T-S)^{\ff {H-\be} {\be}\wedge(1-\be)}\right)(1+K_{\tilde{b}}+ K_{\tilde{\sigma}})\right)
\bigg]\cr
\leq&\tilde{C}(T-S,H,\be,\ga_0,p)\left(1+K_{\tilde{b}} + K_{\tilde{\si}}  +\left(\E|\tilde{X}(S)|^p\right)^{\ff 1 p}\right),
 \end{align*}
where $\tilde C(T-S,H,\be,\ga_0,p)$ is a positive constant satisfying
\begin{align*}
\tilde{C}(T-S,H,\be,\ga_0,p)=O\left((T-S)^{(H-\be)\wedge ((1-\be)\be)}\right),\  \ \text{when}\ \ T-S\ra 0^+.
\end{align*}
Gathering this and \eqref{5Pfmom}, our proof is now complete.
\end{proof}

\textbf{Acknowledgement}

X. Fan was partially supported by the National Natural Science Foundation of China Grant No. 12371145 and the Natural Science Foundation of Anhui Province Grant No. 2008085MA10.
S.-Q. Zhang was supported in part by the National Natural Science Foundation of China Grant No. 11901604, 12371153.

\end{document}